\documentclass[a4paper,reqno,12pt]{amsart}
\usepackage[top=30mm,right=30mm,bottom=30mm,left=30mm]{geometry}

\usepackage{tabularx}
\usepackage{graphicx}
\usepackage{amsmath}
\usepackage{amssymb}
\usepackage{amsfonts}
\usepackage{amsthm}
\usepackage{amstext}
\usepackage{amsbsy}
\usepackage{amsopn}
\usepackage{amscd} 
\usepackage{enumerate}
\usepackage{color}
\usepackage[colorlinks]{hyperref}
\usepackage[hyperpageref]{backref}
\usepackage{multirow}
\usepackage{longtable}
\usepackage{float}
\usepackage{lscape}
\usepackage{pdflscape}
\usepackage{url}

\newtheorem{theorem}{Theorem}[section]
\newtheorem{corollary}[theorem]{Corollary}
\newtheorem{lemma}[theorem]{Lemma}
\newtheorem{proposition}[theorem]{Proposition}

\theoremstyle{definition}

\newtheorem{remark}[theorem]{Remark}

\numberwithin{equation}{section}

\renewcommand{\leq}{\leqslant}
\renewcommand{\geq}{\geqslant}

\begin{document}
\title[Flag-transitive symmetric $2$-designs]{The symmetric $2$-$(v,k,\lambda )$ designs, with $k>\lambda \left(\lambda-3 \right)/2$, admitting a flag-transitive, point imprimitive automorphism group are known}

\author[]{ Alessandro Montinaro}

%
%

\address{Alessandro Montinaro, Dipartimento di Matematica e Fisica “E. De Giorgi”, University of Salento, Lecce, Italy}
\email{alessandro.montinaro@unisalento.it}

\subjclass[MSC 2020:]{05B05; 05B25; 20B25}%
\keywords{ Symmetric design; automorphism group; flag-transitive design}
\date{\today}%

\begin{abstract}
The symmetric $2$-$(v,k,\lambda )$ designs, with $k>\lambda \left(\lambda-3 \right)/2$, admitting a flag-transitive, point-imprimitive automorphism group are completely classified: they are the known $2$-designs with parameters $(16,6,2),(45,12,3),(15,8,4)$ or $(96,20,4)$.  
\end{abstract}

\maketitle

\section{Introduction and Main Result}

A $2$-$(v,k,\lambda )$ \emph{design} $\mathcal{D}$ is a pair $(\mathcal{P},%
\mathcal{B})$ with a set $\mathcal{P}$ of $v$ points and a set $\mathcal{B}$
of blocks such that each block is a $k$-subset of $\mathcal{P}$ and each two
distinct points are contained in $\lambda $ blocks. We say $\mathcal{D}$ is 
\emph{non-trivial} if $2<k<v$, and symmetric if $v=b$. All $2$-$(v,k,\lambda
)$ designs in this paper are assumed to be non-trivial. An automorphism of $%
\mathcal{D}$ is a permutation of the point set which preserves the block
set. The set of all automorphisms of $\mathcal{D}$ with the composition of
permutations forms a group, denoted by $\mathrm{Aut(\mathcal{D})}$. For a
subgroup $G$ of $\mathrm{Aut(\mathcal{D})}$, $G$ is said to be \emph{%
point-primitive} if $G$ acts primitively on $\mathcal{P}$, and said to be 
\emph{point-imprimitive} otherwise. In this setting we also say that $%
\mathcal{D}$ is either \emph{point-primitive} or \emph{point-imprimitive}
respectively. A \emph{flag} of $\mathcal{D}$ is a pair $(x,B)$ where $x$ is
a point and $B$ is a block containing $x$. If $G\leq \mathrm{Aut(\mathcal{D})%
}$ acts transitively on the set of flags of $\mathcal{D}$, then we say that $%
G$ is \emph{flag-transitive} and that $\mathcal{D}$ is a \emph{%
flag-transitive design}.\ 

In 1987, Davies \cite{Da} proved that in a flag-transitive and
point-imprimitive $2$-$(v,k,\lambda )$ design, the block size is bounded
for a given value of the parameter $\lambda $, where $\lambda \geq 2$ by a result of Higman-McLaughlin \cite{HM} dating back to 1961. In 2005, O'Reilly
Regueiro \cite{ORR} obtained an explicit upper bound. Later that year,
Praeger and Zhou \cite{PZ} improved that upper bound and gave a complete
list of feasible parameters. In 2020, Mandi\'{c} and \v{S}ubasi\'{c} \cite{MS} classified the
flag-transitive point-imprimitive symmetric $2$-designs with $\lambda \leq
10 $ except for two possible numerical cases. Recently, Montinaro \cite{Mo2} has classified those with $k>\lambda
\left(\lambda -3 \right)/2$ and such that a block of the $2$-design
intersects a block of imprimitivity in at least $3$ points. In this paper we
complete the work started in \cite{Mo2} by classifying $\mathcal{D}$ with $%
k>\lambda \left(\lambda -3 \right)/2$ regardless the intersection size of a block of $\mathcal{D}$ with a block of
imprimitivity. More precisely, our result is the following.

\bigskip

\begin{theorem}
\label{main} Let $\mathcal{D}=\left(\mathcal{P},\mathcal{B} \right)$ be a
symmetric $2$-$(v,k,\lambda )$ design admitting a flag-transitive,
point-imprimitive automorphism group. If $k> \lambda (\lambda-3)/2$ then
one of the following holds:

\begin{enumerate}
\item $\mathcal{D}$ is isomorphic to one of the two $2$-$(16,6,2)$ designs.

\item $\mathcal{D}$ is isomorphic to the $2$-$(45,12,3)$ design.

\item $\mathcal{D}$ is isomorphic to the $2$-$(15,8,4)$ design.

\item $\mathcal{D}$ is isomorphic to one of the four $2$-$(96,20,4)$ designs.
\end{enumerate}
\end{theorem}

\bigskip

In 1945, Hussain \cite{Hu} and in 1946 Nandi \cite{Na} independently proved the existence of three symmetric $2$-$(16,6,2)$ designs. In 2005, O’Reilly Regueiro \cite{ORR} proved that exactly two of them are flag-transitive and point-imprimitive. In the same paper O’Reilly Regueiro constructed a $2$-$(15,8,4)$ design. Such $2$-design was proved to be unique by Praeger and Zhou \cite{PZ} in 2006. One year later, Praeger  \cite{P2} constructed and proved there is exactly one flag-transitive and point-imprimitive $2$-$(45,12,3)$ design. Finally, in 2009, Law, Praeger and Reichard \cite{LPR} proved there are four flag-transitive and point-imprimitive $2$-$(96,20,4)$ designs.

\bigskip

The outline of the proof is as follows. We start by strengthening the classification result obtained in \cite{MS}. Indeed, in \cite{MS} it is proven that, if $\lambda \leq 10$ then $\mathcal{D}$ is known except for two numerical values for the parameters for $\mathcal{D}$. The two exceptions are ruled out here in Theorem \ref{MMS}. Subsequently, we focus on the case $\lambda>10$. In Proposition \ref{DivanDan} it is
shown that $G^{\Sigma }$ acts primitively on $\Sigma $ by using the results
contained in \cite{Mo2}. Then, it is proven in Theorem \ref{AQP} that, either 
$G$ acts point-quasiprimitively on $\mathcal{D}$ or $G(\Sigma )\neq 1$, $%
G^{\Sigma }$ is almost simple and $\mathcal{D}$ has parameters $%
(2^{d+2}(2^{d-1}-1)^{2}, 2\left( 2^{d}-1\right) \left(
2^{d-1}-1\right) ,2(2^{d-1}-1))$ where $d\geq 4$. Afterwards, by combining
the O'Nan-Scott theorem for quasiprimitive groups achieved in \cite{P1}
with an adaptation of the techniques developed by \cite{ZZ}, in Theorem \ref{ASPQP} we
show that $G^{\Sigma }$ is almost simple also in the quasiprimitive case.
Moreover, if $L$ is the preimage in $G$ of $\mathrm{Soc}(G^{\Sigma })$ and $%
\Delta \in \Sigma $, in Proposition \ref{unapred}, Corollary \ref{QuotL} and Theorem \ref{Large}
it is proven that, either $G(\Sigma )=1$, $L_{\Delta
}$ is contained in a semilinear $1$-dimensional group and $\left\vert L \right \vert \leq 2 \left \vert L_{\Delta
}^{\Delta } \right \vert ^{2}\left\vert \mathrm{Out}(L)\right \vert^{2}$, or $G(\Sigma )=1$, $L_{\Delta
}$ is a non-solvable $2$-transitive permutation group of degree $\left\vert \Delta
\right\vert $ and $\left\vert L\right\vert \leq \left\vert L_{\Delta }\right\vert ^{2}$, or $G(\Sigma )\neq 1$ and a quotient of $L_{\Delta }^{\Sigma }
$ is isomorphic either to $SL_{d}(2)$, or to $A_{7}$ for $d=4$. In particular, in each case $L_{\Delta }^{\Sigma }$ is a large subgroup of $L^{\Sigma }$. Finally, we  use all the above mentioned constraints on $ L^{\Sigma }$ and on $ L_{\Delta }^{\Sigma
}$ together with the results contained in \cite{AB} and \cite{LS} to precisely determine the admissible pairs $( L^{\Sigma },L_{\Delta}^{\Sigma })$ and from these to prove that there are no examples of $\mathcal{D}$ for $\lambda >10$. At
this point, our classification result follows from Theorem \ref{MMS}.

\bigskip 

\section{Preliminary reductions}

\medskip It is well know that, if $\mathcal{D}$ is a symmetric $2$-$%
(v,k,\lambda )$ design, then $r=k$, $b=v$ and $k(k-1)=(v-1)\lambda $ (for instance, see \cite{Demb}). Moreover, the following fact holds:

\begin{lemma}
\label{PP}If $\mathcal{D}$ admits a flag-transitive automorphism group $G$
and $x$ is any point of $\mathcal{D}$, then $\left\vert y^{G_{x}}\right\vert
\lambda =k\left\vert B\cap y^{G_{x}}\right\vert $ for any point $y$ of $%
\mathcal{D}$, with $y\neq x$, and for any block $B$ of $\mathcal{D}$
incident with $x$.
\end{lemma}

\begin{proof}
Let $x$, $y$ be points of $\mathcal{D}$, $y\neq x$, and $B$ be any block of $%
\mathcal{D}$ incident with $x$. Since $(y^{G_{x}},B^{G_{x}})$ is a tactical
configuration by \cite[1.2.6]{Demb}, it follows that $\left\vert
y^{G_{x}}\right\vert \lambda =k\left\vert B\cap y^{G_{x}}\right\vert $.
\end{proof}

\bigskip

The following theorem, which is a summary of \cite{Mo2} and some of the results contained
in \cite{PZ}, is our starting point.

\bigskip

\begin{theorem}
\label{PZM}Let $\mathcal{D}=(\mathcal{P} ,\mathcal{B})$ be a symmetric $2$%
-design admitting a flag-transitive, point-imprimitive automorphism group $G$
that leaves invariant a non-trivial partition $\Sigma =\left\lbrace
\Delta_{1},...,\Delta_{d} \right\rbrace$ of $\mathcal{P} $ such that $%
\left\vert \Delta_{i} \right \vert =c$ for each $i=1,...,d$. Then the
following hold:

\begin{enumerate}
\item[I.] There is a constant $\ell$ such that, for each $B\in \mathcal{B}$
and $\Delta _{i}\in \Sigma $, the size $\left\vert B\cap \Delta
_{i}\right\vert $ is either $0$ or $\ell$.

\item[II.] There is a constant $\theta$ such that, for each $B\in \mathcal{B}
$ and $\Delta _{i}\in \Sigma $ with $\left\vert B\cap \Delta _{i}\right\vert
>0$, the number of blocks of $\mathcal{D}$ whose intersection set with $%
\Delta _{i}$ coincides with $B\cap \Delta _{i}$ is $\theta$.

\item[III.] If $\ell=2$ then $G_{\Delta _{i}}^{\Delta _{i}}$ acts $2$%
-transitively on $\Delta_{i}$ for each $i=1,...,d$.

\item[IV.] If $\ell\geq 3$ then $\mathcal{D}_{i}=\left( \Delta_{i}, \left(B
\cap \Delta_{i} \right)^{G_{\Delta _{i}}^{\Delta _{i}}}\right)$ is a
flag-transitive non-trivial $2$-$\left(c,\ell,\lambda/\theta\right)$ design
for each $i=1,...,d$.
\end{enumerate}

Moreover, if $k>\lambda (\lambda -3)/2$ then one of the following holds:

\begin{enumerate}
\item[V.] $\ell=2$ and one of the following holds:

\begin{enumerate}
\item[1.] $\mathcal{D}$ is a symmetric $2$-$(\lambda ^{2}(\lambda
+2),\lambda (\lambda +1),\lambda )$ design and $\left(c,d
\right)=\left(\lambda+2,\lambda^{2}\right)$.

\item[2.] $\mathcal{D}$ is a symmetric $2$-$\left( \left(\frac{\lambda+2}{2}%
\right) \left(\frac{\lambda^2-2\lambda+2}{2}\right),\frac{\lambda^2}{2}%
,\lambda \right)$ design, $\left(c,d \right)=\left(\frac{\lambda+2}{2},\frac{%
\lambda^2-2\lambda+2}{2}\right)$, and either $\lambda \equiv 0 \pmod{4}$, or 
$\lambda=2w^{2}$, where $w$ is odd, $w \geq 3$, and $2(w^2-1)$ is a square.
\end{enumerate}

\item[VI.] $\ell\geq 3$ and one of the following holds:

\begin{enumerate}
\item[1.] $\mathcal{D}$ is isomorphic to the $2$-$(45,12,3)$ design of \cite[%
Construction 4.2]{P2}.

\item[2.] $\mathcal{D}$ is isomorphic to one of the four $2$-$(96,20,4)$
designs constructed in \cite{LPR}.
\end{enumerate}
\end{enumerate}
\end{theorem}

\bigskip

We are going to focus on case (V) of the previous theorem. \emph{Throughout the paper, a $2$%
-design as in case (V.1) or (V.2) of Theorem \ref{PZM} will be simply
called a $2$-design of type 1 or 2 respectively.}

\bigskip
Let $\Delta \in \Sigma $ and $x \in \Delta $. Since $G(\Sigma) \trianglelefteq G_{\Delta}$ and $G(\Delta) \trianglelefteq G_{x}$, it is immediate to verify that $(G^{\Sigma})_{\Delta}=(G_{\Delta})^{\Sigma}$ and that $\left( G_{\Delta}^{\Delta} \right)_{x}=(G_{x})^{\Delta}$. Hence, in the sequel $(G^{\Sigma})_{\Delta}$ and $\left( G_{\Delta}^{\Delta} \right)_{x}$ will simply be denoted by $G^{\Sigma}_{\Delta}$ and $G_{x}^{\Delta}$ respectively. Moreover, the following holds:
\begin{equation}\label{salvaMAOL}
\frac{G^{\Sigma}_{\Delta}}{G(\Delta)^{\Sigma}} \cong \frac{G_{\Delta}}{G(\Delta)G(\Sigma)} \cong \frac{G^{\Delta}_{\Delta}}{G(\Sigma)^{\Delta}}.
\end{equation}
\bigskip

A further reduction is the following theorem.

\bigskip

\begin{theorem}
\label{MMS}Let $\mathcal{D}=\left( \mathcal{P},\mathcal{B}\right) $ be a
symmetric $2$-$(v,k,\lambda )$ design admitting a flag-transitive,
point-imprimitive automorphism group $G$. If $\lambda \leq 10$ then

\begin{enumerate}
\item $\mathcal{D}$ is isomorphic to one of the two $2$-$(16,6,2)$ designs

\item $\mathcal{D}$ is isomorphic to the $2$-$(45,12,3)$ design.

\item $\mathcal{D}$ is isomorphic to the $2$-$(15,8,4)$ design.

\item $\mathcal{D}$ is isomorphic to one of the four $2$-$(96,20,4)$ designs.
\end{enumerate}
\end{theorem}

\begin{proof}
This results is proven in \cite[Theorem 1]{MS} with the following possible
exceptions of $(v,k,\lambda ,c,d)=(288,42,6,8,36)$ or $(891,90,9,81,11)$.
Note that $k>\lambda (\lambda -3)/2$ in both exceptional cases. Actually,
the latter does not correspond to any case of Theorem \ref{PZM}(V--VI),
and hence it cannot occur. The former corresponds to Theorem \ref{PZM}%
(V.1) for $\lambda =6$. Also, if $\Delta \in \Sigma$ then $G_{\Delta }^{\Delta }\cong AGL_{1}(8)$, $ A\Gamma L_{1}(8)$,  $A_{8}$, $S_{8}$, 
$PSL_{2}(7)$ or $PGL_{2}(7)$ by \cite[Lists (A) and (B)]{Ka} since $G_{\Delta }^{\Delta }$ acts $2$-transitively on $\Delta$.

Assume that $u$ divides the order of $G(\Sigma )$, where $u$ is an odd prime, and let $\psi $ be a $u$-element of $G(\Sigma )$. Then $\psi $
fixes at least a point on each $\Delta \in \Sigma $ since $\left\vert \Delta \right\vert =8$. Thus $\psi $ preserves
at least two distinct blocks of $\mathcal{D}$ by \cite[Theorem 3.1]{La}, say $B_{1}$ and $B_{2}$. Actually, $\psi $ fixes $B_{1}$ and $B_{2}$ pointwise since any of these intersect each element of $\Sigma$ in $0$ or $2$ points and $%
\psi \in G(\Sigma )$. Hence, $\psi $ fixes at least $2\cdot 42-6$ points of $\mathcal{D}$, but this contradicts \cite[Corollary 3.7]{La}. Thus, the
order of $\left\vert G(\Sigma )\right\vert =2^{i}$ with $i \geq 0$.

Assume that $w$ divides the order of $G(\Delta )$, where $w$ is an odd prime, $w \geq 7$, and let $\phi $ be a $w$-element of $G(\Delta )$. Then $\phi 
$ fixes the $6$ blocks incident with any pair of distinct
points of $\Delta $. Therefore, $\phi $ fixes at least $6\cdot 
\binom{8}{2}$ points of $\mathcal{D}$ by \cite[Theorem 3.1]{La}, and we again reach a contradiction by \cite[Corollary 3.7]{La}%
. Thus, $G(\Delta )$ is a $\{2,3,5\}$-group.

Any Sylow $7$-subgroup of $G$ is of order $7$, since $\left\vert \Sigma \right\vert =36$, $G_{\Delta }^{\Delta }$ is one of the groups listed above and the order of $G(\Delta )$ is coprime to $7$. Then, by \cite[Table B.4]{DM} one of the following holds:
\begin{enumerate}
\item $A_{9} \trianglelefteq G^{\Sigma} \leq S_{9}$ and $S_{7}\trianglelefteq G^{\Sigma}_{\Delta} \leq S_{7} \times Z_{2}$;
\item $PSL_{2}(8) \trianglelefteq G^{\Sigma} \leq P \Gamma L_{2}(8)$ and $ D_{14} \trianglelefteq G^{\Sigma}_{\Delta}\leq F_{42}$;
\item $PSU_{3}(3)\trianglelefteq G^{\Sigma} \leq P \Gamma U_{3}(3) $ and $PSL_{2}(7) \trianglelefteq G^{\Sigma}_{\Delta}\leq PGL_{2}(7)$;
\item $ G^{\Sigma} \cong  Sp_{6}(2)$ and $G^{\Sigma}_{\Delta} \cong S_{8}$.
\end{enumerate}
Since $G_{\Delta }^{\Delta }$ is one of the $2$-transitive-groups listed above, (1) is immediately ruled out by (\ref{salvaMAOL}). In (2) the group $G_{\Delta}$ is solvable since $G^{\Sigma}_{\Delta}$ is solvable and $G(\Sigma )$ is a $2$-group. Thus $G_{\Delta }^{\Delta }$ is solvable and hence it is isomorphic to $AGL_{1}(8)$ or $A\Gamma L_{1}(8)$. Moreover, a quotient group of $G_{\Delta }^{\Delta }$ must contain a subgroup isomorphic to $D_{14}$ by (\ref{salvaMAOL}) since $G(\Delta )$ is a $\{2,3,5\}$-group, but this is clearly impossible. It follows that only (3) and (4) are admissible. Also, $G(\Sigma)=G(\Delta)=1$ in both cases by comparing (\ref{salvaMAOL}) with the possibilities for $G_{\Delta}^{\Delta}$ provided above.

Assume that (4) occurs. Then $G$ acts $2$-transitively on $\Sigma$. Let $x \in \Delta$ and $\Delta^{\prime} \in \Sigma \setminus \{ \Delta \}$. Then $G_{x} \cong S_{7}$ and $G_{x,\Delta^{\prime}} \leq G_{\Delta,\Delta^{\prime}}   \cong (S_{4} \times S_{4}):Z_{2}$ by \cite{At}, hence $G_{x}$ acts transitively on $\Sigma \setminus \{ \Delta \}$. On the other hand, $G_{x}=G_{B}$ for some block $B$ of $\mathcal{D}$ since $G$ has a unique conjugate class of subgroups isomorphic to $S_{7}$ by \cite{At}. Therefore, $G_{x}$ permutes transitively the $21$ elements of $\Sigma$ intersecting $B$, whereas $G_{x}$ acts transitively on the $35$ elements $\Sigma \setminus \{ \Delta \}$. So this case is excluded. Finally, (3) is ruled out with the aid of \textsf{GAP} \cite{GAP}, and the proof is thus completed.
\end{proof}

\bigskip

On the basis of Theorem \ref{MMS}, in the sequel we may assume that $\lambda
>10$.

\bigskip

\begin{lemma}
\label{Mpomoc}If $G$ preserves a further partition $\Sigma ^{\prime }$ of
the point set of $\mathcal{D}$ in $d^{\prime}$ blocks of imprimitivity of size $c^{\prime}$. Then $d^{\prime}=d$ and $c^{\prime}=c$, and one of the following holds:

\begin{enumerate}
\item $\Sigma =\Sigma ^{\prime }$;

\item $\left\vert \Delta \cap \Delta ^{\prime }\right\vert \leq 1$ for for each $\Delta \in \Sigma $, $\Delta ^{\prime }\in
\Sigma ^{\prime }$.
\end{enumerate}
\end{lemma}

\begin{proof}
Suppose there is a $G$-invariant partition $\Sigma ^{\prime }$ of the point
set of $\mathcal{D}$. Let $B$ any block of $\mathcal{D}$ and let $\Delta
^{\prime }\in \Sigma ^{\prime }$. If $\left\vert B\cap \Delta ^{\prime }\right\vert \geq 3$ then
either $(v,k,\lambda )=(45,12,3)$ or $(96,20,4)$ by \cite[Theorem 1.1]{Mo2},
whereas $\lambda >10$ by our assumptions. Thus, $\left\vert B\cap \Delta ^{\prime }\right\vert =2$.

If $\mathcal{%
D}$ is of different type with respect to $\Sigma 
$ and to $\Sigma ^{\prime }$, then $%
k=\lambda +2=\lambda ^{2}/2$, which has no integer solutions. Therefore $%
\mathcal{D}$ is of the same type with respect to $\Sigma $ and to $\Sigma ^{\prime }$, and hence $d^{\prime}=d$ and $c^{\prime}=c$. 

Let $\Delta
\in \Sigma $ such that $\Delta \cap \Delta ^{\prime }\neq \varnothing $. If $%
\left\vert \Delta \cap \Delta ^{\prime }\right\vert >1$ then $\Delta =\Delta
^{\prime }$ since $G$ induces a $2$-transitive group on $\Delta $ and $%
\left\vert \Delta \right\vert =\left\vert \Delta ^{\prime }\right\vert $.
Therefore $\Sigma =\Sigma ^{\prime }$ which is (2).
\end{proof}

\bigskip

\begin{proposition}
\label{DivanDan}$G^{\Sigma }$ acts primitively on $\Sigma $. Moreover, and
one of the following holds:

\begin{enumerate}
\item $G(\Sigma )\neq 1$ and $\mathrm{Soc}(G_{\Delta }^{\Delta })\trianglelefteq
G(\Sigma )^{\Delta }$;

\item $G(\Sigma )=1$ and $G$ acts point-quasiprimitively on $\mathcal{D}$.
\end{enumerate}
\end{proposition}

\begin{proof}
It follows from Lemma \ref{Mpomoc} and \cite[Theorem 1.5A]{DM}
that $G^{\Sigma }$ acts primitively on $\Sigma $.

Assume that $G(\Sigma )\neq 1$. If there is $\Delta ^{\prime }\in \Sigma $
such that $G(\Sigma )\leq G(\Delta ^{\prime })$. Then $G(\Sigma )\leq
G(\Delta ^{\prime \prime })$ for each $\Delta ^{\prime \prime }\in \Sigma $,
and hence $G(\Sigma )=1$, since $G(\Sigma )\vartriangleleft G$ and $G$ acts
transitively on $\Sigma $. However, it contradicts the assumption $G(\Sigma )\neq 1$.
Thus $G(\Sigma )\nleq G(\Delta )$ for each $\Delta \in \Sigma $. Then $1\neq
G(\Sigma )^{\Delta }\trianglelefteq G_{\Delta }^{\Delta }$, and hence $%
\mathrm{Soc}(G_{\Delta }^{\Delta })\trianglelefteq G(\Sigma )^{\Delta }$ by \cite[Theorem 4.3B]{DM} since $%
G_{\Delta }^{\Delta }$ is $2$-transitive on $\Delta $.

Assume that $G(\Sigma )=1$. Let $N$ be any normal subgroup of $G$. Then $N=G$
since $G(\Sigma )=1$ and since $G^{\Sigma }$ acts primitively on $\Sigma $.
Therefore, $G$ acts point-quasiprimitively on $\mathcal{D}$.
\end{proof}

\begin{lemma}
\label{Fixpoints}Let $\Delta \in \Sigma $ and $\gamma \in G(\Delta )$, $%
\gamma \neq 1$. Then one of the following holds:

\begin{enumerate}
\item $\mathcal{D}$ is of type 1 and $\left\vert \mathrm{Fix}(\gamma
)\right\vert \leq (\lambda +2)\lambda $;

\item $\mathcal{D}$ is of type 2 and $\left\vert \mathrm{Fix}(\gamma
)\right\vert \leq \lambda ^{2}/2+\sqrt{\lambda ^{2}/2-\lambda }$
\end{enumerate}
\end{lemma}

\begin{proof}
Let $\Delta \in \Sigma $ and $\gamma \in G(\Delta )$, $\gamma \neq 1$. It
follows from by \cite[Theorem 3.1]{La} that%
\begin{equation}
\left\vert \mathrm{Fix}(\gamma )\right\vert \leq \frac{\lambda }{k-\sqrt{%
k-\lambda }}\cdot \left\vert \Delta \right\vert \cdot \left\vert \Sigma
\right\vert  \label{Fixgamma}
\end{equation}%
where $\frac{\lambda }{k-\sqrt{k-\lambda }}\cdot \left\vert \Delta
\right\vert \cdot \left\vert \Sigma \right\vert $ is either $(\lambda
+1)\lambda $ or $\lambda ^{2}/2+\sqrt{\lambda ^{2}/2-\lambda }$ according to
whether $\mathcal{D}$ is of type 1 or 2 respectively.
\end{proof}

\bigskip

\begin{corollary}
\label{CFixT}The following hold:
\begin{enumerate}
\item If $\mathcal{D}$ is of type 1, each prime divisor of $%
\left\vert G(\Delta )\right\vert $ divides $\lambda (\lambda -1)$.
\item If $\mathcal{D}$ is of type 2, each prime divisor of $%
\left\vert G(\Delta )\right\vert $ divides $\lambda (\lambda -1)(\lambda-2)(\lambda-3)$.
\end{enumerate}
\end{corollary}

\begin{proof}
Let $\gamma $ be $w$-element of $G(\Delta )$, where $w$ is a prime such that 
$w\nmid \lambda $. Since $\gamma $ fixes $\Delta $ pointwise, $\gamma $
fixes at least $\mu $ of the $\lambda $ blocks of $\mathcal{D}$ incident
with any two distinct points of $\Delta $, where $\mu =\lambda \pmod{w}$.
Clearly, these fixed blocks are pairwise distinct. 

If $\mathcal{D}$ is of type 1 then $\gamma $ fixes $\mu \frac{%
(\lambda +2)(\lambda +1)}{2}$ blocks of $\mathcal{D}$, and hence $\left\vert \mathrm{Fix}(\gamma )\right\vert \geq \mu 
\frac{(\lambda +2)(\lambda +1)}{2}$. Therefore $\mu =1$ by Lemma \ref%
{Fixpoints}(1) and the assertion (1) follows.

If $\mathcal{D}$ is of type 2 then $\gamma $ fixes $\mu \frac{\lambda}{4}\left(\frac{\lambda}{2}+1\right)$ blocks of $\mathcal{D}$, and hence $\left\vert \mathrm{Fix}(\gamma )\right\vert \geq  
\mu \frac{\lambda}{4}\left(\frac{\lambda}{2}+1\right)$. Therefore $\mu \leq 3$ by Lemma \ref%
{Fixpoints}(2) and the assertion (2) follows.
\end{proof}

\bigskip

\begin{lemma}
\label{2S}Let $x$ be any point of $\mathcal{D}$. Then $G(\Sigma )_{x}$ lies in
a Sylow $2$-subgroup of $G(\Sigma)$.
\end{lemma}

\begin{proof}
Let $x$ be any point of $\mathcal{D}$ and let $\varphi $ be any $w$-element of $G(\Sigma )_{x}$%
, where $w$ is an odd prime. Then $\varphi $ fixes at least a blocks $B$ of $%
\mathcal{D}$ by \cite[Theorem 3.1]{La}. Since $B$ intersects each element of 
$\Sigma $ in $0$ or $2$ points, and since $\varphi $ is a $w$-element of $%
G(\Sigma )$, it follows that $\varphi $ fixes $B$ pointwise. Therefore, $%
\varphi $ fixes at least $k$ points of $\mathcal{D}$. Then $\varphi $ fixes
at least $k$ blocks of $\mathcal{D}$ again by \cite[Theorem 3.1]{La}. Let $%
B^{\prime }$ be further block fixed by $\varphi $. We may repeat the
previous argument with $B^{\prime }$ in the role of $B$ thus obtaining $%
\varphi $ fixing $B^{\prime }$ pointwise. Then $\varphi $ fixes at least $%
2k-\lambda $ points of $\mathcal{D}$, as $\left\vert B\cap B^{\prime
}\right\vert =\lambda $. Thus $\left\vert \mathrm{Fix}(\alpha )\right\vert
\geq 2k-\lambda $ and hence $\left\vert \mathrm{Fix}(\alpha )\right\vert $
is greater than or equal to $\lambda (2\lambda +1)$ or $\lambda ^{2}-\lambda 
$ according to whether $\mathcal{D}$ is of type 1 or 2 respectively.
However, this contradicts Lemma \ref{Fixpoints}, as $\lambda >2$.$%
\allowbreak $
\end{proof}

\bigskip

\begin{lemma}
\label{strike}If $G(\Sigma )\neq 1$ and $v$ is odd, then the following hold:

\begin{enumerate}
\item $G(\Sigma )= Soc(G_{\Delta })$ is an elementary abelian $p$-group, $p$ an odd prime, acting regularly on $\Delta $;

\item $G^{\Sigma }\leq GL_{d}(p)$, with $p^{d}=\left\vert \Delta \right\vert 
$ and $\Delta \in \Sigma $.
\end{enumerate}
\end{lemma}

\begin{proof}
Assume that $G(\Sigma )_{x}\neq 1$, where $x$ is any point of $\mathcal{D}$.
Then $G(\Sigma )_{x}$ is a Sylow $2$-subgroup of $G(\Sigma )$ by Lemma \ref%
{2S}. Denote $G(\Sigma )_{x}$ simply by $S$. Then $S$ fixes the same number $%
t$ of points in each $\Delta \in \Sigma $, where $t \geq 1$ since $v$ is odd. Therefore, $\left\vert \mathrm{Fix%
}(S)\right\vert =t\left\vert \Sigma \right\vert $. Now, if $\alpha $ is any
non-trivial element of $S$ then $\left\vert \mathrm{Fix}(\alpha )\right\vert
\geq t\left\vert \Sigma \right\vert $, and hence $t=1$ by Lemma \ref%
{Fixpoints}. Thus $S$ fixes a unique point on each $\Delta $.

Suppose that $\left\vert S\right\vert \geq 4$ and let $B$ be any block preserved by $S$
and $\Delta ^{\prime }$ be such that $\left\vert B\cap \Delta ^{\prime
}\right\vert =2$. Then there is a subgroup $S_{0}$ of $S$ of index at most $%
2 $ fixing $y$ and $y^{\prime }$ where $\left\{ y,y^{\prime }\right\} =B\cap
\Delta ^{\prime }$. Then $S_{0}\leq G(\Sigma )_{y}\cap G(\Sigma )_{y^{\prime
}}$, where $G(\Sigma )_{y}$ and $G(\Sigma )_{y^{\prime }}$ are two distinct
Sylow $2$-subgroups of $G(\Sigma )$, each of these fixing a unique point in $\Delta^{\prime} $. Suppose there is a point $z$ of $\mathcal{D}$ such that $%
z\in \mathrm{Fix}(G(\Sigma )_{y})\cap \mathrm{Fix}(G(\Sigma )_{y^{\prime }})$%
. The $\left\langle G(\Sigma )_{y},G(\Sigma )_{y^{\prime }}\right\rangle
\leq G(\Sigma )_{z}$ and hence $G(\Sigma )_{y}=G(\Sigma )_{y^{\prime
}}=G(\Sigma )_{z}$ since $G(\Sigma )_{z}$ is a Sylow $2$-subgroup of $%
G(\Sigma )$ by Lemma \ref{2S}. Then $G(\Sigma )_{y}$ fixes also $y^{\prime }$
in $\Delta ^{\prime }$ with $y^{\prime }\neq y$, and we reach a
contradiction. Thus $\mathrm{Fix}(G(\Sigma )_{y})\cap \mathrm{Fix}(G(\Sigma
)_{y^{\prime }})=\varnothing $, and hence $\left\vert \mathrm{Fix}%
(S_{0})\right\vert \geq 2\left\vert \Sigma \right\vert $ since $S_{0}\leq
G(\Sigma )_{y}\cap G(\Sigma )_{y^{\prime }}$. Now, if we use the above
argument this time with $\alpha \in S_{0}$, $\alpha \neq 1$, we reach a contradiction.
Therefore $\left\vert S\right\vert =2$. Also, $G(\Sigma )=O(G(\Sigma )).S$ by Proposition \ref{DivanDan}(1), and $%
\left\vert \mathrm{Fix}(S)\right\vert =\left\vert \Sigma \right\vert $.

Let $\Lambda =\left\{ Fix(S)^{\gamma }:\gamma \in G(\Sigma )\right\} $.
Since $S=G(\Sigma )_{x}$ is a Sylow $2$-subgroup of $G(\Sigma )$, $G(\Sigma
)\vartriangleleft G$ and $G$ acts point-transitively on $\mathcal{D}$, it
follows that $\Lambda $ is a $G$-invariant partition of the point set of $\mathcal{D}$ in $\left\vert \Delta
\right\vert $ blocks each of size $\left\vert \Sigma \right\vert $. Then $%
\left\vert \Sigma \right\vert =\left\vert \Delta \right\vert $ by
Lemma \ref{Mpomoc}(2) and hence $\lambda =2$, but this contradicts our
assumptions.

Assume that $G(\Sigma )_{x}=1$. Then $G(\Sigma )=O(G(\Sigma ))$. Moreover, $%
Soc(G_{\Delta }^{\Delta })\trianglelefteq G(\Sigma )^{\Delta }\cong G(\Sigma
)$ by Proposition \ref{DivanDan}(1). Then $G(\Sigma )= Soc(G_{\Delta
}^{\Delta })$ is an elementary abelian $p$-group for some odd prime $p$,
since $G(\Sigma )$ acts regularly on $\Delta $ and $\left\vert \Delta
\right\vert $ is odd. Hence $G(\Sigma )$ is abelian, and so $G(\Sigma
)\trianglelefteq C_{G}(G(\Sigma ))\trianglelefteq G$. If $C_{G}(G(\Sigma
))\neq G(\Sigma )$, then $G=C_{G}(G(\Sigma ))$ since $G^{\Sigma }$ is
primitive on $\Sigma $ by Proposition \ref{DivanDan}. This implies $%
G(\Sigma ) \leq Z(G)$ and hence $G_{x}\leq G(\Delta )$
for any $x\in \Delta $. This is a contradiction since $G_{\Delta }^{\Delta }$
is $2$-transitive on $\Delta $. Therefore $C_{G}(G(\Sigma ))=G(\Sigma
) = Soc(G_{\Delta }^{\Delta })$. Therefore $G^{\Sigma }\leq \mathrm{Aut}%
(G(\Sigma ))\cong GL_{d}(p)$.
\end{proof}

\bigskip

\section{Further Reductions}

The aim of this section is to prove the following reduction result:

\begin{theorem}
\label{AQP}One of the following holds:

\begin{enumerate}
\item $G(\Sigma )=1$ and $G$ acts point-quasiprimitively on $\mathcal{D}$.

\item $G(\Sigma )$ is a non-trivial self-centralizing elementary abelian $2$%
-subgroup of $G$ and the following hold:

\begin{enumerate}
\item $\mathcal{D}$ is a symmetric $2$-$(2^{d+2}(2^{d-1}-1)^{2},\allowbreak
2\left( 2^{d}-1\right) \left( 2^{d-1}-1\right) ,2(2^{d-1}-1))$ design with $%
d\geq 4$;

\item Either $G_{x}^{\Delta }\cong SL_{d}(2)$, or $G_{x}^{\Delta }\cong
A_{7} $ and $d=4$.

\item $G^{\Sigma }$ is almost simple.
\end{enumerate}
\end{enumerate}
\end{theorem}

Its proof is structured as follows. Case (1) is an immediate consequence of Proposition \ref{DivanDan}(2), whereas the proof of (2) relies mainly on Corollary \ref{CFixT} and on O'Nan Scott theorem applied to $G^{\Sigma}$ for $v$ even, and on Corollary \ref{CFixT} and on \cite[Theorem 3.1]{BP} for $v$ odd. 

\bigskip

\subsection{Designs of type 1 and quasiprimitivity}

\begin{lemma}
\label{T1even}Let $\mathcal{D}$ be of type 1. If $G(\Sigma )\neq 1$ and $v$ is even, then the following hold:

\begin{enumerate}
\item $\mathcal{D}$ is a symmetric $2$-$(2^{d+2}(2^{d-1}-1)^{2},\allowbreak
2\left( 2^{d}-1\right) \left( 2^{d-1}-1\right) ,2(2^{d-1}-1))$ design, $d \geq 4$;

\item $G(\Sigma )$ is a non-trivial self-centralizing elementary abelian $2$%
-subgroup of $G$. Also $G_{\Delta }^{\Sigma }/G(\Delta )^{\Sigma }\cong
G_{x}^{\Delta }$ and one of the following holds:

\begin{enumerate}
\item $G_{x}^{\Delta }\cong SL_{d}(2)$;

\item $G_{x}^{\Delta }\cong A_{7}$ and $d=4$.
\end{enumerate}
\end{enumerate}
\end{lemma}

\begin{proof}
Assume that $\mathcal{D}$ is of type 1, $G(\Sigma )\neq 1$ and $v$ is even. Since $Soc(G_{\Delta
}^{\Delta })\trianglelefteq G(\Sigma )^{\Delta }$ by Proposition \ref%
{DivanDan}(1), it follows that $Soc(G_{\Delta }^{\Delta })_{x}\trianglelefteq
G(\Sigma )_{x}^{\Delta }$, where $x\in \Delta $. Hence, $Soc(G_{\Delta
}^{\Delta })_{x}$ is either trivial or a $2$-group by Lemma \ref{2S}. If $G_{\Delta }^{\Delta
} $ is almost simple, then $Soc(G_{\Delta }^{\Delta })$ acts $2$%
-transitively on $\Delta $, with $Soc(G_{\Delta }^{\Delta })_{x}$ a $2$%
-group. However, this is impossible by \cite[List (A)]{Ka}. Therefore, $%
G_{\Delta }^{\Delta }$ is of affine type, and hence $\left\vert \Delta
\right\vert =2^{d}$ since $v=\lambda^{2}(\lambda+2)$ is even. Then $\lambda =2(2^{d-1}-1)$ and so $%
\left\vert \Sigma \right\vert =2^{2}(2^{d-1}-1)^{2}$, where $d \geq 4$ since $\lambda >10$. In particular, $%
\mathcal{D}$ is a symmetric $2$-design with parameters as in (1). Moreover,
by \cite[List (B)]{Ka} one of the following holds:

\begin{enumerate}
\item[(i)] $G_{x}^{\Delta }\leq \Gamma L_{1}(2^{d})$;

\item[(ii)] $SL_{d/h}(2^{h})\trianglelefteq G_{x}^{\Delta }\leq \Gamma
L_{d/h}(2^{h})$ with $d/h>1$;

\item[(iii)] $Sp_{d/h}(2^{h})\trianglelefteq G_{x}^{\Delta }\leq \Gamma
Sp_{d/h}(2^{h})$ with $d/h>1$ and $d/h$ even,

\item[(iv)] $G_{2}(2^{d/6})\trianglelefteq G_{x}^{\Delta }\leq
G_{2}(2^{d/6}):Z_{d/6}$ with $d\equiv 0\pmod{6}$.

\item[(v)] $G_{x}^{\Delta }\cong A_{6}$ or $A_{7}$ and $d=4$.
\end{enumerate}

Since $G_{x}$ is transitive on the $\lambda (\lambda +1)$ blocks incident
with $x$, it follows that $2\left( 2^{d-1}-1\right) \left( 2^{d}-1\right)
\mid \left\vert G_{x}\right\vert $. Suppose that there is a prime $u$
dividing $\lambda/2=2^{d-1}-1$ but not dividing of the order of $G_{x}^{\Delta }$.
Hence, $u$ divides the order of $G(\Delta )$. Let $U$ be a Sylow $u$%
-subgroup of $G(\Delta )$, then $G_{x}=N_{G_{x}}(U)G(\Delta )$ by the Frattini argument.

Let $y \in \mathrm{Fix(U)} \setminus \{x\}$ and let $\Delta ^{\prime }$ be the element of $\Sigma$ containing $y$. Then $U \leq G_{y}$, and actually $U \leq G(\Delta^{\prime})$ since $G_{x}$ and $G_{y}$ are $G$-conjugate (clearly, $\Delta=\Delta^{\prime}$ is possible). Therefore $\left\vert\mathrm{Fix}(U) \right \vert = t(\lambda+2)$, where $t \geq 2$ since $\left\vert\Sigma \right\vert = \lambda^{2}$. Also, $U$ is a Sylow $u$-subgroup of $G(\Delta ,\Delta ^{\prime })$. Thus $%
G_{x,y}=N_{G_{x,y}}(U)G(\Delta ,\Delta ^{\prime })$ again by Frattini argument since $G(\Delta ,\Delta ^{\prime
})\vartriangleleft G_{x,y}$. Then%
\begin{equation*}
\left\vert G_{x}:G_{x,y}\right\vert = \frac{\left\vert
N_{G_{x}}(U):N_{G_{x,y}}(U)\right\vert \cdot \left\vert G(\Delta
):G(\Delta ,\Delta ^{\prime })\right\vert }{\left\vert N_{G(\Delta
)}(U):N_{G(\Delta ,\Delta ^{\prime })}(U)\right\vert }
\end{equation*}%
and hence $\lambda +1\mid \left\vert
N_{G_{x}}(U):N_{G_{x,y}}(U)\right\vert $, since $\lambda +1\mid
\left\vert G_{x}:G_{x,y}\right\vert $ by Lemma \ref{PP}, and since $\left( \lambda
+1,\left\vert G(\Delta )\right\vert \right) =1$ by Corollary \ref{CFixT} being $\lambda$ even, as $v=\lambda^{2}(\lambda+2)$ is even.
Therefore $\lambda +1\mid \left\vert y^{N_{G_{x}}(U)}\right\vert $ and $y^{N_{G_{x}}(U)} \subseteq \mathrm{Fix}(U)\setminus \{x \}$. Since $\mathrm{Fix}(U)\setminus \{x \}$ is a disjoint union of $N_{G_{x}}(U)$-orbits, it follows that $\lambda +1 \mid \left\vert\mathrm{Fix}(U) \right \vert-1$. Therefore, $\lambda+1 \mid t-1$ since $\left\vert\mathrm{Fix}(U) \right \vert = t(\lambda+2)$. Therefore, $\left\vert \mathrm{Fix}(\zeta)\right\vert \geq \left\vert\mathrm{Fix}(U) \right \vert  \geq (\lambda+2)^{2}$ for any non-trivial element $\zeta \in U$ since $t \geq 2$, but
this contradicts Lemma \ref{Fixpoints}(1). Thus, each prime divisor of $%
2^{d-1}-1$ divides $\left\vert G_{x}^{\Delta }\right\vert $. Then $%
G_{x}^{\Delta }\cong SL_{7}(2)$ for $d=7$, whereas $2^{d-1}-1$ admits a
primitive prime divisor for $d\neq 7$ by \cite[Theorem II.6.2]{Lu}. A this point, it is easy to check that only $%
G_{x}^{\Delta }\cong SL_{d}(2)$, or $G_{x}^{\Delta }\cong A_{7}$ for $d \geq 4$ by using \cite[Proposition 5.2.15]{KL}.

The two possibilities for $G_{x}^{\Delta }$ together with Proposition \ref{DivanDan}(1) imply either $G(\Sigma )^{\Delta }=G_{\Delta }^{\Delta }$ or $%
G(\Sigma )^{\Delta }=Soc(G_{\Delta }^{\Delta })$. The former yields $%
G_{\Delta }=G(\Sigma )G(\Delta )$, and hence $\lambda +1\mid \left\vert G(\Sigma
)_{x}\right\vert $ since $\left( \lambda
+1,\left\vert G(\Delta )\right\vert \right) =1$ by Corollary \ref{CFixT}. However, this contradicts Lemma \ref{2S} since $\lambda +1$ odd. Thus $G(\Sigma )^{\Delta }=Soc(G_{\Delta }^{\Delta })$, and hence 
$
\frac{G_{\Delta }^{\Sigma }}{G(\Delta )^{\Sigma }}\cong \frac{G_{\Delta }^{\Delta }}{G(\Sigma )^{\Delta }}%
\cong G_{x}^{\Delta }
$
by (\ref{salvaMAOL}).
Since $\frac{G(\Sigma )}{G(\Sigma )\cap G(\Delta )}\cong Soc(G_{\Delta
}^{\Delta })$, which is an elementary abelian $2$-group, we have $\Phi (G(\Sigma ))\leq G(\Sigma )\cap G(\Delta )$ for
each $\Delta \in \Sigma $. Thus $\Phi (G(\Sigma ))$ fixes each point of $%
\mathcal{D}$, hence $\Phi (G(\Sigma ))=1$, and so $G(\Sigma )$ is an
elementary abelian $2$-group.

If $C_{G}(G(\Sigma ))\neq G(\Sigma )$, then $%
G=C_{G}(G(\Sigma ))$ since $G^{\Sigma }$ is primitive on $\Sigma $ by
Proposition \ref{DivanDan}. Thus $G(\Sigma )\leq Z(G)$ acts
transitively on each $\Delta $, and hence $G_{x}\leq G(\Delta )$ for any $%
x\in \Delta $. However, this is a contradiction since $G_{\Delta }^{\Delta }$
is $2$-transitive on $\Delta $. Therefore $C_{G}(G(\Sigma ))=G(\Sigma )$,
and we obtain (2).
\end{proof}

\bigskip

\begin{theorem}
\label{veven}If $\mathcal{D}$ is of type 1 and $v$ is even, then Theorem \ref{AQP} holds.
\end{theorem}

\begin{proof}
Assume that that $\mathcal{D}$ is of type 1. The assertion follows from Proposition \ref{DivanDan}(2) for $G(\Sigma )=1$.
Hence, assume that $G(\Sigma )\neq 1$. Since $G^{\Sigma }$ acts
primitively on $\Sigma $ again by Proposition \ref{DivanDan} and $\left\vert
\Sigma \right\vert =2^{2}(2^{d-1}-1)^{2}$ with $d\geq 4$ by Lemma \ref%
{T1even}, by the O'Nan-Scott theorem (e.g. see \cite[Theorem 4.1A]{DM}) one
of the following holds:

\begin{enumerate}
\item[(i)] $Soc(G^{\Sigma })$ is almost simple.

\item[(ii)] $Soc(G^{\Sigma })\cong T^{2}$, where $T$ is a non-abelian simple
group such that $\left\vert T\right\vert =2^{2}(2^{d-1}-1)^{2}$.

\item[(iii)] $Soc(G^{\Sigma })\cong T^{2}$ and there is a
non-abelian almost simple group $Q$ with socle $T$ acting primitively on a set $\Theta$ of size $2(2^{d-1}-1)$ such that $\Sigma=\Theta^{2}$ and $G^{\Sigma}\leq Q\wr Z_{2}$.
\end{enumerate}

Assume that (ii) holds. Then $T\leq G_{\Delta }^{\Sigma }\leq Aut(T)\times
Z_{2}$, and hence $G_{\Delta }^{\Sigma }/T$ is solvable. Therefore $\left\vert
G_{x}^{\Delta }\right\vert \mid \left\vert T\right\vert $ since $G_{\Delta
}^{\Sigma }/G(\Delta )^{\Sigma }\cong G_{x}^{\Delta }$ with $G_{x}^{\Delta }$ isomorphic either to $SL_{d}(2)$, or to $A_{7}$ for $d=4$ by Lemma \ref{T1even}(2).
However, this is impossible since $\left\vert G_{x}^{\Delta }\right\vert \nmid 2^{2}(2^{d-1}-1)^{2}$, and (ii) is excluded.

Assume that (iii) holds. Since $G_{\Delta }^{\Sigma }/G(\Delta )^{\Sigma
}\cong G_{x}^{\Delta }$ by Lemma \ref{T1even}(2), then $\lambda +1=2^{d}-1$ divides the order of $G_{\Delta }^{\Sigma }$. If $d=6$ then $\lambda =2\cdot 31$, $\left\vert
\Sigma \right\vert =2^{2}31^{2}$ and $G_{x}^{\Delta }\cong SL_{6}(2)$. Since 
$\lambda (\lambda +1)\mid \left \vert G_{\Delta } \right \vert $, it follows that $31^{3}\mid
\left\vert G\right\vert $. On the other hand, $G^{\Sigma }\leq Q\wr
Z_{2}$, where $\mathrm{Soc}(Q)$ is isomorphic to one of the groups $A_{62}$ or $PSL_{2}(61)$ by \cite[Table B.2]{DM}, and hence $31$ divides the order of $G(\Sigma
)$, which is not the case by Lemma \ref{T1even}(2). Thus $d\neq 6$, and hence there is a primitive prime divisor of $\lambda+1=2^{d}-1$, say $z$, by \cite[Theorem II.6.2]{Lu}. If $%
z\mid \left\vert G_{\Delta ,\Delta_{1}}\right\vert $ for some $\Delta
_{1}\in \Sigma \setminus \{\Delta\}$, it follows that there is a $u$-element $\varphi $
of $G_{\Delta ,\Delta _{1}}$ fixing a point $x$ in $\Delta $ and $%
x^{\prime }$ in $\Delta _{1}$. Therefore, it fixes at least one of the 
$\lambda =2(2^{d-1}-1)$ blocks incident with $x$ and $x^{\prime }$, say $B$.
Hence, $\varphi $ fixes a further element in $B\cap \Delta $. Thus, $%
\varphi \in G(\Delta )$ by \cite[Theorem 3.5]{He} since $G_{\Delta}^{\Delta}$ is of affine type. However, this is
impossible by Corollary \ref{CFixT}(1) since $u\mid \lambda +1$. Then $z\nmid
\left\vert G_{\Delta ,\Delta _{1}}\right\vert $ and hence $z\mid
\left\vert G_{\Delta }^{\Sigma }:G_{\Delta ,\Delta ^{\prime }}^{\Sigma
}\right\vert $ for each $\Delta
^{\prime}\in \Sigma \setminus \{\Delta\}$ since $G(\Sigma )$ is a $2$-group by Lemma \ref{T1even}(2). Actually, $z\mid
\left\vert G_{\Delta }^{\Sigma }:G(\Delta )^{\Sigma }G_{\Delta ,\Delta
^{\prime }}^{\Sigma }\right\vert $ again by Corollary \ref{CFixT}(1), hence $z\mid \left\vert G_{\Delta
}^{\Sigma }/G(\Delta )^{\Sigma }:G(\Delta )^{\Sigma }G_{\Delta ,\Delta ^{\prime }}^{\Sigma
}/G(\Delta )^{\Sigma }\right\vert $. Thus 
\begin{equation*}
P(G_{x}^{\Delta })\leq \left\vert G_{\Delta }^{\Sigma }/G(\Delta )^{\Sigma
}:G(\Delta )^{\Sigma }G_{\Delta ,\Delta ^{\prime }}^{\Sigma }/G(\Delta )^{\Sigma }\right\vert
\leq \left\vert G_{\Delta }^{\Sigma }:G_{\Delta ,\Delta ^{\prime }}^{\Sigma
}\right\vert \text{,}
\end{equation*}%
where $P(G_{x}^{\Delta })$ is the minimal primitive permutation
representation of $G_{x}^{\Delta }$, since $G_{\Delta }^{\Sigma }/G(\Delta
)^{\Sigma }\cong G_{x}^{\Delta }$ by Lemma \ref{T1even}(2). We may choose $\Delta
^{\prime}\in \Sigma \setminus \{\Delta\}$ such that 
\begin{equation*}
P(G_{x}^{\Delta })\leq \left\vert G_{\Delta }^{\Sigma }:G_{\Delta ,\Delta
^{\prime }}^{\Sigma }\right\vert \leq 2\frac{2(2^{d-1}-1)}{s-1}
\end{equation*}%
where $s$ denotes the rank of $Q$ on $\Theta$ since $G^{\Sigma} \leq Q \wr S_{2}$ and $\Sigma=\Theta^{2}$. If $d\neq 4$ then $%
P(G_{x}^{\Delta })=2^{d}-1$ by \cite[Theorem 5.2.2]{KL}, hence $s=2$. Then $%
Q $ acts $2$-transitively on the set $\Theta $, with $\left\vert \Theta
\right\vert =2(2^{d-1}-1)$, and hence $Q\wr Z_{2}$ acts as a primitive rank $3$ group on $%
\Sigma =\Theta^{2}$. Moreover, the $\left( Q\wr Z_{2}\right) _{\Delta }$-orbits on $\Sigma \setminus \left\{ \Delta \right\} $ are two of length $%
\left\vert \Theta \right\vert -1$, one of $\left( \left\vert \Theta
\right\vert -1\right) ^{2}$. Each of
these orbits is a union of $G_{\Delta }^{\Sigma }$-orbits, and since each $%
G_{\Delta }^{\Sigma }$-orbits distinct from $\left\{ \Delta \right\} $ is
divisible by $z$, then $z\mid \left\vert \Theta \right\vert -1$. So $z\mid
2^{d}-3$, whereas $z$ is a divisor of $2^{d}-1$. Therefore $d=4$, $s=2,3,4$
and $G_{x}^{\Delta }\cong A_{m}$, where $m =7,8$, since $SL_{4}(2)\cong A_{8}$.

Suppose that $s\geq 3$. Then $\left\vert G_{\Delta }^{\Sigma }:G_{\Delta
,\Delta ^{\prime }}^{\Sigma }\right\vert \leq \frac{28}{s-1}\leq 14$, and
hence $\left\vert G_{\Delta }^{\Sigma }:G_{\Delta ,\Delta ^{\prime
}}^{\Sigma }\right\vert =m $ by \cite{At}. Then $m \mid \left\vert
\Sigma \right\vert -1$ with $m=7,8$, and we reach a contradiction since $\left\vert \Sigma
\right\vert -1=\allowbreak 3\cdot 5\cdot 13$. Thus $s=2$,
and hence $Q$ acts $2$-transitively on $\Theta$. As above, $Q\wr Z_{2}$ acts as a primitive rank $3$ group on $%
\Sigma =\Theta^{2}$, and the $\left( Q\wr Z_{2}\right) _{\Delta }$-orbits on $\Sigma \setminus \{\Delta\}$ are two of length $%
13$ and one of length $13^{2}$. None of these lengths is divisible by $z=5$, and so this case is excluded. Thus only (i) occurs, which is the assertion.
\end{proof}

\bigskip
\begin{remark}
If $\mathcal{D}$ is of type 1 with $v$ even and $G(\Sigma) \neq 1$, it follows from Lemma \ref{T1even}(2) and Theorem \ref{veven} that, $G^{\Sigma}$ is an almost simple subgroup of $GL_{d+t}(2)$ , where $\left\vert\Delta\right\vert =2^{d}$ with $\Delta \in \Sigma$, $\left\vert G(\Sigma)\right\vert=2^{d+t}$ and $\left\vert G(\Sigma) \cap G(\Delta)\right\vert=2^{t}$. However, it is not easy to fully exploit the previous property because it is not easy to control the order of $G(\Sigma) \cap G(\Delta)$ in this case. This motivates our choice of an alternative proof in which the embedding of $G^{\Sigma}$ in $GL_{d+t}(2)$ is partially considered.    
\end{remark}
\bigskip

\begin{theorem}
\label{D1PqP}If $\mathcal{D}$ is of type 1 and $v$ is odd, then $G$ acts
point-quasiprimitively on $\mathcal{D}$.
\end{theorem}

\begin{proof}
Assume that $\mathcal{D}$ is of type 1. Recall that $G(\Sigma )$ is an
elementary abelian $p$-group acting regularly
on $\Delta $, $p$ is and odd prime and $G^{\Sigma
}\leq GL_{d}(p)$ by Lemma \ref{strike}. Thus $\lambda+2=\left \vert \Delta \right \vert=p^{d}$, $d \geq 1$, and hence $\lambda =p^{d}-2$. Therefore, $%
(p^{d}-1)(p^{d}-2)$ divides $\left\vert G_{x}\right\vert $, and hence $%
\left\vert G\right\vert $, as $k\mid \left\vert G_{x}\right\vert $.
Also, $(p^{d}-1)(p^{d}-2)^{3}\mid \left\vert G^{\Sigma }\right\vert $ since $\left\vert\Sigma \right\vert=(p^{d}-2)^{2}$.

If $d=2$ then $(p^{2}-1)(p^{2}-2)^{3}\mid \left\vert G^{\Sigma}\right\vert $ with $%
G^{\Sigma }\leq GL_{2}(p)$, which is a contradiction. Thus $d>2$, and hence $%
\Phi _{d}^{\ast }(p)>1$ by \cite[Theorem II.6.2]{Lu} since $p$ is odd. Then $%
G^{\Sigma }$ is an irreducible subgroup of $GL_{d}(p)$ by \cite[Theorem
3.5(iv)]{He}. For each divisor $m$ of $d$ the group $\Gamma L_{d/m}(p^{m})$
has a natural irreducible action. We may choose $m$ to be minimal such that $%
G^{\Sigma }\leq \Gamma L_{d/m}(p^{m})$. Set $K^{\Sigma }=G^{\Sigma }\cap
GL_{d/m}(p^{m})$. Then $%
\Phi _{d}^{\ast }(p)\frac{(p^{d}-2)^{3}}{((p^{d}-2)^{3},d)}\mid \left\vert
K^{\Sigma }\right\vert $ by \cite[Proposition 5.2.15.(ii)]{KL}. Easy
computations show that the order of $GL_{d/m}(p^{m})$, and hence that of $%
K^{\Sigma }$, is not divisible by $p^{d}\Phi _{d}^{\ast }(p)\frac{(p^{d}-2)}{%
((p^{d}-2),d)}$ for $(d/m,p^{m})=(3,5^{2}),(4,3),(6,3),(6,5),(9,3)$. So
these cases are excluded. Thus, bearing in
mind the minimality of $m$ and the fact that $p$ is odd, \cite[Theorem 3.1]{BP} implies that $K^{\Sigma }$ contains a normal subgroup isomorphic to one of the groups $%
SL_{d/m}(p^{m}),Sp_{d/m}(p^{m}),\Omega _{d/m}^{-}(p^{m})$, or $%
SU_{d/m}(p^{m/2})$ with $d/m$ odd. Since $\left\vert
G^{\Sigma }:G_{\Delta }^{\Sigma }\right\vert =\lambda ^{2}$, and $\lambda
=p^{d}-2$ with $p$ odd, it follows that $G_{\Delta }^{\Sigma }$ is a maximal
parabolic subgroup of $G^{\Sigma }$ by Proposition \ref{DivanDan} and \cite[Theorem 1.6]{Se}. Also, $\Phi _{d}^{\ast
}(p)\mid \left\vert G_{\Delta }^{\Sigma }\right\vert $ since $\left( \Phi
_{d}^{\ast }(p),p^{d}-2\right) =1$, but this contradicts \cite[Theorem
3.5(iv)]{He} applied to $G_{\Delta }^{\Sigma }$. Thus $G(\Sigma )=1$, and
hence $G$ acts point-quasiprimitively on $\mathcal{D}$ by Proposition \ref%
{DivanDan}(2).
\end{proof}

\bigskip

\subsection{Designs of type 2 and quasiprimitivity}

\bigskip

\begin{lemma}\label{firstpart}
If $\mathcal{D}$ is of type 2 with $\lambda =2w^{2}$, where $w$ is
odd, $w\geq 3$, and $2(w^{2}-1)$ is a square, then $G$ acts point-quasiprimitively on $\mathcal{D}$.
\end{lemma}

\begin{proof}
Suppose that $\mathcal{D}$ is of type 2 with $\lambda =2w^{2}$, where $u$ is
odd, $u\geq 3$, and $2(w^{2}-1)$ is a square. Then $\left\vert \Delta \right\vert=\lambda/2+1=w^{2}+1$ is even.
If $\mathrm{Soc}(G_{\Delta}^{\Delta})$ is an elementary abelian $2$-group, then $w^{2}+1=2^{s}$ for some $s\geq 1$. However, it has no integer solutions for \cite[B1.1]{Rib}. Thus $\mathrm{Soc}(G_{\Delta}^{\Delta})$ is non-abelian simple, and hence $\mathrm{Soc}(G_{\Delta}^{\Delta})$ is isomorphic to one of the groups $ A_{w^{2}+1}$, $PSL_{d}(s)$ with $w^{2}+1=\frac{s^{d}-1}{s-1}$, $d \geq 2$ and $(d,s)\neq (2,2),(2,3)$, or $PSU_{3}(w^{2/3})$ by \cite[List (A)]{Ka} since $\left\vert \Delta \right\vert=w^{2}+1$. In the second case one has $w^2=s\frac{s^{d-1}-1}{s-1}$, and so $s$ is an even power of an odd prime. Therefore $\left(\frac{w}{s^{1/2}}\right)^{2}=\frac{s^{d-1}-1}{s-1}$, and hence $d=2$ by \cite[A7.1, A8.1 and B1.1]{Rib}, and so $\mathrm{Soc}(G_{\Delta}^{\Delta}) \cong  PSL_{2}(w^{1/2})$. Also, $w^{2}-1 \neq 2^{t}$ for some $t \geq 1$. Indeed, if not so, then $t=w=3$ by \cite[B1.1]{Rib}, and hence $(\lambda,\left\vert \Delta \right \vert, \left\vert \Sigma \right \vert)=(18,10,145)$ and $\mathrm{Soc}(G_{\Delta}^{\Delta}) \cong A_{5}$ and $\mathrm{Soc}(G^{\Sigma}) \cong A_{145}$ by \cite[Table B.4]{DM}. Then $A_{144} \trianglelefteq G_{\Delta}^{\Sigma} \leq S_{144}$ but this contradicts (\ref{salvaMAOL}) since $G_{\Delta}^{\Delta} \leq S_{5}$. Thus in each case $\mathrm{Soc}(G_{\Delta}^{\Delta})$ contains elements of order an odd prime divisor of $w^{2}-1$, and each of these elements fixes at least two points on $\Delta$. 

If $G(\Sigma)\neq 1$ then $\mathrm{Soc}(G_{\Delta}^{\Delta}) \trianglelefteq G(\Sigma)^{\Delta}$ by Proposition \ref{DivanDan}(1). Now, let $\eta \in G(\Sigma)$ be an element of order an odd prime divisor of $w^{2}-1$, which exists by the previous argument. Then, for each $\Delta \in \Sigma$ either $\eta \in G(\Delta)$ or $\eta$ induces an element of $G(\Sigma)^{\Delta}$ fixing at least two points of $\Delta$. Therefore $\left\vert \mathrm{Fix}(\eta)\right\vert \geq 2\left\vert\Sigma \right\vert=\lambda^2-2\lambda+2$, but his contradicts Lemma \ref{Fixpoints}(2). Thus $G(\Sigma)= 1$, and hence $G$ acts point-quasiprimitively on $\mathcal{D}$ by
Proposition \ref{DivanDan}(2).   
\end{proof}

\begin{theorem}
\label{GSigma}If $\mathcal{D}$ is of type 2 then $G$ acts
point-quasiprimitively on $\mathcal{D}$.
\end{theorem}

\begin{proof}
Suppose that $\mathcal{D}$ is of type 2. If $\lambda =2w^{2}$, where $w$ is
odd, $w\geq 3$, and $2(w^{2}-1)$ is a square, the assertion follows from Lemma \ref{firstpart}. Hence, we may assume that $\lambda \equiv 0 \pmod{4}$ by Theorem \ref{PZM}. Hence, $v$ is odd.

Suppose that $G(\Sigma )\neq 1$.
Then $G(\Sigma )\cong Soc(G_{\Delta })$, where $\Delta \in \Sigma $, is an
elementary abelian $p$-group of order acting regularly on $\Delta $, where $%
p $ is odd, and $G^{\Sigma }\leq GL_{d}(p)$, with $p^{d}=\left\vert \Delta
\right\vert $ by Lemma \ref{strike} since $v$ is odd. Then $\lambda =2(p^{d}-1)$ and $%
k=\lambda ^{2}/2=2(p^{d}-1)^{2}$. Then $\left( p^{d}-1\right) ^{2}\mid
\left\vert G_{B}\right\vert $, where $B$ is any block of $\mathcal{D}$,
since $G_{B}$ acts transitively on $B$. Then $\left( p^{d}-1\right) ^{2}$
divides $\left\vert G^{\Sigma }\right\vert $ since $\left\vert G(\Sigma
)\right\vert =p^{d}$. Now, we may proceed as in Theorem \ref{D1PqP} and we
see that no admissible groups arise in this case as well. Thus $G(\Sigma
)=1 $, and hence $G$ acts point-quasiprimitively on $\mathcal{D}$ by
Proposition \ref{DivanDan}(2).
\end{proof}

\bigskip

\begin{proof}[Proof of Theorem \protect\ref{AQP}]
The assertion follows from Theorems \ref{veven} and \ref{D1PqP} for $\mathcal{D}$ of type 1, Theorem \ref{GSigma} for $\mathcal{D}$ of type 2.
\end{proof}

\bigskip

\section{Reduction to the almost simple case}

In this section we analyze the case where $G$ acts point-quasiprimitively on 
$\mathcal{D}$. In this case $G=G^{\Sigma }$. The main tool used to tackle this
case is the O'Nan-Scott Theorem for quasiprimitive groups proven in \cite{P1} and reported below for reader's convenience. We investigate
each of the seven possibilities for $G^{\Sigma }$ provided in the above mentioned theorem by adapting the techniques
developed in \cite{ZZ}, and we show that $\mathrm{Soc}(G^{\Sigma })$ is a
non-abelian simple group. This fact, together with the conclusions of
Theorem \ref{AQP}, yields the following result.

\begin{theorem}
\label{ASPQP}$G^{\Sigma }$ is an almost simple group acting primitively on \ $%
\Sigma $. Moreover, one of the following holds:

\begin{enumerate}
\item $G(\Sigma )=1$ and $G$ acts point-quasiprimitively on $\mathcal{D}$.

\item $G(\Sigma )$ is a self-centralizing elementary abelian $2$%
-subgroup of $G$. Also, the following hold:

\begin{enumerate}
\item $\mathcal{D}$ is a symmetric $2$-$(2^{d+2}(2^{d-1}-1)^{2},\allowbreak
2\left( 2^{d}-1\right) \left( 2^{d-1}-1\right) ,2(2^{d-1}-1))$ design with $%
d\geq 4$;

\item $G_{\Delta }^{\Sigma }/G(\Delta )^{\Sigma }\cong G_{x}^{\Delta }$ and
either $G_{x}^{\Delta }\cong SL_{d}(2)$, or $G_{x}^{\Delta }\cong A_{7}$ and 
$d=4$.
\end{enumerate}
\end{enumerate}
\end{theorem}

\bigskip

We only need to prove that $G^{\Sigma}$ is almost simple, the remaining parts of Theorem \ref{ASPQP} have been already proven in Theorem \ref{AQP}

\bigskip
In the sequel we denote $\mathrm{Soc}(G)$ simply by $L$ and let $x\in 
\mathcal{P}$. Then $L\cong T^{h}$ with $h\geq 1$, where $T$ is a simple
group. By \cite[Theorem 1]{P1}, one of the following holds:

\begin{enumerate}
\item[I.] \emph{Affine groups.} Here $T\cong Z_{p}$ for some prime $p$, and $%
L$ is the unique minimal normal subgroup of $G$ and is regular on $\mathcal{P%
}$ of size $p^{h}$. The set $\mathcal{P}$ can be identified with $L\cong Z_{p}^{h}$ so that $%
G\leq AGL_{h}(p)$ with $L$ the translation group and $G_{x}=G\cap GL_{h}(p)$
acting irreducibly on $L$. Moreover, $G$ acts primitively on $\mathcal{P}$.

\item[II.] \emph{Almost simple groups.} Here $h=1$, $T$ is a non-abelian
simple group, $T\trianglelefteq G\leq \mathrm{Aut}(T)$ and $G=TG_{x}$.\emph{%
\ }

\item[III.] In this case $L\cong T^{h}$ with $h\geq 2$ and $T$ is a non-abelian
simple group. We distinguish three types:

\item[III(a).] \emph{Simple diagonal action.} Define%
\small
\begin{equation*}
W=\left\{ (a_{1},...,a_{h})\cdot \pi :a_{i}\in \mathrm{Aut}(T),\pi \in
S_{h},a_{i}\equiv a_{j}\pmod{\mathrm{Inn}(T)}\text{ for all }i,j\right\} \text{,%
}
\end{equation*}%
\normalsize
where $\pi \in S_{h}$ just permutes the components $a_{i}$ naturally. With
the usual multiplication, $W$ is a group with socle $L\cong T^{h}$, and $%
W=L.(\mathrm{Out}(T)\times S_{h})$. The action of $W$ on $\mathcal{P}$ is
defined by setting%
\begin{equation*}
W_{x }=\left\{ (a,...,a)\cdot \pi :a\in \mathrm{Aut}(T),\pi \in
S_{h}\right\} \text{.}
\end{equation*}%
Thus $W_{x }\cong \mathrm{Aut}(T)\times S_{h}$, $L_{x }\cong T$
and $\left\vert \mathcal{P}\right\vert =\left\vert T\right\vert ^{h-1}$.

For $1\leq i\leq h$ let $T_{i}$ be the subgroup of $L$ consisting of the $h$%
-tuples with $1$ in all but the $i$-th component, so that $T_{i}\cong T$ and 
$L\cong T_{1}\times \cdots \times T_{h}$. Put $\mathcal{T}=\left\{
T_{1},...,T_{h}\right\} $, so that $W$ acts on $\mathcal{T}$. We say
that subgroup $G$ of $W$ is of type III(a) if $L\leq G$ and, letting $P$ the permutation group of $G^{%
\mathcal{T}}$, one of the following holds:

\begin{enumerate}
\item[(i)] $P$ is transitive on $\mathcal{T}$;

\item[(ii)] $h=2$ and $P=1$.
\end{enumerate}

We have $G_{x }\leq \mathrm{Aut}(T)\times P$ and $G\leq L.(\mathrm{Out}%
(T)\times P)$. Moreover, in case (i) $L$ is the unique minimal normal
subgroup of $G$ and $G$ is primitive on $\mathcal{P}$ if and only if $P$ is
primitive on $\mathcal{T}$. In case (ii) $G$ has two minimal normal
subgroups $T_{1}$ and $T_{2}$, both regular on $\mathcal{P}$, and $G$ is
primitive on $\mathcal{P}$.

\item[III(b).] \emph{Product action.} Let $H$ be a quasiprimitive permutation group on a set $%
\Gamma $ of type II or III(a). For $l>1$, let $W=H\wr S_{l}$, and take $W$
to act $\Lambda =\Gamma ^{l}$ in its natural product action. Then for $y\in
\Gamma $ and $z=(y,...,y)\in \Lambda $ we have $W_{z}=H_{y}\wr S_{l}$ and $%
\left\vert \Lambda \right\vert =\left\vert \Gamma \right\vert ^{l}$. If $K$
is the socle $H$, then the socle $L$ of $W$ is $K^{l}$.

Now $W$ acts naturally on the $l$ factors in $K^{l}$, and we say that a
subgroup $G$ of $W$ is of type III(b) if $L\leq G$, $G$ acts transitively on
these $l$ factors, and one of the following holds:

\begin{enumerate}
\item[(i)] $H$ is of type II, $K=T$, $h=l$, and $L$ is the unique minimal normal
subgroup $G$; further $\Lambda $ is a $G$ invariant partition of $\mathcal{P}
$ and, for $x$ in the part $z\in \Lambda $, $L_{z}=T_{y}^{h}<L$ an for some
non-trivial normal subgroup $R$ of $T_{y}$, $L_{x}$ is a subdirect product
of $R^{k}$, that is $L_{x}$ projects surjectively on each of the direct factors $%
R$.

\item[(ii)] $H$ is of type III(a), $\mathcal{P}=\Lambda $, $K\cong T^{h/l}$
and both $G$ and $H$ have $m$ minimal normal subgroups where $m\leq 2$ . If $%
m=2$ then each of the two minimal normal subgroups of $G$ is regular on $%
\mathcal{P}$.
\end{enumerate}

\item[III(c).] \emph{Twisted wreath action.} Here $G$ is a twisted wreath
action $T\wr _{\phi }P$ defined as follows. Let $P$ have a transitive action
on $\left\{ 1,...,h\right\} $ and let $Q$ be the stabilizer $P_{1}$ of
the point $1$ in this action. We suppose that there is an homomorphism $\phi
:Q\rightarrow \mathrm{Aut}(T)$ such that $\mathrm{core}_{P}\left( \phi
^{-1}(\mathrm{Inn}(T))\right) =\cap _{x\in P}\phi ^{-1}(\mathrm{Inn}%
(T))^{x}=\left\{ 1\right\} $. Define%
\begin{equation*}
L=\left\{ f:P\rightarrow T:f(\alpha \beta )=f(\alpha )^{\phi (\beta )}%
\text{ for all }\alpha \in P\text{, }\beta \in Q\right\} \text{.}
\end{equation*}%
Then $B$ is a group under the pointwise multiplication, and $L\cong T^{l}$.
Let $P$ act on $L$ by%
\begin{equation*}
f^{ }(\gamma )=f(\alpha \gamma )\text{ for }\alpha ,\gamma \in P\text{.%
}
\end{equation*}%
Define $T\wr _{\phi }P$ to be the semidirect product of $L$ by $P$ with this
action, and define the action on $\mathcal{P}$ by setting $G_{x}=P$. Then $%
\left\vert \mathcal{P}\right\vert =\left\vert T\right\vert ^{h}$, and $L$ is
the unique minimal normal subgroup of $G$ and acts regularly on $\mathcal{P}$%
. We say that $G$ is of type III(c).
\end{enumerate}

\bigskip

\begin{theorem}\label{ASt1}
If $\mathcal{D}$ is of type 1 then $G$ is almost simple.
\end{theorem}

\begin{proof}
Suppose that $\mathcal{D}$ is of type 1. Case (I) is ruled out since $G$
acts imprimitively on $\mathcal{P}$ by our assumptions. Suppose that $G$ is
of type III(a) or III(c). Then $(\lambda +2)\lambda ^{2}=\left\vert \mathcal{%
P}\right\vert =\left\vert T\right\vert ^{j}$ where $T$ is non-abelian simple and $j=h-1$ or $h$, with $%
h\geq 2$, respectively. Hence, $\lambda $ is even.

If $j=1$, then $%
h=j+1=2$ and hence $G$ acts primitively on $\mathcal{P}$ (see III(a)), which is a
contradiction. Thus $j>1$. Note that $\left\vert
x^{T_{i}}\right\vert =\left\vert T\right\vert $ since $T_{i}$ acts
semiregularly on $\mathcal{P}$, where $T_{i}$ is the subgroup of $L$ consisting of the $h$%
-tuples with $1$ in all but the $i$-th component, $T_{i}\cong T$. Moreover, $x^{T_{i_{1}}}\cap
x^{T_{i_{2}}}=\left\{ x\right\} $ for each $i_{1},i_{2}$ such that $%
i_{1}\neq i_{2}$. Let $w\in x^{T_{1}}\setminus \left\{ x\right\} $, then $w^{G_{x}}$ is the disjoint union of $w^{G_{x}}\cap x^{T_{i}}$ for $%
1\leq i\leq h$. Therefore $\left\vert w^{G_{x}}\right\vert =\left\vert
w^{G_{x}}\cap x^{T_{i}}\right\vert h$ since $G_{x}$ permutes transitively $T_{1},...,T_{h}$. On the other hand $\lambda+1 \mid \left\vert w^{G_{x}}\right\vert $ by Lemma \ref{PP}, hence $\lambda +1\leq \left\vert
w^{G_{x}}\cap x^{T_{i}}\right\vert h\leq \left\vert T\right\vert h$. Thus $(\lambda +2)\lambda ^{2}=\left\vert T\right\vert ^{j}$ implies
\begin{equation*}
\left\vert T\right\vert ^{j}\leq (\left\vert T\right\vert
h-1)^{2}(\left\vert T\right\vert h+1)\leq \left\vert T\right\vert ^{3}h^{3},
\end{equation*}%
and hence $j=2$ or $3$, or $j=4$ and $T \cong A_{5}$, as $j>1$ and $\left\vert T\right\vert \geq 60$. The latter is ruled out since it does not provide integer solutions for $(\lambda +2)\lambda ^{2}=\left\vert T\right\vert ^{4}$. If $j=3$ then $\left\vert T\right\vert
>\lambda >\left\vert T\right\vert -2$ since $(\lambda +2)\lambda
^{2}=\left\vert T\right\vert ^{3}$, and hence $\lambda =\left\vert
T\right\vert -1$, but this contradicts $\lambda$ even. Thus, $j=2$. If $h=3$ then $G$ acts primitively on $\mathcal{P}$. Indeed, $G$ is as in III(a) and $Z_{3} \leq P \leq S_{3}$. Thus $j=h=2$ and $G$ is as in III(c). Moreover, $\lambda +1\mid \left\vert
w^{G_{x}}\cap x^{T_{1}}\right\vert$ for each $w\in x^{T_{1}}\setminus \left\{ x\right\} $ since $\lambda$ is even. Then there is $\theta \geq 1$ such that $\left\vert T\right\vert =\theta
(\lambda +1)+1$ since $\left\vert x^{T_{1}}\right\vert=\left\vert T \right \vert$. Then $\lambda ^{2}(\lambda +2)=\left\vert T\right\vert ^{2}$
implies 
\begin{equation}
\left( \lambda +1\right) \left( \left( \lambda +1\right) \lambda -1\right)
=\theta ^{2}(\lambda +1)^{2}+2\theta (\lambda +1)  \label{theta}
\end{equation}%
and hence $\lambda (\lambda +1)-1=\theta ^{2}(\lambda +1)+2\theta $. Then 
\[
\theta =\frac{(\lambda +1)t-1}{2}
\]%
for some $t\geq 1$. If $t\geq 2$ then $\theta >\lambda $ and we reach a
contradiction, thus $t=1$ and hence $\theta =\lambda /2$ which substituted in
(\ref{theta}) yields a contradiction too. Therefore, $G$ is not of type III(a) or III(c).

Suppose that $G$ is of type III(b.ii). Then $G_{z}\leq W_{z}=H_{y}\wr S_{l}$
where $z=(y,y,...,y)$, and denoted by $\mu $ the number of $W_{z}$-orbits on $%
\Delta $, we see that $\mu \geq 4$ by \cite[Corollary 1.9]{DGLPP}.

Let $y_{1}\in \Delta $, $y_{1}\neq y$ such that $\left\vert y_{1}^{H}\right\vert
\leq \frac{\left\vert \Delta \right\vert -1}{\mu -1}$, and let $z_{1}=(y_{1},y,\dots,y)$. Since $\left(H_{y,y_{1}}\times H_{y}\times \cdots \times H_{y} \right):S_{l-1}\leq
W_{z,z_{1}}$, it results 
\begin{equation*}
\left\vert z_{1}^{G_{z}}\right\vert \leq \left\vert z_{1}^{W_{z}}\right\vert
\leq \frac{\left\vert H_{y}\right\vert ^{l}\left( l!\right) }{\left\vert
H_{y,y_{1}}\right\vert \left\vert H_{y}\right\vert ^{l-1}\left(
(l-1)!\right) }=l\left\vert y_{1}^{H}\right\vert \leq \frac{l\left(
\left\vert \Delta \right\vert -1\right) }{\mu -1}\text{,}
\end{equation*}%
and since $\lambda +1\mid \left\vert z_{1}^{G_{z}}\right\vert $ by Lemma \ref{PP}, it follows
that%
\begin{equation*}
\lambda +1\leq \frac{l\left( \left\vert \Delta \right\vert -1\right) }{\mu -1%
}\text{.}
\end{equation*}%
Then $\lambda $ is even since $(\lambda +2)\lambda ^{2}=\left\vert \mathcal{P}\right\vert =\left(
\left\vert T\right\vert ^{h/l-1}\right) ^{l}=\left\vert T\right\vert ^{h-l}$, $l>1$. Thus
\begin{equation*}
\left\vert T\right\vert ^{l(h/l-1)/3}-1\leq \frac{l\left( \left\vert T\right\vert
^{h/l-1}-1\right) }{\mu -1}< \frac{h\left( \left\vert T\right\vert
^{h/l-1}-1\right) }{3}\text{.}
\end{equation*}%
and hence $l=2,3$, as $l>1$. If $l=3$ then $\left\vert T\right\vert ^{h/3-1}>\lambda >\left\vert
T\right\vert ^{h/3-1}-2$ and hence $\lambda =\left\vert T\right\vert ^{h/3-1}-1$ with $h>l$, whereas $\lambda$ is even. Thus $l=2$, and since $\lambda$ is even and $ \left(\Delta\setminus \{y\} \times \{y \} \right) \cup \left( \{ y\} \times \Delta\setminus \{y\} \right)$ is union of some non-trivial $G_{z}$-orbits, being $G_{z} \leq H_{z}\wr S_{l}$, it follows from Lemma \ref{PP} that $ \lambda+1 \mid \left\vert T\right\vert ^{h/2-1}-1$. Now, we may apply the final argument used to rule out III(a) with $\left\vert T\right\vert ^{h/2-1}$ in the role of $\left\vert T\right\vert$ to exclude this case as well.

Finally, assume that $G$ is of type III(b.i). Then $h=l$ and $T^{h}$ is the
unique minimal normal subgroup of $G$. In this case $\Sigma$ can be identified with the Cartesian product $\Gamma^{h}$, hence each $\Delta \in \Sigma$ corresponds to a unique $h$-tuple of elements of $\Gamma$. Therefore, it results that $\left\vert
\Sigma\right\vert=\lambda ^{2}=\left\vert
\Gamma\right\vert ^{h}$. Let $y\in \Gamma $ and $\Delta=(y,...,y)\in \Sigma $. Then
$$ \bigcup_{i=1}^{h} \left( \{y\}\times \cdots \times \Gamma_{i}\setminus \{y\} \times \cdots \times \{y\}\right) $$
is a union of some non-trivial $G_{\Delta}$-orbits and ultimately of some non-trivial $G_{x}$-orbits, where $x \in \Delta$. Then $\lambda +1\mid (\lambda+2)h\left( \left\vert \Gamma\right\vert -1\right)$ by Lemma \ref{PP}. Thus $\lambda
+1\mid (\lambda +2)h\left( \lambda ^{2/h}-1\right) $, and hence $\lambda
+1\mid h\left( \lambda ^{2/h}-1\right) $.

If $h=2$ then $\lambda +1\mid 4$, whereas $\lambda >10$. Thus $%
h\geq 3$ and hence $\lambda ^{1/3}\leq \lambda ^{1-2/h}\leq h$. Therefore $%
5^{h/3}\leq \left\vert T:T_{\Delta }\right\vert ^{h/3}\leq h^{2}$ since $T$ is a non-abelian simple group and $\left\vert T:T_{\Delta }\right\vert=\lambda^{2}$, and so $%
h\leq 7$. Actually, $\left\vert T:T_{\Delta }\right\vert ^{h/2(1-2/h)}\leq h$
implies $h=3$, and hence $\lambda +1\mid 3\left( \lambda
^{2/3}-1\right) $ which is impossible since $\lambda>10$. This completes the proof.
\end{proof}

\bigskip

\begin{theorem}\label{ASt2}
If $\mathcal{D}$ is of type (2) then $G$ is almost simple.
\end{theorem}

\begin{proof}
If $\mathcal{D}$ is of type (2) then $v \not \equiv 0 \pmod{4}$ (see Theorem \ref{PZM}). Case (I) is ruled out since $G$ acts imprimitively on $\mathcal{P}$
by our assumptions. Also $v\neq $ $\left\vert T\right\vert ^{j}$ with $%
j\geq 1$, where $T$ is non-abelian simple, and hence $G$ is not of type
III(a) or III(c). Assume that $G$ is of type III(b.i). Arguing as in Theorem \ref{ASt1} we see that
\begin{equation*}
\frac{\lambda ^{2}-2\lambda +2}{2}=\left\vert \Sigma \right\vert
=\left\vert \Gamma \right\vert ^{h}\text{ and } \lambda +1\mid (\lambda /2+1)h\left( \left\vert
\Gamma \right\vert -1\right) 
\end{equation*}%
where $h\geq 2$. It follows that $\lambda +1\mid h\left(
\left\vert \Gamma \right\vert -1\right) $. Since $\left\vert \Gamma
\right\vert ^{l}-1=\allowbreak \frac{1}{2}\lambda \left( \lambda -2\right) $%
, it follows $\lambda +1\mid 3h$ and hence $2\cdot 60^{h}\leq 2\left\vert
\Gamma \right\vert ^{h}<9h^{2}$, which is impossible for $h\geq 2$.

Assume that $G$ is of type III(b.ii). Then 
\begin{equation}
\left( \frac{\lambda +2}{2}\right) \left( \frac{\lambda ^{2}-2\lambda +2}{2}%
\right) =\left\vert \mathcal{P} \right\vert =\left\vert \Gamma \right\vert ^{l}%
\text{,}  \label{ded}
\end{equation}%
where $\left\vert \Gamma \right\vert =\left\vert T\right\vert ^{k/l-1}$
with $k/l\geq 2$, and from (\ref{ded}) we derive that $v \equiv 0 \pmod{4}$, a
contradiction. This completes the proof.
\end{proof}

\bigskip
\begin{proof}[Proof of Theorem \ref{ASPQP}]
The assertion immediately follows from Theorems \ref{AQP}, \ref{ASt1} and \ref{ASt2}.
\end{proof}

\bigskip

\section{Reduction to the case $\lambda \leq 10$}

Let $L$ be the preimage in $G$ of $\mathrm{Soc}(G^{\Sigma })$. Hence $%
L^{\Sigma }\trianglelefteq G^{\Sigma }\leq \mathrm{Aut}(L^{\Sigma })$ with $L^{\Sigma }$ non-abelian simple by Theorem \ref{ASPQP}.
Moreover, $G=G^{\Sigma }$ and $L=L^{\Sigma }$ when $G(\Sigma)=1$. The first part of this section is devoted to prove that either $L_{\Delta
}^{\Delta }$ acts $2$-transitively on $\Delta $, where $\Delta \in \Sigma $, or $G(\Sigma)=1$, $Soc(G_{\Delta }^{\Delta })< L_{\Delta }^{\Delta } \leq G_{\Delta }^{\Delta } \leq A \Gamma L_{1}(u^{h})$, where $u^{h}=\left\vert \Delta \right\vert$. Then we prove that either $\left\vert L^{\Sigma }\right\vert \leq \left\vert L_{\Delta
}^{\Sigma }\right\vert ^{2}$ or $\left\vert L\right\vert \leq 4\left\vert L_{\Delta }^{\Delta
}\right\vert^{2} \left\vert \mathrm{Out}(L)\right\vert^{2}$ respectively.
As we will see, these constraints on $L_{\Delta }^{\Sigma }$ are combined with \cite{AB} and \cite{LS} allows us to to completely classify $\mathcal{D}$. Finally, the analysis of $2$-designs $\mathcal{D}$ of type 1 and 2 is carried out in separate subsections.

\bigskip

In the sequel the minimal primitive permutation representation of a non-abelian simple
group $\Gamma $ will be denoted by $P(\Gamma )$. It is known that $P(A_{\ell })=\ell \geq 5$, whereas $P(\Gamma )$ is determined in \cite{At}, \cite[Theorem 5.2.2]{KL}
and in \cite{Va1, Va2, Va3} according to whether $\Gamma $ is sporadic,
classical or exceptional of Lie type respectively. The following technical lemma is
useful to show that $L_{\Delta }^{\Delta }$ acts $2$-transitively on $\Delta 
$ provided that $G_{\Delta}^{\Delta}$ is not a semilinear $1$-dimensional group.

\bigskip

\begin{lemma}
\label{MPPRG}If $\Gamma $ is a non-abelian simple group non-isomorphic to 
$PSL_{2}(q)$ and such that $P(\Gamma )<2\left( \left\vert \mathrm{Out}(\Gamma
)\right\vert +1\right) \left\vert \mathrm{Out}(\Gamma )\right\vert $, then
one of the following holds:

\begin{enumerate}
\item $\Gamma \cong A_{\ell }$ with $\ell =7,8,9,10,11$;

\item $\Gamma \cong M_{11}$;

\item $\Gamma \cong PSL_{3}(q)$ with $q=4,7,8,16$;

\item $\Gamma \cong PSp_{4}(3)$;

\item $\Gamma \cong PSU_{n}(q)$ with $(n,q)=(3,5),(3,8),(4,3)$;

\item $\Gamma \cong P\Omega_{8}^{+}(3)$.
\end{enumerate}
\end{lemma}

\begin{proof}
Assume that $\Gamma \cong A_{\ell }$ with $\ell \geq 5$. Then $\ell \neq 5,6$
since $A_{5} \cong PSL_{2}(5)$ and $A_{6} \cong PSL_{2}(9)$. Then $7 \leq \ell
=P(\Gamma )<12$ and we obtain (1).

If $\Gamma $ is sporadic then $P(\Gamma )<12$, and hence $\Gamma \cong
M_{11} $ by \cite{At}, which is (2).

If $\Gamma $ is a simple exceptional group of Lie type then $P(\Gamma)$ is provided in \cite{Va1, Va2, Va3}, and it is easy to check that no cases arise.

Finally, suppose that $\Gamma$ is a simple classical group.
Assume that $\Gamma \cong PSL_{n}(q)$, where $q=p^{f}$, $f \geq 1$, and $n \geq 3$. We may also assume that $(n,q) \neq (4,8)$ since $PSL_{4}(2)\cong
A_{8}$ was analyzed above. Then $P(\Gamma )=\frac{q^{n}-1}{q-1}$ by \cite[Theorem 5.2.2]{KL} since $n \geq 3$, and hence   
\begin{equation}
\frac{q^{n}-1}{q-1}<2(2(n,q-1)f+1)\cdot 2(n,q-1)f\text{.}  \label{prvaeq}
\end{equation}
Assume that $p^{f}\geq 2f(f+1)$. Then $\frac{q^{n}-1}{q-1}<4q(q-1)^{2}$ and hence $n\leq 4$. Actually, $(n,q)=(3,7)$ by (\ref{prvaeq}) since $n \geq 3$.

Assume that $p^{f}<2f(f+1)$. Then either $p=2$ and $2\leq f\leq 6$, or $p=3$
and $f=1,2$. Hence, $n=3$ and $q=4,8,16$ by (\ref{prvaeq}) since $n \geq 3$ and since $PSL_{3}(2) \cong PSL_{2}(7)$ cannot occur. Thus, we get (3).

We analyze the remaining classical groups by proceeding as in the $PSL_{n}(q)$-case. It is a routine exercise to show that one obtains $PSp_{4}(2)^{\prime }$, $%
PSp_{4}(3)$, $PSU_{n}(q)$ with $(n,q)=(3,5),(3,8),(4,2),(4,3)$, and $P\Omega_{8}^{+}(3)$. Since $PSp_{4}(2)^{\prime }\cong PSL_{2}(9)$ this
case is excluded, whereas the remaining ones yield (4), (5) and (6) respectively, bearing in mind that $PSU_{4}(2) \cong PSp_{4}(3)$.
\end{proof}

\bigskip

\begin{lemma}\label{njami}
Let $x \in \Delta$, if $G(\Sigma)=1$ and $\mathrm{Soc}(
G^{\Delta}_{\Delta}) \cong (Z_{u})^{h}$, where $u$ is an odd prime, then $u^{h} \neq 3^{2},5^{2},7^{2},11^{2}$, $19^{2},23^{2},29^{2},59^{2}$ and $3^{4}$ 
\end{lemma}

\begin{proof}
Suppose the $u^{h} = 3^{2},5^{2},7^{2},11^{2},19^{2},23^{2},29^{2},59^{2}$ or $3^{4}$. Then $\lambda=u^{h}-2$ or $2(u^{h}-1)$ according to whether $\mathcal{D}$ is of type $i=1$ or $2$ respectively. Now, Table \ref{tav0} contains the admissible pairs  $(\lambda, \left\vert \Sigma \right\vert)_{i}$, where $i=1,2$.
\tiny
\begin{table}[h!]
\caption{{\protect\small {Admissible $(\lambda, \left\vert \Sigma \right\vert)_{i}$ for $\mathcal{D}$ of type $i=1,2$ }}}
\label{tav0}\centering
\begin{tabular}{|c|c|c|c|c|c|}
\hline
                              & $3^{2}$      & $5^{2}$       & $7^{2}$      & $11^{2}$ & $19^{2}$ \\
                              
\hline
$(\lambda, \left\vert \Sigma \right\vert)_{1}$&  $(7,7^{2})$ &  $(23,23^{2})$& $(47,47^{2})$& $(7\cdot17,7^{2}\cdot17^{2})$ & $(17,17^{2})$ \\ 
\hline
$(\lambda, \left\vert \Sigma \right\vert)_{2}$ &  $(16,113)$ &  $(48,1105)$ & $(96,4513)$& $(7\cdot 103,53\cdot4877)$ & $(240,13^{4})$\\   
\hline
\hline

                                                & $23^{2}$ & $29^{2}$ & $59^{2}$ &  $3^{4}$ & \\
\hline
 $(\lambda, \left\vert \Sigma \right\vert)_{1}$ & $(17\cdot 31,17^{2}\cdot 31^{2})$& $(3^{3},3^{6})$  & $(3 \cdot 19, 3^{2} \cdot 19^{2})$ & $(79,79^{2})$&\\
 \hline
$(\lambda, \left\vert \Sigma \right\vert)_{2}$   & $(1056,556513)$ & $(41^{2}, 17 \cdot 82913)$& $( 2^{4}\cdot 3\cdot 5 \cdot 29, 101 \cdot 149 \cdot 1609)$ & $(160,12641)$ &\\

\hline
\end{tabular}
\end{table}

\normalsize
From Table \ref{tav0} we deduce that $\left\vert \Sigma \right\vert$ is either a power of a prime or a squarefree number. It is straightforward to check that no admissible cases occur by \cite[Theorem 1]{Gu} and \cite[Theorem 1]{LiSe} respectively since $L$ is almost simple by Theorem \ref{ASPQP}, being $G(\Sigma)=1$, and since $G^{\Delta}_{\Delta}$ must act $2$-transitively on $\Delta$, where $\Delta \in \Sigma$.   
\end{proof}

\bigskip

\begin{proposition}
\label{unapred} Let $\Delta \in \Sigma$, then $Soc(G_{\Delta }^{\Delta })\trianglelefteq L_{\Delta }^{\Delta }$ and $G=G_{x}L$, where $x$ is any point of $\mathcal{D}$. Moreover, one of the following holds:
\begin{enumerate}
\item $G(\Sigma)=1$, $\mathrm{Soc}(G_{\Delta }^{\Delta })< L_{\Delta }^{\Delta } \leq G_{\Delta }^{\Delta } \leq A \Gamma L_{1}(u^{h})$, where $u^{h}=\left\vert \Delta \right\vert$, and $\left \vert G_{\Delta }^{\Delta }: L_{\Delta }^{\Delta } \right \vert  \mid \left \vert \mathrm{Out}(L) \right \vert $.
\item $L_{\Delta }^{\Delta }$ acts $2$-transitively on $\Delta $.
\end{enumerate}
\end{proposition}

\begin{proof}
Since $L_{\Delta }\trianglelefteq G_{\Delta }$ and $G_{\Delta
}^{\Delta }$ acts primitively on $\Delta $, either $L_{\Delta }=L(\Delta )$
or $Soc(G_{\Delta }^{\Delta })\trianglelefteq L_{\Delta }^{\Delta }$ by \cite[Theorem 4.3B]{DM}.
Moreover, $G=G_{\Delta }L$ since $L^{\Sigma}=\mathrm{Soc}(G^{\Sigma})$ and $G^{\Sigma}$ acts primitively on $\Sigma$ by Theorem \ref{ASPQP}, and hence $G/L=G_{\Delta
}L/L\cong G_{\Delta }/L_{\Delta }$. Also, it results that $G/G(\Sigma)L\cong G^{\Sigma
}/L^{\Sigma }\leq \mathrm{Out}(L^{\Sigma })$.
 
Assume that $G(\Sigma) \neq 1$. Since $(G/L)/(G(\Sigma)L/L) \cong G/G(\Sigma)L$ is solvable and since $G(\Sigma)L/L \cong G(\Sigma)/L(\Sigma)$ is a $2$-group by Theorem \ref{ASPQP}(2), it follows that $G/L$ is solvable. Therefore, $G_{\Delta }/L_{\Delta}$, and hence $G^{\Delta}_{\Delta }/L^{\Delta}_{\Delta } \cong G_{\Delta }/G(\Delta)L_{\Delta }$ is solvable. Then $L_{\Delta }^{\Delta }=G_{\Delta }^{\Delta }$ since   $G_{\Delta }^{\Delta } \cong E_{2^{d}}:G_{x}^{\Delta }$, where $G_{x}^{\Delta }$ is either $SL_{d}(2)$, or $G_{x}^{\Delta }\cong A_{7}$ for $d=4$ by Theorem \ref{ASPQP}(2.b), and the assertion (2) follows in this case.

Assume that $G(\Sigma) = 1$. Then $G=G^{\Sigma}$ and $L=L^{\Sigma}$, and hence $G_{\Delta }/L_{\Delta }$ is isomorphic to a subgroup of $\mathrm{Out}(L)$. 

If $L_{\Delta }=L(\Delta )$ then $G_{\Delta }^{\Delta }=G_{\Delta }/G(\Delta )\cong \left(
G_{\Delta }/L_{\Delta }\right) /\left( G(\Delta )/L_{\Delta }\right) $.
Therefore $G_{\Delta }^{\Delta }$ is isomorphic to a quotient group of $%
G_{\Delta }/L_{\Delta }$, and hence $\mathrm{Out}(L)$ contains a subgroup with a quotient isomorphic to $G_{\Delta }^{\Delta }$. Then $\left\vert \Delta
\right\vert \left( \left\vert \Delta \right\vert -1\right) \mid \left\vert 
\mathrm{Out}(L)\right\vert $ since $G_{\Delta }^{\Delta }$ acts $%
2 $-transitively on $\Delta $. Therefore 
\begin{equation}
P(L)\leq \left\vert L:L_{\Delta }\right\vert =\left\vert \Sigma \right\vert \leq 2
\left\vert \Delta \right\vert \left( \left\vert \Delta \right\vert -1\right)
\leq 2\left\vert \mathrm{Out}(L)\right\vert \text{,}  \label{vise}
\end{equation}%
where $P(L)$ is the minimal primitive permutation representation
of $L$. Then $L \cong PSL_{2}(q)$ by Lemma \ref{MPPRG}. If $q>11$ then $q+1\leq
4(2,q-1)\log q$ which has no solutions. So $q=4,5,7,8,9,11$, which are ruled out since their corresponding value of $\left\vert \Delta
\right\vert$ does not fulfill $\left\vert \Delta
\right\vert \left( \left\vert \Delta \right\vert -1\right) \mid \left\vert 
\mathrm{Out}(L)\right\vert $. Thus $L_{\Delta }=L(\Delta )$ is ruled out. Therefore $Soc(G_{\Delta }^{\Delta })\trianglelefteq L_{\Delta }^{\Delta }$, and hence $L$ acts transitively  on $\Delta$. Then $L$ acts point-transitively on $\mathcal{D}$ since $L$ acts transitively on $\Sigma$. Thus, $G=G_{x}L$ where $x$ is any point of $\mathcal{D}$. 

Since the assertion (2) immediately follows if $Soc(G_{\Delta }^{\Delta })$ is non-abelian simple, we may assume that $Soc(G_{\Delta }^{\Delta })$ is an elementary abelian $u$-group for some prime $u$. Suppose that $L_{\Delta }^{\Delta } = Soc(G_{\Delta }^{\Delta })$. Since $%
G/L=G_{\Delta }L/L\cong G_{\Delta }/L_{\Delta }$, $G/L\leq \mathrm{Out}(L)$ and $G_{\Delta }/G(\Delta
)L_{\Delta }\cong G_{\Delta }^{\Delta }/L_{\Delta }^{\Delta }$, it follows
that $G_{\Delta }^{\Delta }/L_{\Delta }^{\Delta }$ is isomorphic to a
quotient subgroup of $\mathrm{Out}(L)$. Thus $\left\vert \Delta \right\vert
-1\mid \left\vert \mathrm{Out}(L)\right\vert $. Moreover, 
\begin{equation}
P(L)\leq \left\vert L:L_{\Delta} \right\vert = \left\vert \Sigma \right\vert <2\left\vert \Delta \right\vert \left(
\left\vert \Delta \right\vert -1\right) \leq 2\left( \left\vert \mathrm{Out}(L)\right\vert +1\right) \left\vert \mathrm{Out}(L)\right\vert  \label{najvise}
\end{equation}

Then either $L\cong PSL_{2}(q)$, or $L$ is as in Lemma \ref%
{MPPRG}. In the latter case $L$ is neither isomorphic to $A_{\ell
} $ with $7\leq \ell \leq 11$, nor to $M_{11}$, $PSp_{4}(3)$ or $PSU_{3}(5)$ since $\left\vert \Delta \right\vert
-1 \leq \left\vert \mathrm{Out}(L)\right\vert $ (we use this weaker constraint rather than $\left\vert \Delta \right\vert
-1\mid \left\vert \mathrm{Out}(L)\right\vert $ in order to apply the argument here to Theorem \ref{Large} as well), where $\left\vert \Delta \right\vert$ is $\lambda +2$ or $\lambda /2+1$ according to whether $\mathcal{D}$ is of type 1 or 2 respectively, and since $\lambda >10$
by our assumption. The unique admissible group among the remaining ones listed in Lemma \ref%
{MPPRG} is $L\cong PSU_{4}(3)$ and $(\left\vert \Delta \right \vert,\left\vert \Sigma \right \vert)=(20,18^{2})$ or $(38,36^{2})$ by \cite{At} and \cite[Tables 8.3--8.4]{BHRD} since $\left\vert \Sigma \right \vert$ is not of the form as in Theorem \ref{PZM}(V). However, both exceptional cases are ruled out since $\left\vert \Delta \right\vert
> \left\vert \mathrm{Out}(PSU_{4}(3))\right\vert $. Thus $L\cong PSL_{2}(q)$ and hence $%
\left\vert \Delta \right\vert \leq (2,q-1)f$. Then $%
f\geq 3$ for $q$ odd and $f\geq 6$ for $q$ even since $\lambda>10$ and since $\left\vert \Delta \right\vert$ is $\lambda +2$ or $%
\lambda /2+1$ according to whether $\mathcal{D}$ is of type 1 or 2 respectively. Thus $P(L)=q+1$, hence (\ref{najvise}) becomes $p^{f}+1\leq 2(2,q-1)f((2,q-1)f+1)\leq
4f(2f+1)$ which yields $q=2^{6},3^{3},3^{4}$ and $\mathcal{D}$ is of type 2
with $\lambda =12,12,16$ respectively. Thus $\left\vert \Sigma \right\vert
=61,61,113$ respectively, but $L\cong PSL_{2}(q)$ with $%
q=2^{6},3^{3},3^{4}$ has no such transitive permutation degrees. Thus $Soc(G_{\Delta }^{\Delta })<L_{\Delta }^{\Delta }$ and $\left \vert G_{\Delta }^{\Delta }: L_{\Delta }^{\Delta } \right \vert  \mid \left \vert \mathrm{Out}(L) \right \vert $.

If $G_{\Delta }^{\Delta } \not \leq A \Gamma L_{1}(u^{h})$, then $G_{\Delta }^{\Delta }$ contains as a normal subgroup one of the groups $SL_{n}(u^{h/n})$, $Sp_{n}(u^{h/n})$, $G_{2}^{\prime}(2^{h/6})$ with $h \equiv 0 \pmod{6}$, $A_{6}$ or $A_{7}$ for $(u,h)=(2,4)$, or $SL_{2}(13)$ for $(u,h)=(3,6)$ by \cite[List(B)]{Ka} since $u^{h} \neq 3^{2},5^{2},7^{2},11^{2}$, $19^{2},23^{2},29^{2},59^{2}$ and $3^{4}$ by Lemma \ref{njami}, and $L_{\Delta }^{\Delta }$ acts $2$-transitively on $\Delta$ in these cases. This completes the proof.
\end{proof}

\bigskip

\begin{corollary}
\label{QuotL} If $G(\Sigma )\neq 1$, then $\lambda =2(2^{d-1}-1)$, $\left\vert \Sigma
\right\vert =\lambda ^{2}$ and a quotient group of $L_{\Delta }^{\Sigma }$
is isomorphic either to $SL_{d}(2)$ for $d\geq 4$, or to$A_{7}$ for $d=4$.
\end{corollary}

\begin{proof}
If $G(\Sigma )\neq 1$, then $G_{\Delta }^{\Sigma }/G(\Delta )^{\Sigma }\cong
G_{x}^{\Delta }$ where either $G_{x}^{\Delta }\cong SL_{d}(2)$ for $d\geq 4$, or $%
G_{x}^{\Delta }\cong A_{7}$ for $d=4$ by Theorem \ref{ASPQP}. On the other hand, $G_{x}/L_{x} \cong G_{x}L/L =G/L$ is isomorphic to a subgroup of $\mathrm{Out}(L)$ as a consequence of Proposition \ref{unapred}. Therefore $%
G_{x}/L_{x}$, and hence $G_{x}/G(\Delta )L_{x}$ is solvable. Since $G_{x}/G(\Delta )L_{x}\cong G_{x}^{\Delta }/L_{x}^{\Delta }$ and  $G_{x}^{\Delta }$ is non-abelian simple, it follows that $%
L_{x}^{\Delta }=G_{x}^{\Delta }$. Thus $L_{\Delta }^{\Sigma } \not \leq G(\Delta )^{\Sigma}$, and hence $L_{\Delta }^{\Sigma } / \left(L_{\Delta }^{\Sigma }  \cap G(\Delta )^{\Sigma}\right) \cong G_{\Delta }^{\Sigma } / G(\Delta )^{\Sigma} \cong G_{x}^{\Delta }$, which is the assertion.
\end{proof}

\bigskip

\begin{theorem}
\label{Large} Let $\Delta \in \Sigma$ then $L_{\Delta }^{\Sigma }$ is a
large subgroup of $L^{\Sigma }$. Moreover one of the following holds:
\begin{enumerate}
\item $\left\vert L^{\Sigma }\right\vert \leq \left\vert L_{\Delta
}^{\Sigma }\right\vert ^{2}$.
\item $G(\Sigma)=1$, $Soc(G_{\Delta }^{\Delta })< L_{\Delta }^{\Delta } \leq G_{\Delta }^{\Delta } \leq A \Gamma L_{1}(u^{h})$, where $u^{h}=\left\vert \Delta \right\vert$. Furthermore, the following holds:
\begin{enumerate}
\item $L_{\Delta}$ does not act $2$-transitively on $\Delta$;
\item $\left\vert L\right\vert \leq 4\left\vert L_{\Delta }^{\Delta
}\right\vert^{2} \left\vert \mathrm{Out}(L)\right\vert^{2}$;
\item $ \left\vert L(\Delta) \right\vert <2 \left\vert \mathrm{Out}(L) \right\vert < \left\vert L_{\Delta} \right\vert$.
\end{enumerate}
\end{enumerate}
\end{theorem}

\begin{proof}
Suppose that $G(\Sigma)=1$. Then $G=G^{\Sigma }$ and $%
L=L^{\Sigma }$. Assume that $L_{\Delta }^{\Delta }$ acts $2$-transitively on $\Delta $. If $\mathcal{D}$ is of type 1 then $\left\vert L:L_{\Delta
}\right\vert =\left\vert \Sigma \right\vert \leq \left\vert \Delta
\right\vert \left( \left\vert \Delta \right\vert -1\right) \leq \left\vert
L_{\Delta }^{\Delta }\right\vert $, and hence $\left\vert L\right\vert \leq \left\vert L_{\Delta
}\right\vert ^{2}$, which is the assertion (1) in this case.

If $\mathcal{D}$ is of type 2, since $G$ acts flag-transitively on $\mathcal{D}$, it follows that $\lambda ^{2}/2\mid \left\vert G_{x}\right\vert $ and
hence $\lambda ^{2}/2\mid \left\vert L_{x}\right\vert \left\vert \mathrm{Out}%
(L)\right\vert $. On the other hand, $\left\vert L_{\Delta }\right\vert
=\left( \frac{\lambda }{2}+1\right) \frac{\lambda }{2}\left\vert
L_{x,y}\right\vert $ since $L_{\Delta }$ induces a $2$-transitive group on $%
\left\vert \Delta \right\vert $. In particular, 
$\left\vert L_{x}:L_{x,y}\right\vert =\frac{\lambda }{2}$ and so $\lambda
\mid \left\vert L_{x,y}\right\vert \left\vert \mathrm{Out}(L)\right\vert $.
If $\lambda \mid \left\vert \mathrm{Out}(L)\right\vert $ then $P(L)\leq
\left\vert L:L_{\Delta }\right\vert =\frac{\lambda ^{2}-2\lambda +2}{2}%
<\left\vert \mathrm{Out}(L)\right\vert ^{2}$. Then $L\cong PSL_{2}(q)$ or $L$
is one of the groups listed in Lemma \ref{MPPRG}. Actually, in the latter
case only $L\cong PSL_{3}(q)$ with $(q,\lambda)=(4,12),(16,12),(16,24)$ are admissible since $\lambda \mid \left\vert \mathrm{Out}%
(L)\right\vert $ and $\lambda >10$. Then $\left\vert \Sigma \right\vert =61$
or $265$ respectively, but none of these divides the order of
the corresponding $L$. So these cases are excluded, and hence $L\cong
PSL_{2}(q)$ with $q=p^{f}$ and $f\geq 6$, as $\lambda \mid \left\vert \mathrm{%
Out}(L)\right\vert $ and $\lambda >10$ and $\left\vert \mathrm{Out}%
(L)\right\vert =(2,p^{f}-1)f$. However, $p^{f}+1\leq \left\vert L:L_{\Delta
}\right\vert \leq 2f^{2}-2f+1$ has no admissible solutions for $f\geq 6$.
Then $\left( \left\vert \mathrm{Out}(L)\right\vert ,\lambda \right) <\lambda 
$ and hence $\left\vert L_{x,y}\right\vert \geq 2$ since $\lambda \mid
\left\vert L_{x,y}\right\vert \left\vert \mathrm{Out}(L)\right\vert $. Then $%
\left\vert L_{\Delta }\right\vert \geq \left( \frac{\lambda }{2}+1\right)
\lambda >\left\vert \Sigma \right\vert =\left\vert L:L_{\Delta }\right\vert $
and we obtain the assertion (1) also in this case.

Assume that $L_{\Delta }^{\Delta }$ does not act $2$-transitively on $\Delta $. Then $Soc(G_{\Delta }^{\Delta })< L_{\Delta }^{\Delta }$ and  $G_{\Delta }^{\Delta }$ is a $2$-transitive subgroup of the semilinear $1$-dimensional group by Proposition \ref{unapred}. Note that $\left\vert L:L_{\Delta }\right\vert =\left\vert G:G_{\Delta
}\right\vert \leq 2\left\vert G_{\Delta }^{\Delta }\right\vert $. Also, it results that $G_{\Delta }/L_{\Delta } \cong G_{\Delta }L/L=G/L\leq \mathrm{Out}(L)$ and $G_{\Delta }/G(\Delta
)L_{\Delta }\cong G_{\Delta }^{\Delta }/L_{\Delta }^{\Delta }$. Hence, $G_{\Delta }^{\Delta }/L_{\Delta }^{\Delta }$ is isomorphic to a quotient group of a subgroup of $\mathrm{Out}(L)$. Therefore, $%
\left\vert L:L_{\Delta }\right\vert \leq 2\left\vert L_{\Delta }^{\Delta
}\right\vert \left\vert \mathrm{Out}(L)\right\vert$.

Assume that $\left \vert L(\Delta) \right \vert \geq 2\left\vert \mathrm{Out}(L)\right\vert$. Then $2\left\vert L_{\Delta }^{\Delta
}\right\vert \left\vert \mathrm{Out}(L)\right\vert \leq \left \vert L_{\Delta} \right \vert$, and we still obtain (1).

Assume that $\left \vert L(\Delta) \right \vert < 2\left\vert \mathrm{Out}(L)\right\vert$. Then $\left\vert L\right\vert \leq 4\left\vert L_{\Delta }^{\Delta
}\right\vert^{2} \left\vert \mathrm{Out}(L)\right\vert^{2}$. Suppose that $\left\vert
L_{\Delta }\right\vert \leq \left\vert \mathrm{Out}(L)\right\vert $. Then $%
2\left\vert \Delta \right\vert \leq \left\vert L_{\Delta }\right\vert \leq
2\left\vert \mathrm{Out}(L)\right\vert $ since $\mathrm{Soc}(G_{\Delta
}^{\Delta })<L_{\Delta }^{\Delta }$. Then $\left\vert \Delta \right\vert
\leq \left\vert \mathrm{Out}(L)\right\vert $, and hence $P(L) \leq \left\vert
L:L_{\Delta }\right\vert \leq 2\left\vert \mathrm{Out}(L)\right\vert
(\left\vert \mathrm{Out}(L)\right\vert -1)$. Then $L_{\Delta }$ is isomorphic either to $PSL_{2}(q)$ or to one of the groups listed in Lemma \ref{MPPRG}. However, the same argument used in Proposition \ref{unapred} can be also used here to rule out all these groups since $\left\vert \Delta \right\vert
\leq \left\vert \mathrm{Out}(L)\right\vert $. Thus $\left\vert L_{\Delta }\right\vert >2\left\vert \mathrm{Out}(L)\right\vert $ and hence $%
\left\vert L:L_{\Delta }\right\vert <\left\vert L_{\Delta }\right\vert ^{2}$, which means that $L_{\Delta }$ is a large subgroup of $L$, and we obtain assertions (2a)--(2c).

Suppose that $G(\Sigma) \neq 1$. Then $\lambda =2(2^{d-1}-1)$, $\left\vert \Sigma
\right\vert =\lambda ^{2}$ and  a quotient group of $L_{\Delta }^{\Sigma}$
is isomorphic either to $SL_{d}(2)$ for $d\geq 4$, or to $A_{7}$ for $d=4$ by Corollary \ref{QuotL}. In either case one has $\left\vert \Delta
\right\vert \left( \left\vert \Delta \right\vert -1\right) \leq \left\vert
L_{\Delta }^{\Sigma }\right\vert $ since $\left\vert \Delta \right\vert
=2^{d}$. Since $G(\Sigma )\trianglelefteq L_{\Delta }<L$ and since $%
G^{\Sigma }$ acts primitively on $\Sigma $, it follows that $\left\vert
L^{\Sigma }:L_{\Delta }^{\Sigma }\right\vert =\left\vert L:L_{\Delta
}\right\vert =\left\vert \Sigma \right\vert \leq \left\vert \Delta
\right\vert \left( \left\vert \Delta \right\vert -1\right) \leq \left\vert
L_{\Delta }^{\Sigma }\right\vert $ and the assertion (1) follows in this
case.

\end{proof}

\bigskip

\section{Classification of the $2$-designs of type 1}

In this section we mainly use the constraints for $L^{\Sigma}$ provided in Proposition \ref{unapred}, Corollary \ref{QuotL}  and Theorem \ref{Large} to prove Theorem \ref{T1} stated below. It is worth noting that, when $L^{\Sigma}$ is a Lie type simple group we show that $L^{\Sigma}_{\Delta}$ is a large subgroup of $L^{\Sigma}$ of order divisible by a suitable primitive prime divisor of $p^{\zeta}-1$, where $\zeta$ is determined in \cite[Proposition 5.2.16]{KL}. We combine this constraints on $L^{\Sigma}_{\Delta}$ to show that a small number of groups listed in \cite{AB} are admissible. These groups are then ruled out by exploiting the $2$-transitivity of $G_{\Delta}^{\Delta}$ on $\Delta$.  

\begin{theorem}
\label{T1}If $\mathcal{D}$ is a symmetric $2$-$\left( \left( \lambda
+2\right) \lambda ^{2},(\lambda +1)\lambda ,\lambda \right) $ design admitting a flag-transitive and
point-imprimitive automorphism group, then $\lambda \leq 10$.
\end{theorem}

\bigskip

We analyze the cases where $L^{\Sigma }$ is sporadic, alternating, a Lie type simple classical or exceptional group separately.  
\bigskip

\begin{lemma}
\label{spor1}$L^{\Sigma }$ is not sporadic.
\end{lemma}

\begin{proof}
Either $L_{\Delta }^{\Sigma }$ is maximal in $L^{\Sigma }$, or $%
\mathrm{Out}(L)\cong Z_{2}$ and $G_{\Delta }^{\Sigma }$ is a novelty. In the
latter case, the unique admissible case is $G^{\Sigma }\cong M_{12}:Z_{2}$ and $G_{\Delta
}^{\Sigma }\cong PGL_{2}(11)$ since $\left\vert \Sigma \right\vert $ is a square. Hence, $\left\vert \Sigma \right\vert =144$
by \cite[Table 1]{Wi1} (see also \cite{Wi2} for the cases $L^{\Sigma }\cong
Fi_{22}$ or $Fi_{24}^{\prime }$). Then $\lambda +2=14$ and hence $G(\Sigma
)=1$ by Theorem \ref{ASPQP}. Since $G_{\Delta }\cong PGL_{2}(11)$ has
no transitive permutation representations of degree $14$, the case is ruled out by
Proposition \ref{unapred}. Therefore, $%
L_{\Delta }$ is maximal in $L$.

Assume that $L^{\Sigma }\cong M_{i}$, where $i=11,12,22,23$ or $24$. Since $%
\left\vert L^{\Sigma }:L_{\Delta }^{\Sigma }\right\vert =\lambda ^{2}$, it
follows from \cite[Table 5.1.C]{KL} that, $k^{2}=2^{a_{1}}3^{a_{2}}$ for
some $a_{1},a_{2}\geq 2$. Then $\lambda =12$ and either $L\cong M_{11}$ and $%
L_{\Delta }\cong F_{55}$, or $L^{\Sigma }\cong M_{12}$ and $L^{\Sigma }\cong
PSL_{2}(11)$ by \cite{At}. However, these cases are ruled out by Proposition %
\ref{unapred} since $\lambda +2$ does not divide the order of $L_{\Delta }$.

Assume that $L^{\Sigma }\cong J_{i}$, where $i=1,2,3$ or $4$. Then $\lambda
^{2}$ divides $2^{2}$, $2^{6}3^{2}5^{2}$, $2^{6}3^{4}$, or $%
2^{20}3^{2}11^{2} $, respectively, by \cite[Table 5.1.C]{KL}. Then $i=2$ and
either $\lambda =10$ and $L_{\Delta }^{\Sigma }\cong PSU_{3}(3)$, or $%
\lambda =60$ and $L_{\Delta }^{\Sigma }\cong PSL_{2}(7)$ by \cite{At}.
However these cases are ruled out, as $L_{\Delta }^{\Sigma }$ does not have
a $2$-transitive permutation representation of degree $12$ or $62$
respectively, and this contradicts Proposition \ref{unapred}.

Assume that $L^{\Sigma }$ is isomorphic to one of the groups $HS$ or $McL$.
By \cite[Table 5.1.C]{KL} $\lambda ^{2}$ divides $2^{8}3^{2}5^{2}$ or $%
2^{6}3^{6}5^{2}$ respectively. Then either $L^{\Sigma }\cong HS$, $L_{\Delta
}^{\Sigma }\cong M_{22}$ and $\lambda=10$, or $L^{\Sigma }\cong McL$, $L_{\Delta
}^{\Sigma }\cong M_{22}$ and $\lambda=45$. Both these cases are excluded since they contradict
Proposition \ref{unapred}.

It is straightforward to check that the remaining cases are ruled out
similarly, as they do not have transitive permutation representations of
degree $\lambda ^{2}$ by \cite{At} and \cite{Wi2}.
\end{proof}

\bigskip

\begin{lemma}
\label{AltLM}If $L^{\Sigma }\cong A_{\ell }$, $\ell \geq 5$, then $L^{\Sigma
}$ acts primitively on $\Sigma $.
\end{lemma}

\begin{proof}
Assume that $L \cong A_{6}$. Then $A_{6} \cong PSL_{2}(9) \trianglelefteq G \leq P\Gamma L_{2}(9)$, hence $G$ does not have a primitive permutation representation of degree $\lambda^{2}$ with $\lambda>10$ by \cite{At}. Thus $\ell \neq 6$ by Theorem \ref{ASPQP}, and hence $\mathrm{Out%
}(L)\cong Z_{2}$.

Suppose the contrary of the statement. Then there is a subgroup $M$ of $L$
containing $L_{\Delta }$ such that $L_{\Delta }^{\Sigma }<M^{\Sigma
}<L^{\Sigma }$ with $M^{\Sigma }$ maximal in $L^{\Sigma }$. Let $x\in \Delta 
$, then $x^{M}$ is a union of $\theta$ elements of $\Sigma $, where $\theta = \left\vert M^{\Sigma }: L^{\Sigma }_{\Delta} \right\vert$. Therefore $\left\vert
x^{M}\right\vert =\theta (\lambda +2)$ with $\lambda ^{2}=a\theta $ for some 
$a\geq 1$. Then $x^{M}\setminus \left\{ x\right\} $ is a union of $%
L_{x} $-orbits since $L_{x}<L_{\Delta }<L$. Therefore $\frac{\lambda +1}{%
\eta }\mid \theta (\lambda +2)-1$ by Lemma \ref{PP}, where $\eta =(\lambda +1,2)$ since $%
\mathrm{Out}(L)\cong Z_{2}$. Then $\theta =f\frac{\lambda +1}{\eta }+1$ for
some $f\geq 1 $, hence 
\begin{equation}
\left( f\frac{\lambda +1}{\eta }+1\right) a-1=\lambda ^{2}-1  \label{en}
\end{equation}%
and so $a=e\frac{\lambda +1}{\eta }+1$ for some $e\geq 1$. Then (\ref{en})
becomes%
\begin{equation*}
ef\left( \frac{\lambda +1}{\eta }\right) ^{2}+(e+f)\left( \frac{\lambda +1}{%
\eta }\right) +1=\lambda ^{2}
\end{equation*}%
Then $ef<\eta ^{2}=(\lambda +1,2)^{2}$, and hence $\eta =2$ and $\lambda $ is odd. Therefore $%
ef=3$ and $\lambda =15$, $\left\vert \Delta \right \vert=17$ and $\left\vert\Sigma \right \vert=225$ since $\lambda>10$. Moreover, $G(\Sigma )=1$ by
Theorem \ref{ASPQP}. Then $L_{\Delta }\cong A_{224}$ by \cite[%
Table B.4]{DM}. However, this case cannot occur by Proposition \ref%
{unapred} since $A_{224}$ has no quotient groups with a transitive permutation
representations of degree $17$.
\end{proof}

\begin{lemma}
\label{Alt1}$L^{\Sigma }$ is not isomorphic to $A_{\ell }$, $\ell \geq 5$.
\end{lemma}

\begin{proof}
Since $L_{\Delta }^{\Sigma }$ is a large maximal subgroup of $L^{\Sigma }$
by Theorem \ref{Large} and Lemma \ref{AltLM}, and since $\left\vert \Sigma \right\vert =\lambda^{2}$ with $\lambda>10$, only the following cases are admissible by \cite[Theorem 2]{AB}:

\begin{enumerate}
\item[(i)] $L_{\Delta }^{\Sigma }\cong \left( S_{t}\times S_{\ell -t}\right)
\cap A_{\ell }$ with $1\leq t\leq \ell /2$;

\item[(ii)] $L_{\Delta }^{\Sigma }\cong \left( S_{t}\wr S_{\ell /t}\right)
\cap A_{\ell }$ with $2\leq t\leq \ell /2$.
\end{enumerate}

Suppose that (i) holds. Then $\binom{\ell }{t}=\lambda ^{2}$, and hence
either $t=1,2$, or $t=3$ and $\ell =50$ by \cite[Chapter 3]{AZ}.

Assume that $t=1$. Since $L_{\Delta }^{\Sigma }\cong A_{\lambda ^{2}-1}$ with $\lambda>10$, no quotient groups of $%
L_{\Delta }^{\Sigma }$ are isomorphic to $SL_{d}(2)$ for any $d \geq 4$. Also, the minimal transitive permutation representation of $%
A_{\lambda ^{2}-1}$ is strictly greater than $\lambda+2$. Thus, this case cannot occur by Proposition \ref{unapred} and Corollary \ref{QuotL}.

Assume that $t=2$. Then $L_{\Delta }^{\Sigma }\cong \left( S_{2}\times
S_{\ell -2}\right) \cap A_{\ell }$. Suppose that $G(\Sigma )\neq 1$. Then a
quotient group of $G_{\Delta }^{\Sigma }$ is isomorphic to $SL_{d}(2)$ with $d>4$ by
Theorem \ref{ASPQP} since $\binom{\ell }{2}\neq 14 ^{2}$, but this is clearly impossible. Thus $G(\Sigma
)=1$ and hence a quotient group of $L_{\Delta }\cong \left( S_{2}\times
S_{\ell -2}\right) \cap A_{\ell }$ must have a $2$-transitive permutation
representation of degree $\ell-2=\lambda +2$ by Proposition \ref{unapred} since $\lambda >10$. However, $\binom{\lambda +4 }{2}=\lambda ^{2}$ has no integer solutions. Thus, this case is
excluded.

Assume $t=3$ and $\ell =50$. Then $\lambda +2=142$, as $\lambda>10$, and hence $G(\Sigma )=1$
by Theorem \ref{ASPQP}. Then $\lambda +2\mid \left\vert G_{\Delta
}\right\vert $, whereas $G_{\Delta }=\left(
S_{3}\times S_{47}\right) \cap G$, which is a contradiction.

Suppose that (ii) holds. Hence, $L_{\Delta }^{\Sigma }\cong \left( S_{\ell /t}\wr
S_{t}\right) \cap A_{\ell }$ where $s/t,t>1$. Then $\left\vert \Sigma
\right\vert =\frac{\ell !}{((\ell /t)!)^{t}t!}$ and $A_{\ell /t}\wr
A_{t}\leq L_{\Delta }^{\Sigma }\leq \left( S_{\ell /t}\wr S_{t}\right) \cap
L $. Easy computations show that $\ell >25$ since $\left \vert \Sigma \right \vert =\lambda^{2}$. Moreover, $\left( A_{\ell /t}\right) ^{t}\trianglelefteq L_{\Delta
}^{\Sigma }$ and $A_{t} \leq L_{\Delta }^{\Sigma }/\left( A_{\ell /t}\right) ^{t}\leq
(Z_{2})^{t}:S_{t}$, where $(Z_{2})^{t}$ is a permutation module for $A_{t}$. Thus, by \cite[Lemma 5.3.4]{KL}, $L_{\Delta }^{\Sigma
}/\left( A_{\ell /t}\right) ^{t}$ is isomorphic to one of the following groups:
\begin{enumerate}
\item $A_{t},S_{t}$;
\item $A_{t}\times Z_{2},S_{t} \times Z_{2}$;
\item $(Z_{2})^{t-1}:A_{t},(Z_{2})^{t-1}:S_{t}$;
\item $(Z_{2})^{t}:A_{t},(Z_{2})^{t}:S_{t}$.
\end{enumerate}

Suppose that $G(\Sigma )\neq 1$. Then a quotient group of $L_{\Delta }^{\Sigma }$
is isomorphic either to $SL_{d}(2)$ for $d\geq 4$, or to $A_{7}$ for $d=4$ by Corollary \ref{ASPQP}. Matching such information with (1)--(4) one obtains $d=4$, $t=7,8$ and $\lambda =14$. Then $((\ell /t)!)^{t-1}<\frac{\ell !}{((\ell
/t)!)^{t}t!}=196$ as shown in \cite[(3.4)]{MF} and hence $\ell =14,16$. However, $\left\vert \Sigma \right\vert$ is not a square for such values of $t$ and $\ell$.

Suppose that $G(\Sigma )=1$. Then $L_{\Delta }/\left( A_{\ell /t}\right) ^{t}$ is one of the groups in (1)--(4). Assume that $L_{\Delta }^{\Delta }$ does not act $2$-transitively on $\Delta $. Then  $\mathrm{Soc}(G_{\Delta }^{\Delta })< L_{\Delta }^{\Delta } \leq G_{\Delta }^{\Delta } \leq A \Gamma L_{1}(2^{h})$, where $\lambda+2=2^{h}$ and $\left \vert G_{\Delta }^{\Delta }: L_{\Delta }^{\Delta } \right \vert \leq 2$ by Proposition \ref{unapred}, as $\ell>6$, which force $L_{\Delta }^{\Delta }$ to act $2$-transitively on $\Delta $, and we reach a contradiction. Therefore, $L_{\Delta }^{\Delta }$ acts $2$-transitively on $\Delta $.

If $L_{\Delta }^{\Delta }$ is of affine type then $\lambda+2 =2^{i}$ with $i \leq t$. Therefore $((\ell /t)!)^{t-1}<\frac{\ell !}{((\ell/t)!)^{t}t!}<2^{2t}$ and hence $\ell/t=2$, and $t>12$ since $\ell>25$. However, this is impossible by \cite[List (B)]{Ka} since $A_{t} \leq L_{\Delta }/\left( A_{\ell /t}\right) ^{t}$. Thus $L_{\Delta }^{\Delta }$ is almost simple and hence $%
A_{t}\trianglelefteq L_{\Delta }^{\Delta }\leq S_{t}$ with $t \geq 5$ by (1)--(4). Therefore, $t=\lambda
+2>12$ and $2^{t-1} \leq ((\ell
/t)!)^{t-1}<\frac{\ell !}{((\ell /t)!)^{t}t!}<t^{2}$, which is a contradiction. This completes the proof.
\end{proof}

\bigskip

\begin{lemma}
\label{nMag2}$L^{\Sigma }\ncong PSL_{2}(q)$.
\end{lemma}

\begin{proof}
Assume that $L^{\Sigma }\cong PSL_{2}(q)$, $q=p^f$, $p$ prime and $f \geq 1$. Then $q\neq 7$ since $PSL_{2}(7)$
does not have transitive permutation representations of square degree, and $q\neq 9$ by Lemma \ref{Alt1} since $PSL_{2}(9)\cong A_{6}$.
Moreover, if $q=p$, then $p\mid \left\vert L_{\Delta }^{\Sigma }\right\vert $
since $\left\vert L^{\Sigma }:L_{\Delta }^{\Sigma }\right\vert =\lambda ^{2}$%
. Thus no novelties occur  by \cite[Table 8.1%
]{BHRD}, hence $L_{\Delta }^{\Sigma }$ is maximal in $L^{\Sigma }$. Moreover, no quotients groups of $L_{\Delta }^{\Sigma }$ are isomorphic to $SL_{d}(2)$ with $d\geq 4$ or to $A_{7}$
and $d=4$ for $G(\Sigma ) \neq1$, hence $G(\Sigma )=1$ by Corollary \ref{QuotL}. Thus, $L_{\Delta }$ is maximal in $L$. 

Assume that $L_{\Delta }$ is isomorphic to any of the groups $A_{4},S_{4}$ or $A_{5}$. Then $\lambda+2 \mid \left\vert L_{\Delta} \right 	\vert$, hence $\lambda+2$ is not a power of a prime since $\lambda >10$. Thus $L_{\Delta }$ is non-solvable and acts $2$-transitively on $\Delta$ by Proposition \ref{unapred}, hence $L_{\Delta } \cong A_{5}$ and $\lambda+2=6$, whereas $\lambda>10$. Thus, these groups are ruled out.

Assume that $L_{\Delta }$ is isomorphic to $D_{\frac{q \pm 1}{(2,q-1)}}$. Then $\lambda \leq (q-1)/(2,q-1)$ since $\lambda+2 \mid \left\vert L_{\Delta} \right \vert$, and hence $q(q\mp 1)= \left\vert \Sigma \right \vert \leq \frac{(q-1)^{2}}{(2,q-1)^{2}}$, which has no admissible solutions.

Assume that $PSL_{2}(q^{1/m})\trianglelefteq
L_{\Delta }\leq PGL_{2}(q^{1/m})$. Then $L_{\Delta }$ acts $2$-transitively on $\Delta$ by Proposition \ref{unapred}. Moreover, $\lambda +2=q^{1/m}+1$, since $q^{1/m}> 11$ being $\lambda>10$. Thus $\lambda= q^{1/m}-1$, and hence $L_{\Delta }$ must contain a Sylow $p$-subgroup of $L$, which is a contradiction.

Finally, assume that $L_{\Delta }\cong E_{q}:Z_{\frac{q-1}{(q-1,2)}}$. Then $\lambda
^{2}=q+1$, which has no solutions by \cite[A5.1]{Rib} since $\lambda >10$.
\end{proof}

\bigskip

\begin{lemma}
\label{ciciona}$L^{\Sigma }$ is not isomorphic to one of the groups $%
PSL_{3}(4)$, $PSU_{4}(2)$, $PSL_{6}(2)$, $PSp_{6}(2)$, $P\Omega _{8}^{+}(2)$, $G_{2}(2)^{\prime }$, $^{2}G_{2}(3)^{\prime }$ or $^{2}F_{4}(2)^{\prime }$.
\end{lemma}

\begin{proof}
Assume that $L^{\Sigma }\cong PSL_{3}(4)$. Then $35\mid \left\vert L_{\Delta
}^{\Sigma }\right\vert $ since $\left\vert L^{\Sigma }:L_{\Delta }^{\Sigma
}\right\vert =\lambda ^{2}$, but $L^{\Sigma }$ does not contain such a group
by \cite{At}.

Assume that $L^{\Sigma }\cong PSU_{4}(2)$. Then $5\mid \left\vert L_{\Delta
}^{\Sigma }\right\vert $ since $\left\vert L^{\Sigma }:L_{\Delta }^{\Sigma
}\right\vert =\lambda ^{2}$, and hence $L_{\Delta }^{\Sigma }\cong S_{6}$
and $\lambda =6$ by \cite{At}, whereas $\lambda >10$ by our assumptions.

Assume that $L^{\Sigma }\cong PSL_{6}(2)$. Thus $2\cdot 5\cdot 31\mid
\left\vert L_{\Delta }^{\Sigma }\right\vert $ since $\left\vert L^{\Sigma
}:L_{\Delta }^{\Sigma }\right\vert =\lambda ^{2}$, and hence either $%
L_{\Delta }^{\Sigma }\cong SL_{5}(2)$ or $L_{\Delta }^{\Sigma }\cong
E_{2^{5}}:SL_{5}(2)$ by \cite[Tables 8.24 and 8.25]{BHRD}. However, both cases are ruled out since $7$ divides $\left\vert L^{\Sigma
}:L_{\Delta }^{\Sigma }\right\vert$ but $7^{2}$ does not.
 
Assume that $L^{\Sigma }\cong PSp_{6}(2)$. Then $\mathrm{Out}(L^{\Sigma })=1$, $L_{\Delta }^{\Sigma }\cong S_{8}$
and $\lambda =6$ by \cite{At}, whereas $\lambda>10$.

Assume that $L^{\Sigma }\cong P\Omega _{8}^{+}(2)$. Therefore $21\mid
\left\vert L_{\Delta
}^{\Sigma }\right\vert $ since $\left\vert L^{\Sigma }:L_{\Delta
}^{\Sigma }\right\vert =\lambda ^{2}$, and hence $\lambda ^{2}\mid
2^{12}\cdot 3^{4}\cdot 5^{2}$. Then $\left\vert L^{\Sigma }:L_{\Delta
}^{\Sigma }\right\vert $ must be divisible by one among its primitive degrees $120$, $%
135$ or $960$ by \cite{At}. Thus $15^{2}$ divides $\left\vert L^{\Sigma
}:L_{\Delta }^{\Sigma }\right\vert $ in each case, and hence $\lambda=15j$ for some $j \geq 1$. If $G(\Sigma) \neq 1$ then $\lambda=15j=2^{d}-2$ with $d \geq 4$ by Theorem \ref{ASPQP}. Easy computations show that no admissible cases occur. Therefore $G(\Sigma)=1$, and hence $G$ is a subgroup of $P\Omega _{8}^{+}(2).S_{3}$. Moreover, the order of $G$ is divisible by $(15j)^{3}(15j+2)(15j+1)$ since $G_{\Delta}^{\Delta}$ acts $2$-transitively on $\Delta$, which has size $15j+2$, and since $G$ acts flag-transitively on $\mathcal{D}$ and $k=\lambda(\lambda+1)$. Since no groups occur, this case is excluded.     

The case $L^{\Sigma }\cong G_{2}(2)^{\prime }$ is ruled out in Lemma \ref%
{nMag2} since $G_{2}(2)^{\prime }\cong PSL_{2}(8)$. Also, if $L^{\Sigma }\cong $ $%
^{2}G_{2}(3)^{\prime }\cong PSU_{3}(3)$ then $L_{\Delta }^{\Sigma }\cong
PSL_{2}(7)$ and $\left\vert \Sigma \right\vert =36$ by \cite{At}. So $%
\lambda =6$, whereas $\lambda >10$.

Finally, if $L^{\Sigma }\cong $ $%
^{2}F_{4}(2)^{\prime }$ then $L_{\Delta }^{\Sigma }\cong PSL_{3}(3):Z_{2}$
and $\left\vert \Sigma \right\vert =1600$ by \cite{At}. Then $\lambda
=40 $ and hence $\lambda +1=41$ must divide the order of $G$ by Lemma \ref{PP}. However, this is impossible since $G^{\Sigma }\leq $%
$^{2}F_{4}(2)$ and the order of $G(\Sigma)$ is either $1$ or a power of $2$ by Theorem \ref{ASPQP}. \end{proof}

\bigskip

A divisor $s$ of $q^{e}-1$ that is coprime to each $q^{i}-1$ for $i<e$ is
said to be a \emph{primitive divisor}, and we call the largest primitive
divisor $\Phi _{e}^{\ast }(q)$ of $q^{e}-1$ the \emph{primitive part} of $%
q^{e}-1$. One should note that $\Phi _{e}^{\ast }(q)$ is strongly related to
cyclotomy in that it is equal to the quotient of the cyclotomic number $\Phi
_{e}(q)$ and $(n,\Phi _{e}(q))$ when $e>2$. Also $\Phi _{e}^{\ast }(q)>1$
for $e>2$ and $(q,e)\neq (2,6)$ by Zsigmondy's Theorem (for instance, see 
\cite[P1.7]{Rib}).

\bigskip

Let $p$ be a prime, $w$ a prime distinct from $p$, and $m$ an integer which
is not a power of $p$. Also let $\Gamma $ be a group which is not a $p$%
-group. Then we define%
\begin{eqnarray*}
\zeta _{p}(w) &=&\min \left\{ z:z\geq 1\text{ and }p^{z}\equiv 1\pmod{w}%
\right\} \\
\zeta _{p}(m) &=&\max \left\{ \zeta _{p}(w):w\text{ prime, }w\neq p\text{
and }w\mid m\right\} \\
\zeta _{p}(X) &=&\zeta _{p}(\left\vert X\right\vert )\text{.}
\end{eqnarray*}

If $L^{\Sigma }$ is isomorphic to a simple group of Lie type over $GF(q)$, $%
q=p^{f}$, then $\zeta _{p}(L^{\Sigma })$ is listed in \cite[Proposition 5.2.16 and Table
5.2.C]{KL}. In the sequel we will denote $\zeta _{p}(L^{\Sigma })$ simply by $\zeta $. It is worth noting that $\Phi _{\zeta }^{\ast }(p)>1$ by Lemmas \ref{nMag2} and \ref{ciciona}.

\bigskip

\begin{lemma}
\label{Phi0}$L_{\Delta}^{\Sigma }$ is a large subgroup of $L^{\Sigma }$ such that
$\left( \Phi _{\zeta }^{\ast }(p), \left\vert L_{\Delta
}^{\Sigma }\right\vert \right) >1$.
\end{lemma}

\begin{proof}
Suppose $\left( \Phi _{\zeta}^{\ast }(p), \left\vert
L_{\Delta }^{\Sigma }\right\vert \right) =1$. Note that $\left \vert G_{\Delta }^{\Sigma }: L_{\Delta }^{\Sigma } \right \vert \mid \left \vert \mathrm{Out}(L^{\Sigma
}) \right \vert$ since $\left\vert L^{\Sigma }:L_{\Delta }^{\Sigma }\right\vert= \left\vert G^{\Sigma }:G_{\Delta }^{\Sigma }\right\vert$, as $G^{\Sigma}$ acts primitively on $\Sigma$. Thus $\left( \Phi _{\zeta }^{\ast }(p), \left\vert G_{\Delta}^{\Sigma }\right\vert \right) =1$ since $\left( \Phi _{\zeta }^{\ast }(p), \left\vert \mathrm{Out}(L^{\Sigma
})\right\vert \right) =1$ by \cite[Proposition 5.2.15(ii)]{KL}. Therefore $\Phi _{\zeta }^{\ast
}(p)\mid \left\vert G^{\Sigma }:G_{\Delta }^{\Sigma }\right\vert $, and hence $\Phi
_{\zeta }^{\ast }(p)\mid \lambda ^{2}$ since $\left\vert G^{\Sigma
}:G_{\Delta }^{\Sigma }\right\vert =\lambda ^{2}$. Then $(\Phi
_{\zeta }^{\ast }(p),\left\vert G_{x}\right\vert )>1$, where $x$ is any
point of $\mathcal{D}$, since $\lambda (\lambda +1)\mid \left\vert
G_{x}\right\vert $ being $G$ flag-transitive on $\mathcal{D}$. Therefore $%
(\Phi _{\zeta }^{\ast }(p),\left\vert G_{\Delta }\right\vert )>1$, where $%
\Delta $ is the element of $\Sigma $ containing $x$, and hence $(\Phi
_{\zeta }^{\ast }(p),\left\vert G_{\Delta }^{\Sigma }\right\vert )>1$ which is a contradiction. Thus $\left( \Phi _{\zeta}^{\ast }(p), \left\vert L_{\Delta }^{\Sigma }\right\vert \right) >1$, and $L_{\Delta }^{\Sigma }$ is large by Theorem \ref{Large}.
\end{proof}

\begin{lemma}
\label{Exc1}$L^{\Sigma }$ is not isomorphic to an exceptional simple group
of Lie type.
\end{lemma}

\begin{proof}
Let $M$ be a subgroup of $L$ such that $M^{\Sigma }$ is a maximal subgroup
of $L^{\Sigma }$ containing $L_{\Delta }^{\Sigma }$. Also $M^{\Sigma }$ is
large since $L_{\Delta }^{\Sigma }$ is so by Lemma \ref{Large}. Therefore $%
M^{\Sigma }$ is one of the groups classified in \cite[Theorem 5]{AB}.
Moreover, $\left( \Phi _{\zeta }^{\ast }(p), \left\vert M^{\Sigma
}\right\vert \right) >1$ by Lemma \ref{Phi0}.

Assume that $M^{\Sigma }$ is parabolic. If $L^{\Sigma }$ is untwisted then $%
M^{\Sigma }$ can be obtained by deleting the $i$-th node in the Dynkin
diagram of $L^{\Sigma }$, and we see that none of these groups is of order
divisible by a prime factor of $\Phi _{\zeta }^{\ast }(p)$. Indeed, for instance consider the Levi
factors of the  maximal parabolic subgroups of $L^{\Sigma }\cong F_{4}(q)$, $%
q=p^{f}$ are of type $B_{3}(p^{f})$, $C_{3}(p^{f})$ or $A_{1}(p^{f})\times
A_{2}(p^{f})$ and none of these has order divisible by a prime factor of $\Phi
_{12f}^{\ast }(p)$.
If $L^{\Sigma }$ is twisted, that is $L^{\Sigma }$ is
centralized by an automorphism $\gamma $ of the corresponding untwisted
group and $\gamma$ induces a non-trivial symmetry $\rho $ on the Dynkin diagram. In
this case the $M^{\Sigma }$ exist only when deleting the resulting subset
obtained by deleting the $i$-th node in the Dynkin diagram of corresponding
untwisted group is $\rho $-invariant. The Levi factor of $M^{\Sigma }$ is
obtained by taking the fixed points of the automorphism $\gamma $ on the
Levi factor of the corresponding untwisted subgroup. Also in the twisted case the order of any maximal subgroups of $L^{\Sigma }$ is not divisible by a prime factor of $\Phi _{\zeta }^{\ast
}(p)$. Indeed, for instance, the Levi factors of the maximal parabolic
subgroups of $^{2}E_{6}(q)$, $q=p^{f}$, are of types $^{2}A_{5}(q)$, $%
^{2}D_{4}(q)$, $A_{1}(q)\times A_{2}(q^{2})$ and $A_{1}(q^{2})\times
A_{2}(q) $, and none of these has order divisible by a prime factor of $\Phi
_{18f}^{\ast }(p)$.

Assume that $M^{\Sigma }$ is not parabolic. Then $(L^{\Sigma },M^{\Sigma })$
is listed in \cite[Table 2]{AB}. Since $%
\left( \Phi _{\zeta }^{\ast }(p), \left\vert M^{\Sigma }\right\vert
\right) >1$, only the groups contained in Table \ref{tav2} are admissible by \cite{LSS}. 
\begin{table}[h!]
\caption{\small {Admissible $(L^{\Sigma },M^{\Sigma })$}}\label{tav2}
\centering
\begin{tabular}{|l|l|l|}
\hline
$L^{\Sigma }$ & $M^{\Sigma }$ & Conditions \\ \hline \hline
$E_{7}(q)$ & $ (3,q+1).(^{2}E_{6}(q) \times \frac{q+1}{(3,q+1)}).(3,q+1).{2}$ &  \\ \hline
$E_{6}(q)$ & $F_{4}(q)$ &  \\ 
& $(q^{2}+q+1 \times$ $^{3}D_{4}(q)).3$ &  \\ \hline
$F_{4}(q)$ & $^{3}D_{4}(q).3$ &  \\ 
& $^{2}F_{4}(q)$ &  \\ \hline
$G_{2}(q)$ & $SU_{3}(q):2$ &  \\ 
& $^{2}G_{2}(q)$ & $q=3^{2e+1}>1$ \\ 
& $G_{2}(2)$ & $q=5$ \\ 
& $PSL_{2}(13)$ & $q=3,4$ \\ 
& $2^{3}.SL_{3}(2)$ & $q=3$ \\ \hline
$^{2}B_{2}(q)$ & $13:4$ & $q=8$ \\ \hline
$^{3}D_{4}(q)$ & $G_{2}(q)$ &  \\ \hline
\end{tabular}%
\end{table}

$L^{\Sigma }$ is not isomorphic to any groups $G_{2}(3), G_{2}(5)$ and $^{2}B_{2}(8)$ since these do not have a transitive permutation representation of square degree by \cite{At}. Also, if $L^{\Sigma } \cong G_{2}(4)$ then $L_{\Delta}^{\Sigma } \cong PSL_{2}(13)$ again by \cite{At} and hence $\lambda=480$. Then $G(\Sigma)=1$ by Corollary \ref{QuotL}, and hence $PSL_{2}(13) \trianglelefteq G_{\Delta}^{\Delta} \leq PGL_{2}(13)$. However, $\left\vert \Delta \right \vert = \lambda+2=482$ does not divide the order of $G_{\Delta}^{\Delta}$, and we reach a contradiction.  

Suppose that $ \left\vert M^{\Sigma }\right\vert
^{2} < \left\vert L^{\Sigma }\right\vert$. Then $G(\Sigma)=1$, $L_{\Delta }^{\Delta }$ is solvable, $\left\vert
L\right\vert \leq 4 \left\vert \mathrm{Out}(L) \right \vert ^{2} \left\vert L_{\Delta }^{\Delta}\right\vert ^{2}$ and $ \left\vert L(\Delta) \right\vert <2 \left\vert \mathrm{Out}(L) \right \vert$ by Theorem \ref{Large}. Thus $\left\vert M\right\vert^{2} < \left\vert L \right\vert \leq 4 \left\vert \mathrm{Out}(L) \right \vert ^{2} \left\vert L_{\Delta }\right\vert ^{2}$, and hence $\left\vert M:L_{\Delta}\right\vert < 2 \left\vert \mathrm{Out}(L) \right \vert$. In the remaining admissible groups of Table \ref{tav2} the last term of the derived series $M^{(\infty)}$ of $M$ is non-abelian simple. Let $P(M^{(\infty)})$ be the minimal primitive permutation representation of $M^{(\infty)}$. If $M^{(\infty)} \not \leq L_{\Delta}$, then
$$P(M^{(\infty)}) \leq \left\vert M^{(\infty)}:L_{\Delta} \cap M^{(\infty)} \right\vert \leq \left\vert M:L_{\Delta}\right\vert < 2 \left\vert \mathrm{Out}(L) \right \vert,$$
and we reach a contradiction by Lemma \ref{MPPRG}
Therefore $M^{(\infty)} \leq L_{\Delta}$, and hence $M^{(\infty)} \leq L(\Delta)$ since $L_{\Delta }^{\Delta }$ is solvable. Then $P(M^{(\infty)}) \leq \left\vert L(\Delta) \right \vert < 2 \left\vert \mathrm{Out}(L) \right \vert$ and we again reach a contradiction by Lemma \ref{MPPRG}.

Suppose that $\left\vert L^{\Sigma }\right\vert \leq \left\vert M^{\Sigma }\right\vert ^{2}$. Then one of the following holds by \cite{LSS}:

\begin{enumerate}
\item $L^{\Sigma }\cong E_{7}(q)$ and $L_{\Delta }^{\Sigma }=M^{\Sigma
}\cong (3,q+1).(^{2}E_{6}(q)\times (q-1)/(3,q+1)).(3,q+1).2$;

\item $L^{\Sigma }\cong E_{6}(q)$ and $L_{\Delta }^{\Sigma }=M^{\Sigma
}\cong F_{4}(q)$;

\item $L^{\Sigma }\cong F_{4}(q)$ and $L_{\Delta }^{\Sigma }=M^{\Sigma
}\cong $ $^{3}D_{4}(q).Z_{3}$;

\item $L^{\Sigma }\cong G_{2}(q)$ and $L_{\Delta }^{\Sigma }=M^{\Sigma
}\cong SU_{3}(q):Z_{2}$.
\end{enumerate}

Then $G(\Sigma)=1$ by Corollary \ref{QuotL}. If $L_{\Delta }^{\Delta }$ does not act $2$-transitively on $\Delta$, then $L_{\Delta }^{\Delta }$ is solvable by Proposition \ref{unapred}, hence $ \left\vert L(\Delta) \right\vert <2 \left\vert \mathrm{Out}(L) \right \vert$ by Theorem \ref{Large}. However, this is impossible in cases (1)--(4). Then $L_{\Delta }^{\Delta }$ acts $2$-transitively on $\Delta$, and hence only (4) occurs with $\lambda+2
=q^{3}+1$. Then $\left\vert \Sigma \right\vert =(q^{3}-1)^{2}$, whereas $%
L_{\Delta }^{\Sigma }$ is a maximal non-parabolic subgroup of $L^{\Sigma }$.
So this case is excluded, and the proof is completed.
\end{proof}

\bigskip

Now, it remains to analyze the case where $L^{\Sigma}$ is a simple classical group.

\bigskip

\begin{proposition}
\label{Phi}$L_{\Delta}^{\Sigma }$ is a large maximal geometric subgroup of $L^{\Sigma }$. Moreover, it results that
$\left( \Phi _{\zeta }^{\ast }(p), \left\vert L_{\Delta
}^{\Sigma }\right\vert \right) >1$.
\end{proposition}

\begin{proof}
Recall that $L_{\Delta}^{\Sigma }$ is a large subgroup of $L^{\Sigma }$ such that
$\left( \Phi _{\zeta }^{\ast }(p), \left\vert L_{\Delta
}^{\Sigma }\right\vert \right) >1$ by Lemma \ref{Phi0}. Let $M$ be a subgroup of $L$ such that $M^{\Sigma }$ is a maximal subgroup of $L^{\Sigma }$ containing $L_{\Delta }^{\Sigma }$, then $M^{\Sigma }$ is large and $\left( \Phi _{nf}^{\ast }(p), \left\vert
M^{\Sigma }\right\vert \right) >1$. If $\left(L^{\Sigma}, L_{\Delta}^{\Sigma} \right)$ is not $\left(P \Omega^{+}_{8}(q),G_{2}(q) \right)$ and $\left(PSU_{4}(3) , A_{7}\right)$ then we may use the same argument of \cite[Theorem 7.1]{Mo1}, with $L^{\Sigma }, L^{\Sigma }_{\Delta}$ and $M^{\Sigma }$ in the role of $X,X_{x}$ and $Y$ respectively, to prove that $M^{\Sigma }$ is a geometric subgroup of $L^{\Sigma }$.

Assume that $L^{\Sigma} \cong P \Omega^{+}_{8}(q)$ and $L_{\Delta}^{\Sigma} \cong G_{2}(q)$. Then $G(\Sigma )=1$ by Corollary \ref{QuotL}. If $L_{\Delta }^{\Delta }$ does not act $2$-transitively on $\Delta $, then $\mathrm{Soc}(G_{\Delta }^{\Delta })<L_{\Delta }^{\Delta }\leq G_{\Delta}^{\Delta }\leq A\Gamma L_{1}(u^{h})$, where $u^{h}=\lambda +2$. Then $\lambda <f$, where $q=p^{f}$, and hence $p^{2f} \leq \left\vert L: L_{\Delta}\right \vert< f^{2}$ a contradiction. Thus $L_{\Delta }^{\Delta }$ acts $2$-transitively on $\Delta $, and hence $q=2$. However, this is impossible by Lemma \ref{ciciona}.

Assume that $L^{\Sigma} \cong PSU_{4}(3)$ and $L_{\Delta}^{\Sigma} \cong A_{7}$. Then $\lambda=36$, and hence $G(\Sigma )=1$ by Corollary \ref{QuotL}. Therefore, one obtains $G \leq P \Gamma U_{4}(3)$. However, this is impossible by Lemma \ref{PP} since $\lambda+1=37$ does not divide the order of $G$. Thus, we have proven that $M^{\Sigma }$ is a geometric subgroup of $L^{\Sigma }$ containing $L_{\Delta}^{\Sigma}$.

If $L_{\Delta}^{\Sigma } \neq M^{\Sigma }$, then $G_{\Delta}^{\Sigma }$ is a novelty, and hence $L_{\Delta}^{\Sigma }$ is listed in \cite[Tables
3.5.H--I]{KL} for $n\geq 13$ and in \cite{BHRD} for $3 \le n\leq 12$. Now, the candidates for $L_{\Delta}^{\Sigma }$ must be large subgroups of $L^{\Sigma }$ and their order must have a factor in common with $\Phi _{\zeta}^{\ast }(p)$. For instance,  if $L^{\Sigma } \cong PSL_{n}(q)$, $q=p^{f}$, then $\zeta=nf$ by \cite[Proposition 5.2.16]{KL}, and the unique admissible case is when $L_{\Delta}^{\Sigma }$ lies in a maximal member of $\mathcal{C}_{1}(L^{\Sigma})$. However, this is impossible by \cite[Theorem 3.5(iv)]{He}. Therefore, no novelties occur when $L^{\Sigma } \cong PSL_{n}(q)$. The remaining simple classical groups are analyzed similarly and it is straightforward to check that there are no novelties which are compatible with the constraints on $L_{\Delta}^{\Sigma}$. Thus $L_{\Delta}^{\Sigma } = M^{\Sigma }$, which is the assertion.
\end{proof}

\bigskip

\begin{lemma}
\label{PSL1}$L^{\Sigma }$ is not isomorphic to $PSL_{n}(q)$.
\end{lemma}

\begin{proof}
Assume that $L^{\Sigma }\cong PSL_{n}(q)$. Then $n\geq 3$ by Lemma \ref%
{nMag2}, and $L_{\Delta}^{\Sigma }$ is a large maximal geometric subgroup of $L^{\Sigma }$ such that
$\left( \Phi _{nf }^{\ast }(p), \left\vert L_{\Delta
}^{\Sigma }\right\vert \right) >1$ by Proposition \ref{Phi} and by \cite[Proposition 5.2.16]{KL}. Then $L_{\Delta}^{\Sigma } \notin \mathcal{C}_{1}(L^{\Sigma})$ by \cite[Theorem 3.5(iv)]{He}, and hence
one of the following holds by \cite[Propositions 4.7]{AB}:

\begin{enumerate}
\item[(i)] $\mathrm{Soc}(L_{\Delta }^{\Sigma })$ is one of the groups $%
PSp_{n}^{\prime }(q)$, $PSU_{n}(q^{1/2})$ and $n$ odd, or $P\Omega
_{n}^{-}(q)$;

\item[(ii)] $L_{\Delta }^{\Sigma }$ is a $\mathcal{C}_{3}$-group of type $%
GL_{n/t}(q^{t})$, where $t=2$, or $t=3$ and either $q=2,3$ or $q=5$ and $n$
is odd.
\end{enumerate}

Assume that (i) holds. Then $G(\Sigma )=1$ by Corollary \ref{QuotL}. If $%
L_{\Delta }^{\Delta }$ does not act $2$-transitively on $\Delta $, then $%
\mathrm{Soc}(G_{\Delta }^{\Delta })<L_{\Delta }^{\Delta }\leq G_{\Delta
}^{\Delta }\leq A\Gamma L_{1}(u^{h})$, where $u^{h}=\lambda +2$ by
Proposition \ref{unapred}. If $\mathrm{Soc}(L_{\Delta }^{\Sigma })\cong
PSU_{n}(q^{1/2})$ then $\lambda +2\leq q-1$ by \cite[Proposition 4.8.5(II)]%
{KL}, and hence $q^{2}+q+1\leq P(L)\leq \left\vert \Sigma \right\vert \leq
(q-3)^{2}$ since $n\geq 3$, which is a contradiction. Thus, $L_{\Delta
}^{\Delta }$ acts $2$-transitively on $\Delta $ in this case. The previous
argument together with \cite[Propositions 4.8.3--4.8.4]{KL} can be used to
show that $L_{\Delta }^{\Delta }$ acts $2$-transitively on $\Delta $ also in
the remaining cases. Furthermore, $(n,q)\neq (4,2),(6,2)$ by Lemmas \ref%
{Alt1} and \ref{Phi} respectively, and $n\geq 3$. Then one of the following
holds by \cite[Propositions 4.8.3--4.8.5]{KL}:

\begin{enumerate}
\item $n > 6$, $q=2$, $L_{\Delta } \cong Sp_{n}(2)$ and $\lambda=2^{2n-1}\pm
2^{n}-2$;

\item $n=3$, $\mathrm{Soc}(L_{\Delta }) \cong PSU_{3}(q^{1/2})$ and $%
\lambda=q^{3/2}-1$;

\item $n=4$, $\mathrm{Soc}(L_{\Delta }) \cong P\Omega _{4}^{-}(q)\cong
PSL_{2}(q^{2})$ and $\lambda=q^{2}-1$.
\end{enumerate}

Cases (2) and (3) are immediately ruled out since $(\lambda,p)=1$ but $%
L_{\Delta}$ does not contain a Sylow $p$-subgroup of $L$. In case (1) $%
2^{n(n-1)/2-2}$ must divide the order of a Sylow $2$-subgroup of $L_{\Delta}$
which is $2^{n^{2}/4}$, and we reach a contradiction since $n>6$.

Assume that (ii) holds. Then $L_{\Delta }^{\Sigma }\cong
Z_{a}.PSL_{n/t}(q^{t}).Z_{e}.Z_{t}$, where $a=\frac{(q-1,n/t)(q^{t}-1)}{%
(q-1)(q-1,n)}$ and $e=\frac{(q^{t}-1,n/t)}{(q-1,n/t)}$ by \cite[Proposition
4.3.6.(II)]{KL}, and $t=2,3$. Then $G(\Sigma )=1$ by Corollary \ref{QuotL}.
Recall that $\mathrm{Soc}(G_{\Delta }^{\Delta })\trianglelefteq L_{\Delta
}^{\Delta }$ by Proposition \ref{unapred}.

If $G_{\Delta }^{\Delta }$ is of affine type, then $L_{\Delta }^{\Delta
}\leq Z_{e}.Z_{t}$ for $n>t$. Thus $\lambda +2\leq e\leq n/t$, and hence 
\begin{equation}
\frac{q^{n}-1}{q-1}=P(L)\leq \left\vert L:L_{\Delta }\right\vert =\left\vert
\Sigma \right\vert \leq n^{2}/2  \label{rain}
\end{equation}%
by \cite[Theorem 5.2.2]{KL} since $n\geq 3$ and $(n,q)\neq (4,2)$. However, (%
\ref{rain}) has no admissible solutions. Thus $n=t=3$ since $t=2,3$ and $%
n\geq 3$. Moreover, $L_{\Delta }\cong Z_{\frac{q^{2}+q+1}{(3,q-1)}}.Z_{3}$
and hence $\left\vert \Sigma \right\vert =\frac{1}{3}q^{3}(q+1)(q-1)^{2}$.
On the other hand, since $\lambda +2$ is a power of prime and divides the
order of $L_{\Delta }$, it follows that $\lambda \leq q^{2}+q-1$ and $%
\left\vert \Sigma \right\vert \leq (q^{2}+q-1)^{2}$, and we reach a
contradiction.

If $G_{\Delta }^{\Delta }$ is almost simple, then $t<n$ and $L_{\Delta
}^{\Delta }$ acts $2$-transitively on $\Delta $ by Proposition \ref{unapred}%
. Thus $L_{\Delta }^{\Delta }\cong PSL_{n/t}(q^{t}).Z_{e}.Z_{t}$ and hence
either $(n/t,q^{t})=(2,9)$ and $\lambda +2=6$, or $\lambda +2=\frac{q^{n}-1}{%
q^{t}-1}$ since $t=2,3$. The former contradicts $\lambda >10$, the latter
yields $\lambda =\frac{q^{n}-2q^{t}+1}{q^{t}-1}$ and hence $\left( \lambda
,p\right) =1$. Then $L_{\Delta }^{\Sigma }$ must contain a Sylow $p$%
-subgroup of $L^{\Sigma }$, which is a contradiction.
\end{proof}

\bigskip

\begin{lemma}
\label{PSU1}$L^{\Sigma }$ is not isomorphic to $PSU_{n}(q)$.
\end{lemma}

\begin{proof}
Assume that $L^{\Sigma }\cong PSU_{n}(q)$. Then $n\geq 3$ by Lemma \ref%
{nMag2}. Moreover, $L_{\Delta}^{\Sigma }$ is a large maximal geometric subgroup of $L^{\Sigma }$ such that
$\left( \Phi _{\zeta }^{\ast }(p), \left\vert L_{\Delta
}^{\Sigma }\right\vert \right) >1$, where $\zeta$ is either $nf$ or $(n-1)f$ according to whether $n$ is even or odd respectively, by Proposition \ref{Phi} and \cite[Proposition 5.2.16]{KL}. Then $L_{\Delta}^{\Sigma } \notin \mathcal{C}_{1}(L^{\Sigma})$ by \cite[Theorem 3.5(iv)]{He}, and hence
one of the following holds by \cite[Propositions 4.17]{AB}:

\begin{enumerate}
\item[(i)] $L_{\Delta }^{\Sigma }$ is a $\mathcal{C}_{1}$-subgroup of $%
L^{\Sigma }$;

\item[(ii)] $L_{\Delta }^{\Sigma }$ is a $\mathcal{C}_{3}$-subgroup of $%
L^{\Sigma }$ of type $GU_{n/3}(3^{3})$ with $n$ odd.

\end{enumerate}

Assume that (i) holds. Then $n$ is even by \cite[Theorem 3.5(iv)]{He}, and hence $L_{\Delta
}^{\Sigma } \cong Z_{\frac{q+1}{(q+1,n)}}.PSU_{n-1}(q).Z_{(q+1,n-1)}$ is the stabilizer of a non-isotropic point of $PG_{n-1}(q^{2})$
by \cite[Propositions 4.1.4.(II)-4.1.18.(II)]{KL}. Also, $G(\Sigma)=1$ by Corollary \ref{QuotL}. If $L_{\Delta }^{\Delta }$ does not act $2$-transitively on $\Delta $, then $%
\mathrm{Soc}(G_{\Delta }^{\Delta })<L_{\Delta }^{\Delta }\leq G_{\Delta
}^{\Delta }\leq A\Gamma L_{1}(u^{h})$ with $u^{h}=\lambda +2$ by
Proposition \ref{unapred}. Then $\lambda \leq (q+1,n-1) -2$ by \cite[Propositions 4.1.4.(II)]{KL}, and hence $q^{n-1} \leq \left\vert L:L_{\Delta} \right \vert <q$, which is impossible for $n \geq 4$. Then $L_{\Delta }^{\Delta }$ acts $2$-transitively on $\Delta $, and hence $n=4$ and $\lambda^{2}=q^{3}(q-1)$. So $q-1$ is a square, which is impossible by \cite[B1.1]{Rib}. 

Assume that (ii) holds. Then $L_{\Delta}^{\Sigma}\cong
Z_{7}.PSU_{n/3}(3^{3}).Z_{(n/3,7)}.Z_{3}$ by \cite[Proposition 4.3.6(II)]{KL}. Hence, $G(\Sigma)=1$ by Corollary \ref{QuotL}. Also, $L_{\Delta}^{\Delta}$ is forced to act $2$-transitively on $\Delta$ by Proposition \ref{unapred} since $\lambda>10$. Therefore, $n=9$ and $\lambda= 3^{9}-2$. However, $\left\vert L:L_{\Delta }\right\vert \neq \lambda ^{2}$ in this case, which is then ruled out.
\end{proof}

\bigskip

\begin{lemma}
\label{PSp1}$L^{\Sigma }$ is not isomorphic to $PSp_{n}(q)^{\prime}$.
\end{lemma}

\begin{proof}
Assume that $L^{\Sigma }\cong PSp_{n}(q)^{\prime}$. Then $n\geq 4$ by Lemma \ref%
{nMag2} since $n$ is even. Also $(n,q) \neq (4,2)$ by Lemma \ref{ASt1} since $PSp_{4}(2)^{\prime} \cong A_{6}$. Thus $L^{\Sigma }\cong PSp_{n}(q)$. Moreover, $L_{\Delta}^{\Sigma }$ is a large maximal geometric subgroup of $L^{\Sigma }$ such that
$\left( \Phi _{nf }^{\ast }(p), \left\vert L_{\Delta
}^{\Sigma }\right\vert \right) >1$ by Proposition \ref{Phi} and \cite[Proposition 5.2.16]{KL}. Then $L_{\Delta}^{\Sigma } \notin \mathcal{C}_{1}(L^{\Sigma})$ by \cite[Theorem 3.5(iv)]{He}, and hence
one of the following holds by \cite[Propositions 4.22]{AB}:
\begin{enumerate}
\item $L_{\Delta}^{\Sigma}$ is a $\mathcal{C}_{8}$-subgroup of $L^{\Sigma}$;

\item $L_{\Delta}^{\Sigma}$ is a $\mathcal{C}_{3}$-subgroup of $L^{\Sigma}$ of type $Sp_{n/2}(q^{2})$, $%
Sp_{n/3}(q^{3})$ or $GU_{n/2}(q)$;

\item $\left( L^{\Sigma}, L_{\Delta}^{\Sigma}\right)$ is either $(PSp_{4}(7),2^{4}.O_{4}^{-}(2))$ or $(PSp_{4}(3),2^{4}.\Omega_{4}^{-}(2))$.

\end{enumerate}

Assume that Case (i) holds. Then $L_{\Delta}^{\Sigma}\cong O_{n}^{\varepsilon}(q)$, with $\varepsilon= \pm$ and $q$ even, by \cite%
[Proposition 4.8.6.(II)]{KL}. Then $\lambda^{2}=\frac{q^{n/2}}{2}(q^{n/2}+ \varepsilon)$ since $\lambda^{2}=\left\vert L^{\Sigma}:L_{\Delta}^{\Sigma}\right\vert $, and hence $q^{n/2}+ \varepsilon$ is a square. Then $(n,q)=(6,2)$ by \cite[B1.1]{Rib} since $n\geq 4$. However, $L^{\Sigma} \cong PSp_{6}(2)$ cannot occur by Lemma \ref{ciciona}.

Assume the Case (ii) holds. Then either $L_{\Delta}^{\Sigma}$ is isomorphic either to $PSp_{n/t}(q^{t}).Z_{t}$ with $t=2,3$ or $Z_{(q+1)/2}.PGU_{n/2}(q).Z_{2}$ with $q$ odd by \cite[Propositions 4.3.7(II) and 4.3.10.(II)]{KL}. Also, in both cases it results $G(\Sigma)=1$ by Corollary \ref{QuotL}.

If $L_{\Delta}^{\Delta}$ does not act $2$-transitively on $\Sigma$, then $Z_{(q+1)/2}.PSU_{n/2}(q) \leq L(\Delta)$ and $\left \vert L_{\Delta}^{\Delta}\right \vert \mid 4f(n/2,q+1)$ since $\lambda>10$. By \cite[Corollary 4.3(ii)--(iii)]{AB} we obtain $$p^{f \frac{n^{2}+2n}{4}}  \leq\left\vert L^{\Sigma}:L_{\Delta}^{\Sigma}\right\vert \leq 16f^{2}(n/2,q+1)^{2},$$ which has no solutions for $q$ odd and $n \geq 4$.

If $L_{\Delta}^{\Delta}$ acts $2$-transitively on $\Sigma$. Then $L_{\Delta}^{\Sigma}$ is isomorphic either to $PSp_{2}(q^{t}).Z_{t}$ with $t=2,3$ or $Z_{(q+1)/2}.PGU_{3}(q).Z_{2}$ with $q$ odd. Then either $\lambda =q^{t}-1$ with $t=2,3$ or $q^{3}-1$ respectively. In each case $(\lambda,q)=1$, and hence $L_{\Delta }^{\Sigma }$ must contain a Sylow $p$%
-subgroup of $L^{\Sigma }$, which is a contradiction.

Finally, Case (iii) cannot occur since $\left\vert L^{\Sigma}:L_{\Delta}^{\Sigma}\right\vert $ is a non-square. 
\end{proof}

\bigskip

\begin{lemma}
\label{Cla1}$L^{\Sigma }$ is not isomorphic to simple classical group.
\end{lemma}

\begin{proof}
In order to prove the assertion we only need to tackle the case $L^{\Sigma
}\cong P\Omega _{n}^{\varepsilon }(q)$, where $\varepsilon \in \left\{ \circ
,\pm \right\} $ since the remaining classical groups are ruled out in
Lemmas \ref{PSL1}, \ref{PSU1} and \ref{PSp1}. Hence, assume that $L^{\Sigma
}\cong P\Omega _{n}^{\varepsilon }(q)$, where $\varepsilon \in \left\{ \circ
,\pm \right\} $. Then $n\neq 4$ by Lemma \ref{nMag2}. Moreover $n\neq 5,6$
by Lemmas \ref{PSL1}, \ref{PSU1} and \ref{PSp1}, since $P\Omega _{5}^{\circ
}(q)\cong PSp_{4}(q)$ with $q$ odd, $P\Omega _{6}^{+}(q)\cong PSL_{4}(q)$
and $P\Omega _{6}^{-}(q)\cong PSL_{6}(q)$. Thus $n\geq 7$. Moreover $%
(n,q,\varepsilon )\neq (8,2,+)$ by Lemma \ref{ciciona}. Finally, $L_{\Delta}^{\Sigma }$ is a large maximal geometric subgroup of $L^{\Sigma }$ such that $\left( \Phi _{\zeta }^{\ast }(p), \left\vert L_{\Delta
}^{\Sigma }\right\vert \right) >1$ by Proposition \ref{Phi}, where $\zeta$ is either $nf$, $(n-1)f$ or $(n-2)f$ by \cite[Proposition 5.2.16]{KL} according to whether $\varepsilon$ is $-,\circ$ or $+$ respectively. Thus, one of the following holds by \cite[Proposition 4.23]{AB}:

\begin{enumerate}
\item[(i)] $L_{\Delta }^{\Sigma }$ is $\mathcal{C}_{1}$-subgroup
of $L^{\Sigma }$;

\item[(ii)] Either $(n,q)=(7,3),(7,5)$ or $(n,q,\varepsilon)=(8,3,+)$ and $L_{\Delta }^{\Sigma }$ is of type $O_{1}(q) \wr S_{n}$;

\item[(iii)]  $L_{\Delta }^{\Sigma }$ is a $\mathcal{C}_{3}$-subgroup
of $L^{\Sigma }$. Moreover, its type is either $O_{n/2}^{\varepsilon^{\prime}}(q^{2})$ with $(\varepsilon, \varepsilon^{\prime})=(-,-)$ and $n/2$ even or $(\varepsilon, \varepsilon^{\prime})=(+,\circ)$ and $n/2$ odd, or $GU_{n/2}(q)$ with $\varepsilon=-$ and $n/2$ odd or $\varepsilon=+$ and $n/2$ even;
  
\item[(iv)] $L^{\Sigma } \cong P\Omega^{+}_{8}(3)$ and $L_{\Delta }^{\Sigma
}\cong 2^{6}.\Omega_{6}^{+}(2)$.
\end{enumerate}

Assume that (i) holds. Then one of the following cases occurs by \cite[Propositions
4.1.6(II), 4.1.7(II) and 4.1.20(II)]{KL}:

\begin{enumerate}
\item $L_{\Delta }^{\Sigma }$ is the stabilizer in $L^{\Sigma }$ of a
non-singular point of $PG_{n-1}(q)$:

\begin{enumerate}
\item $\varepsilon =\circ $ and $L_{\Delta }^{\Sigma }\cong \Omega
_{n-1}^{-}(q).Z_{2}$

\item $\varepsilon =+$ and $L_{\Delta }^{\Sigma }\cong \Omega _{n-1}(q)$
with $q\equiv 1\pmod 4$, or $q\equiv 3\pmod 4$ and $n/2$ even;

\item $\varepsilon =+$ and $L_{\Delta }^{\Sigma }\cong \Omega _{n-1}(q).Z_{2}
$ with $q\equiv 3\pmod 4$ and $n/2$ odd;

\item $\varepsilon =+$ and $L_{\Delta }^{\Sigma }\cong Sp_{n-2}(q)$ with $q$
even.
\end{enumerate}

\item $\varepsilon =+$ and $L_{\Delta }^{\Sigma }$ is the stabilizer in $%
L^{\Sigma }$ of a non-singular line of type \textquotedblleft $-$" of $%
PG_{n-1}(q)$:

\begin{enumerate}
\item $L_{\Delta }^{\Sigma }\cong \left( Z_{\frac{q+1}{(q+1,2)}}\times
\Omega _{n-2}^{-}(q)\right) .Z_{2}$ with $q$ even or $q\equiv 1\pmod4$;

\item $L_{\Delta }^{\Sigma }\cong \left( Z_{\frac{q+1}{2}}\times \Omega
_{n-2}^{-}(q)\right) .[4]$ with $q\equiv 3\pmod 4$ and $n/2$ odd;

\item $L_{\Delta }^{\Sigma }\cong Z_{2}.\left( Z_{\frac{q+1}{4}}\times
P\Omega _{n-2}^{-}(q)\right) .[4]$ with $q\equiv 3\pmod 4$ and $n/2$ even.
\end{enumerate}
\end{enumerate}

In each case $G(\Sigma )=1$ by Corollary \ref{QuotL} since $n\geq 7$. Hence, 
$L=L^{\Sigma }$ and $G=G^{\Sigma }$. Moreover, $L_{\Delta }^{\Delta }$ acts $%
2$-transitively on $\Delta $ by Proposition \ref{unapred} since $\lambda >10$%
. Then $\varepsilon=+$, $q=2$, $L_{\Delta }^{\Sigma }\cong Sp_{n-2}(2)$ and $\lambda =2^{2(n/2-1)}\pm 2^{n/2-2}-2$. On the
other hand, 
\begin{equation*}
\left\vert L^{\Sigma }:L_{\Delta }^{\Sigma }\right\vert =\lambda
^{2}=2^{n/2-1}\left( 2^{n/2}-1\right)
\end{equation*}%
and hence $2^{2}(2^{2(n/2-1)-1}\pm 2^{n/2-1}-1)^{2}=2^{n/2-1}\left(
2^{n/2}-1\right) $, which has no admissible integer solutions for $n\geq 8$.

Assume that (iii) holds. The possibilities for $L_{\Delta }^{\Sigma }$ are
provided in \cite[Propositions 4.3.16(II), 4.3.18(II) and 4.3.20(II)]{KL}.
Thus $G(\Sigma )=1$ by Corollary \ref{QuotL}. If $%
L_{\Delta }^{\Delta }$ does not act $2$-transitively on $\Delta $, then $%
\mathrm{Soc}(G_{\Delta }^{\Delta })<L_{\Delta }^{\Delta }\leq G_{\Delta
}^{\Delta }\leq A\Gamma L_{1}(u^{h})$ with $u^{h}=\lambda +2$ by
Proposition \ref{unapred}. Thus $L_{\Delta }$ is forced to be
of type $GU_{n/2}(q)$ with $\left\vert L_{\Delta }^{\Delta }\right\vert
\mid (n/2,2,q)(q+1,n/2)$ since $\lambda >10$. Then $\lambda \leq n-2$ and so 
$q^{\frac{1}{8}n\left( n+2\right) }\leq \left\vert L:L_{\Delta }\right\vert
\leq (n-2)^{2}$, which has no solutions for $n\geq 8$. Thus $L_{\Delta
}^{\Delta }$ acts $2$-transitively on $\Delta $, and hence $L \cong P\Omega^{-}_{8}(q)$ and $L_{\Delta } \cong P\Omega^{-}_{4}(q).Z_{4}\cong PSL_{2}(q^{4}).Z_{4}$ since $n \geq 8$. Therefore $\lambda=q^{4}-2$. If $q$ is odd then $L_{\Delta}$ must contain a Sylow $p$-subgroup of $L$ since $\left\vert L:L_{\Delta} \right\vert=\lambda^{2}$, which is not the case. So $q$ is even and $\left\vert \Sigma \right \vert = q^{12}(q^{6}-1)(q^{2}-1)$ which is different from $(q^{4}-2)^{2}$. 


Finally, it is easy to check that $\left\vert L^{\Sigma }:L_{\Delta
}^{\Sigma }\right\vert $ is a non-square in (ii) and (iv), hence these are
ruled out. This completes the proof.
\end{proof}

\bigskip

\begin{proof}[Proof of Theorem \protect\ref{T1}]
Since $L^{\Sigma }$ is almost simple by Theorem \ref{ASPQP}, the assertion
follows from Lemmas \ref{spor1}, \ref{Alt1}, \ref{Exc1} or \ref{Cla1}
respectively.
\end{proof}

\section{Classification of the $2$-designs of type 2}

In this section we assume that $\mathcal{D}$ is of type $2$. Recall that $G(\Sigma )=1$ and $G$ is an almost simple
group acting point-quasiprimitively on $\mathcal{D}$ by Theorem \ref{ASPQP}.
Thus, $G=G^{\Sigma }$ and $L=L^{\Sigma }$ where $L=\mathrm{Soc}(G)$.
Further constraints for $L$ are provided in Proposition \ref{unapred} and Theorem \ref{Large} which are then combined with the results
contained in \cite{AB, LS}. An important restriction is provided in Lemma \ref{Sevdah} where it is proven that, if $L$ is Lie type simple group, either $L_{\Delta }$ lies in a maximal parabolic subgroup of $L$ or $L_{\Delta }^{\Delta }$ is a non-solvable group acting $2$-transitively on $\Delta $. We use all these information to prove the
following result.

\bigskip

\begin{theorem}
\label{T2}If $\mathcal{D}$ is a symmetric $2$-$\left( \left( \frac{\lambda +2%
}{2}\right) \left( \frac{\lambda ^{2}-2\lambda +2}{2}\right) ,\frac{\lambda
^{2}}{2},\lambda \right) $ design admitting a flag-transitive and
point-imprimitive automorphism group, then $\lambda \leq 10$.
\end{theorem}

\bigskip
We analyze the cases where $L^{\Sigma }$ is sporadic, alternating, classical or exceptional of Lie type separately. 

\bigskip
Recall that, when $\mathcal{D}$ is of type (2) either $\lambda \equiv 0 \pmod{4}$ and hence $\left\vert \Delta \right \vert =\lambda/2+1$ is odd, or $\lambda =2w^{2}$, where $w$ is odd, $w\geq 3$, $2(w^{2}-1)$ is a square and $\left\vert \Delta \right \vert=w^{2}+1$. In both cases it results that $\left\vert \Delta \right \vert \not \equiv 0 \pmod{4}$.

\medskip

A preliminary filter in the study of the the $2$-designs of type 2 is the following lemma.

\bigskip

\begin{lemma}
\label{Jedan}If $\mathcal{D}$ is of type 2, then the following hold:

\begin{enumerate}

\item $\left\vert \Sigma \right\vert $ is odd and $2\left\vert \Sigma \right\vert -1$ is a square.

\item If $u$ is any prime divisor of $\left\vert \Sigma \right\vert $, then $%
u\equiv 1\pmod{4}$.

\item If $\lambda=2w^{2}$, $w$ odd, $w \geq 3$ and such that $2(w^{2}-1)$ is a square, then $\mathrm{Soc}(G_{\Delta}^{\Delta})$ is isomorphic to one of the groups $ A_{w^{2}+1}$, $PSL_{2}(w^{2})$ or $PSU_{3}(w^{2/3})$.  
\end{enumerate}
\end{lemma}

\begin{proof}
$\left\vert \Sigma \right\vert $ is clearly odd. If $\lambda =2w^{2}$, where $w$ is odd, $w\geq 3$, and $2(w^{2}-1)=x^{2}$ then $2w^{4}-2w^{2}t+\left( 1-\left\vert \Sigma \right\vert
\right) =0$, whereas if $\lambda =4t$ for some $t\geq 1$, then $16t^{2}-8t+\left( 2-2\left\vert \Sigma \right\vert
\right) =0$. In both cases $y^{2}=2\left\vert \Sigma \right\vert -1$ for some
positive integer $y$. Therefore, $\left\vert \Sigma \right\vert =\frac{y^{2}+1%
}{2}$ is odd. Moreover, if $u$ is any prime divisor of $\left\vert \Sigma
\right\vert $ then $y^{2}\equiv -1\pmod{u}$ and hence $u\equiv 1\pmod{4}$.
Thus, we obtain (1) and (2).

Finally, (3) follows from the first part of the proof of Lemma \ref{firstpart}.
\end{proof}  
\bigskip

\begin{lemma}
\label{spor2}$L$ is not isomorphic to a sporadic group.
\end{lemma}

\begin{proof}
Assume that $L$ is sporadic. Then the possibilities for $G$ and $G_{x}$, where 
$x$ is any point of $\mathcal{D}$, and hence for $\left\vert \Sigma \right\vert
=\left\vert G:G_{x}\right\vert $ are provided in \cite{As}. It is easy to
see that $2\left\vert \Sigma \right\vert -1$ is never square if $\left\vert \Sigma \right\vert$ is any
of such degrees. Thus $L$ cannot be a sporadic simple group by Lemma \ref%
{Jedan}(2).
\end{proof}

\bigskip

\begin{lemma}
\label{Alt2}If $L$ is not isomorphic to $A_{s}$ with $s \geq 5$.
\end{lemma}

\begin{proof}
Assume that $L\cong A_{s}$, where $s\geq 5$. Then one of the following holds by \cite{LS}:
\begin{enumerate}
\item $G_{\Delta }\cong \left(S_{t}\times S_{s-t} \right)\cap G$, $1\leq t<s/2$;
\item $G_{\Delta }\cong (S_{s/t}\wr S_{t})\cap G$, $s/t,t>1$;
\item $G_{\Delta }\cong A_{7}$ and $\left\vert \Sigma \right\vert =15$.
\end{enumerate}

Assume that (i) holds. Then $\left\vert \Sigma
\right\vert =\binom{s}{t}$ and $A_{t}\times A_{s-t}\trianglelefteq L_{\Delta
}\leq \left( S_{t}\times S_{s-t}\right) \cap L$. Suppose that $t\geq 5$.
Then $s-t>t\geq 5$. Hence, both $A_{t}$ and $A_{s-t}$ are simple groups. Moreover, $%
L_{\Delta }^{\Delta }$ is either $A_{t}$ or $A_{s-t}$ by Proposition \ref%
{unapred} since $\lambda>10$. If $%
L_{\Delta }^{\Delta } \cong A_{t}$ then either $\left\vert \Delta \right\vert =t$, or $t=6$ and $\left\vert \Delta \right\vert =10$, or $t=7,8$ and $\left\vert \Delta \right\vert =15$ since $t\geq 5$. Actually, $t=6$ and $\left\vert \Delta \right\vert=10$ imply $\lambda=2w^{2}=18$ and $\left\vert \Sigma \right \vert=145$, which is not of the form $\binom{s}{6}$ and hence it cannot occur. Also, if $t=7,8$ and $\left\vert \Delta
\right\vert =15$ then $\binom{s}{t}=\left\vert \Sigma \right\vert
=365$, and we reach a contradiction. Thus, $\lambda =2(t-1)$, and hence $\left\vert \Sigma \right\vert =\allowbreak 2t^{2}-6t+5$. Therefore, we have 
\begin{equation*}
2^{t}<\left( \frac{s}{t}\right) ^{t}\leq \binom{s}{t}=2t^{2}-6t+5\text{,}
\end{equation*}%
which is impossible for $t\geq 5$. We reach the same contradiction for $L_{\Delta }^{\Delta }\cong A_{s-t}$.

Assume that $1\leq t\leq 4$. If $L_{\Delta }^{\Delta }$ is non-solvable,
then $s-t\geq 5$ and $L_{\Delta }^{\Delta }\cong A_{s-t}$ and the previous
argument rules out this case. Thus $L_{\Delta }^{\Delta }$ is solvable, and
hence $\left\vert \Delta \right\vert =3$ by Proposition \ref{unapred} since $L_{\Delta }\leq \left(
S_{t}\times S_{s-t}\right) \cap L$ and $\left\vert \Delta \right\vert $
is odd by Lemma \ref{Jedan}(3). Then $\lambda =4$, whereas $\lambda >10$ by our assumptions.

Assume that (ii) holds. Then 
$\left\vert \Sigma \right\vert =\frac{s!}{((s/t)!)^{t}(t!)}$ and $A_{s/t}\wr
A_{t}\leq L_{\Delta }\leq \left( S_{s/t}\wr S_{t}\right) \cap L$. Moreover, $%
\left( A_{s/t}\right) ^{t}\trianglelefteq L_{\Delta }$ and $A_{t}\leq
L_{\Delta }/\left( A_{s/t}\right) ^{t}\leq (Z_{2})^{t}:S_{t}$, where the action
of $A_{t}$ on $(Z_{2})^{t}$ and on its permutation module are equivalent. Thus $%
L_{\Delta }/\left( A_{s/t}\right) ^{t}$ is isomorphic to one of the groups $A_{t}$, $S_{t}$, $Z_{2}\times
A_{t}$, $Z_{2}\times S_{t}$, $(Z_{2})^{t-1}(2):A_{t}$, $(Z_{2})^{t-1}(2):S_{t}$, $%
(Z_{2})^{t}:A_{t}$ or $(Z_{2})^{t}:S_{t}$ by \cite[Lemma 5.3.4]{KL}.
If $t \leq 4$ then $L_{\Delta}^{\Delta}$ is solvable, and hence $\lambda/2+1=\left\vert \Delta \right\vert =3$ by Proposition \ref{unapred} since $\left\vert \Delta \right\vert $
is odd by Lemma \ref{Jedan}(3). However, this is impossible since $\lambda>10$. Thus $t\geq 5$. Also, $L_{\Delta}^{\Delta}$ is non-solvable otherwise we reach a contradiction as above. Then $L_{\Delta}^{\Delta}$ acts $2$-transitively on $\Delta$ by Proposition \ref{unapred}. Therefore $L_{\Delta }^{\Delta }\cong A_{t}$, and hence either $\left\vert \Delta
\right\vert =t$ for $t\geq 5$, or $t=6$ and $\left\vert \Delta \right\vert
=10$, or $t=7,8$ and $\left\vert \Delta \right\vert
=15$. On the other hand, $2^{t-1}\leq ((s/t)!)^{t-1}<\frac{s!}{((s/t)!)^{t}(t!)%
}<2t^{2}$ as shown in \cite[(34)]{MF}. Thus $(t,s)=(3,6),(5,10),(7,14)$ and
hence $\left\vert \Sigma \right\vert =15$, $945$, $135135$,
which are ruled out since they violate Lemma \ref{Jedan}(1). Then either $t=6$, $\left\vert \Delta \right\vert =15$ and $\left\vert \Sigma \right\vert =41$, or $t=7,8$, $%
\left\vert \Delta \right\vert =15$ and $\left\vert \Sigma \right\vert =365$. However both cases cannot occur since $\frac{(2t)!}{(2)^{t}(t!)}>
\left\vert \Sigma \right\vert$.

Finally, (iii) is excluded by Lemma \ref{Jedan}(1), as $2\left\vert \Sigma \right\vert -1$ is
not a square.
\end{proof}

\bigskip

\begin{lemma}
\label{Sevdah}The following hold:

\begin{enumerate}
\item If either $p\mid \lambda -\mu $ for some $\mu \in \left\{
0,1,2\right\} $, or $p\mid \lambda -3$ and $p\neq 5$, then $L_{\Delta }$
lies in a maximal parabolic subgroup of $L$.

\item If $L_{\Delta }$ does not lie in a maximal parabolic subgroup of $L$
then $L_{\Delta }^{\Delta }$ is a non-solvable group acting $2$-transitively
on $\Delta $.
\end{enumerate}
\end{lemma}

\begin{proof}
Since $\left\vert \Sigma \right\vert =\frac{\lambda ^{2}-2\lambda +2}{2}$ it
is immediate to see that $\left( \left\vert \Sigma \right\vert ,\lambda -\mu
\right) =1$ for either $\mu =0,1,2$, or $\mu =3$ and $p\neq 5$. In these
cases $L_{\Delta }$ contains a Sylow $p$-subgroup of $L$, and hence $%
L_{\Delta }$ lies in a maximal parabolic subgroup of \cite[Theorem 1.6]{Se}
since $L$ is a non-abelian simple group acting transitively on $\Sigma $,
and (1) holds.

Suppose to the contrary that $L_{\Delta }$ does not lie in a maximal
parabolic subgroup of $L$ and that $L_{\Delta }^{\Delta }$ is solvable. Then 
$p\mid \left\vert \Sigma \right\vert $ by \cite[Theorem 1.6]{Se}. Also $%
p\neq 2,3$ by Lemma \ref{Jedan}(1)--(2), and $\mathrm{Soc}(G_{\Delta }^{\Delta
})<L_{\Delta }^{\Delta }<G_{\Delta }^{\Delta }\leq A\Gamma L_{1}(u^{h})$,
where $\frac{\lambda }{2}+1=\left\vert \Delta \right\vert =u^{h}$ for some
prime $u$ by Proposition \ref{unapred}. Also, $u$ is odd since either $\lambda \equiv 0 \pmod{4}$, or $\lambda=2w^{2}$ with $w$ odd and $w\ge 3$ by Theorem \ref{PZM}. Then, $\lambda =2(u^{h}-1)$
with $h>1$ since $\lambda >10$. 

If $p=5$ divides $\lambda -3$, then $u=p$ since $\left( \left\vert \Sigma
\right\vert ,\left\vert \Delta \right\vert \right) =\left( \left\vert \Sigma
\right\vert ,\lambda -3\right) =\left( \left\vert \Delta \right\vert
,\lambda -3\right) \mid 5$. Thus $\left\vert \Sigma \right\vert =\allowbreak
2\cdot 5^{2h}-6\cdot 5^{h}+5$, and $5^{2}\nmid \left\vert \Sigma \right\vert 
$ since $h>1$. Hence, $L_{\Delta }$ contains a subgroup of index $p$ of a
Sylow $p$-subgroup of $L$ with $p=5$. It is a straightforward check that
none of the groups listed in \cite{LS} fulfills the previous constraint. So
this case is excluded. 

If $p\nmid \lambda -\mu $ for some $\mu \in \left\{ 0,1,2,3\right\} $. Then $%
\left( p,\left\vert L(\Delta )\right\vert \right) =1$ by Corollary \ref%
{CFixT}(2) since $p\neq 2,3$. Thus $\left\vert L_{\Delta }\right\vert
_{p}=\left\vert L_{\Delta }^{\Delta }\right\vert _{p}$, and hence $p\mid 
\frac{\lambda }{2}\left( \frac{\lambda }{2}+1\right) h$ since $L_{\Delta
}^{\Delta }<G_{\Delta }^{\Delta }\leq A\Gamma L_{1}(u^{h})$. Actually, $%
p\nmid \frac{\lambda }{2}$ by our assumption. If $p\mid \frac{\lambda }{2}+1$%
, then $p=5$ divides $\lambda -3$ since $\left( \left\vert \Sigma
\right\vert ,\left\vert \Delta \right\vert \right) =\left( \left\vert \Sigma
\right\vert ,\lambda -3\right) =\left( \left\vert \Delta \right\vert
,\lambda -3\right) \mid 5$ which we saw being impossible. Therefore, $%
\left\vert L_{\Delta }\right\vert _{p}\mid h$ and any Sylow $p$-subgroup of $%
L_{\Delta }$ is cyclic since $L_{\Delta }^{\Delta }<G_{\Delta }^{\Delta
}\leq A\Gamma L_{1}(u^{h})$. However, this is impossible by \cite{LS}. Thus $%
L_{\Delta }^{\Delta }$ is non-solvable, and hence $L_{\Delta }^{\Delta }$
acts $2$-transitively on $\Delta $ by Proposition \ref{unapred}.
\end{proof}

\bigskip

\begin{lemma}
\label{Exc2}$L$ is not a simple exceptional group of Lie type.
\end{lemma}

\begin{proof}
Assume that $L_{\Delta}$ is parabolic. Then $L \cong E_{6}(q)$ and $L_{\Delta} \cong [q^{16}].d.(P\Omega_{10}^{+}(q) \times (q-1)/ed).d$, where $d=(2,q-1)$ e $e=(3,q-1)$ by \cite[Table 1]{LS}. If $L_{\Delta}^{\Delta}$ is solvable then the order of $L_{\Delta}^{\Delta}$ must divide $(q-1)/e$. So does $\left\vert \Delta \right \vert$, and hence
$$\frac{q^{9}-1}{q-1}\cdot (q^{8}+q^{4}+1)=\left\vert \Sigma \right \vert \leq 2\left\vert \Delta \right \vert^{2} \leq 2\frac{(q-1)^{2}}{e^{2}},$$
which is clearly impossible. Thus $L_{\Delta}^{\Delta}$ is non-solvable acting $2$-transitively on $\Delta$ by Proposition \ref{unapred}, and this is impossible too.

Assume that $L_{\Delta}$ is not parabolic. Then $L_{\Delta}^{\Delta}$ is non-solvable acting $2$-transitively on $\Delta$ by Lemma \ref{Sevdah} with $\left\vert \Delta \right \vert$ odd. Also, $\left\vert L\right\vert \leq \left\vert L_{\Delta}\right\vert^{2}$ by Theorem \ref{Large}. All these constraints together with \cite[Table 1]{LS} lead to the following admissible cases:
\begin{enumerate}
\item $L \cong E_{7}$ and $\mathrm{Soc}(L_{\Delta}^{\Delta}) \cong PSL_{2}(q)$;
\item $L \cong$ $^{3}D_{4}(q)$ and $\mathrm{Soc}(L_{\Delta}^{\Delta})$ is isomorphic to one of the groups $ PSL_{2}(q^{j})$ with $j=1$ or $3$, $PSL_{3}(q)$ or $PSU_{3}(q)$;
\item $L \cong$ $^{2}G_{2}(q)$, $q=3^{2m+1}$, $m \geq 1$, and $\mathrm{Soc}(L_{\Delta}^{\Delta}) \cong PSL_{2}(q)$. 
\end{enumerate}   
Then possible values for $\left\vert \Delta \right \vert$ are $q+1,q^{2}+1,q^{2}+q+1,q^{2}+1$, hence $q \mid \lambda$ since $\left\vert \Delta \right \vert =\lambda/2+1$. Therefore $q$ is coprime to $\left\vert \Sigma \right \vert= \frac{\lambda^{2}-2\lambda+2}{2}$, and hence $L_{\Delta}$ must contain a Sylow $p$-subgroup of $L$, which is not the case. This completes the proof.   
\end{proof}

\bigskip

\begin{lemma}
\label{Filt}The following cases are admissible:

\begin{enumerate}
\item $q$ is even and $L_{\Delta }$ lies in a maximal parabolic subgroup of $%
L$;

\item $q$ is odd and one of the following holds:

\begin{enumerate}
\item $L_{\Delta }$ lies in maximal member of $\mathcal{C}_{1}(L)\cup \mathcal{C}_{2}(L)$, with $%
L_{\Delta }$ lying in a maximal parabolic subgroup of $L$ of type either $P_{i}$ or $P_{m,m-i}$ only for $%
L\cong PSL_{n}(q)$.

\item $L\cong PSL_{2}(q)$ and $L_{\Delta }$ is isomorphic to one of the
groups $D_{q\pm 1}$, $A_{4}$, $S_{4}$, $A_{5}$ or $PGL_{2}(q^{1/2})$.
\end{enumerate}
\end{enumerate}
\end{lemma}

\begin{proof}
Since $G$ acts primitively on $\Sigma$ and the size of this one is odd, then $G$ is one of the groups classified by \cite{LS}. Actually, $L\cong X(q)$, where $X(q)$ denotes any simple classical group as a consequence of Lemmas \ref{spor2}, \ref{Alt2} and \ref{Exc2}.

Assume that $G_{\Delta }=N_{G}(X(q_{0}))$ with $q=q_{0}^{c}$ and $q,c$ odd by \cite{LS}. Then $L_{\Delta }=X(q_{0})$ is maximal in $L$ by \cite[Tables H--I]{KL} for $n \geq 13$ and \cite[Section 8.2]{BHRD} for $2\leq n \leq 12$. Then $c=3$ and $L\cong PSL_{m}^{\epsilon }(q)$, where $\epsilon =\pm $ by \cite[Propositions 4.7, 4.17, 4.22 and 4.23]{AB} since $L_{\Delta }$ is a large subgroup of $L$ by Theorem \ref{Large}. Then $L_{\Delta }\cong \frac{j}{%
(q-\epsilon ,n)}.PGL_{n}^{\epsilon }(q^{1/3})$, where $j=\frac{q-\epsilon }{%
(q^{1/3}-\epsilon ,\frac{q-\epsilon }{(q-\epsilon ,n)})}$, by \cite[%
Proposition 4.5.3(II)]{KL}.

Assume that $L_{\Delta }^{\Delta}$ does not act $2$-transitively on $\Delta$. Then $L_{\Delta }^{\Delta}$ is solvable and $\lambda/2+1 \leq (n,q^{1/3}- \varepsilon) $ by Proposition \ref{unapred}. Moreover, it results that $\lambda \equiv 0 \pmod{4}$ by Lemma \ref{Jedan}(3). Therefore, $ q^{n(8n+3)/18} \leq \left\vert L: L_{\Delta}\right\vert <2(n,q^{1/3}- \varepsilon)^2 $, which is impossible for $n \geq 2$. Thus $L_{\Delta }^{\Delta}$ acts $2$-transitively on $\Delta$ and hence $n=2,3$. Then $\lambda=2q^{1/3}$, $2(q^{2/3}+q^{1/3})$ or $2q$ according to whether $n=2$, $(n,\varepsilon)=(3,+)$ or $(n,\varepsilon)=(3,-)$ respectively. Therefore, $\left\vert \Sigma \right \vert$ is coprime to $q$ and we reach a contradiction by \cite[Theorem 1.6]{Se} since $L^{\Delta}$ is a non-parabolic subgroup of $L$.

Assume that $L\cong \Omega _{7}(q)$ and $L_{\Delta }\cong \Omega _{7}(2)$. Then $q=3,5$ since $L_{\Delta}$ is a Large subgroup of $L$, and hence $\left\vert \Sigma \right\vert =3159$ or $157421875$ respectively. However, both contradict Lemma \ref{Jedan}(1), as $2\left\vert \Sigma \right\vert -1$ is
not a square.

Assume that $L\cong P\Omega _{8}^{+}(q)$, where $q$ is a prime and $q\equiv
\pm 3\pmod{8}$, and either $L_{\Delta }\cong \Omega _{8}^{+}(2)$ or $%
L_{\Delta }\cong 2^{3}\cdot 2^{6}\cdot PSL_{3}(2)$. In the former case $q=3,5$ since $L_{\Delta}$ is a Large subgroup of $L$, and hence $%
\left\vert \Sigma \right\vert =28431$ or $51162109375$. However, both contradict Lemma \ref{Jedan}(1). Then $L_{\Delta }\cong 2^{3}\cdot 2^{6}\cdot PSL_{3}(2)$, and hence $\left\vert \Sigma \right\vert =57572775$, but $2\left\vert \Sigma \right\vert -1$ is
not a square.     

Finally, the case $L\cong PSU_{3}(5)$ and $L_{\Delta }\cong P\Sigma L_{2}(9)$
implies $\left\vert \Sigma \right\vert =175$, and we again reach a
contradiction by Lemma \ref{Jedan}(1).
\end{proof}

\bigskip

\begin{lemma}
\label{m=2}$L$ is not isomorphic to $PSL_{2}(q)$.
\end{lemma}

\begin{proof}
Assume that $L\cong PSL_{2}(q)$ and $L_{\Delta }$ is isomorphic to one of
the groups $D_{q\pm 1}$, $A_{4}$, $S_{4}$, $A_{5}$ or $PGL_{2}(q^{1/2})$. If 
$L_{\Delta }\cong D_{q\pm 1}$ then $\left\vert \Sigma \right\vert =\frac{%
q(q\mp 1)}{2}$ and hence $2\left\vert \Sigma \right\vert -1=q^{2}+q+1$ or $%
(q-1)^{2}+(q-1)+1$ must be square by Lemma \ref{Jedan}(1). However this is
impossible by \cite[A7.1]{Rib}.

If $L_{\Delta }\cong PGL_{2}(q^{1/2})$ then $\left\vert \Sigma \right\vert
=q^{1/2}(q+1)/2$, and hence $2\left\vert \Sigma \right\vert
-1=q^{3/2}+q^{1/2}-1$. Moreover, $L_{\Delta }^{\Delta}$ acts $2$-transitively on $\Delta$. If it is not so, then $\lambda/2+1 \leq (2,q^{1/2}-1)$ as a consequence of Proposition \ref{unapred}, whereas $\lambda>10$. Thus, either $\left\vert \Delta
\right\vert =q^{1/2}+1$ or $\left\vert \Delta \right\vert =q^{1/2}$ and $q^{1/2}=7,11$ since either $\lambda \equiv 0 \pmod{4}$ and hence $\left\vert \Delta \right \vert =\lambda/2+1$ is odd, or $\lambda =2w^{2}$, where $w$ is odd, $w\geq 3$, $2(w^{2}-1)$ is a square and $\left\vert \Delta \right \vert=w^{2}+1$. The two numerical cases are ruled out since they
violate Lemma \ref{Jedan}(1), whereas the former yields $\lambda =2q^{1/2}$ with $q$ odd. Then $\left \vert \Sigma \right \vert$ is coprime to $q$ and hence $L_{\Delta}$  must contain a Sylow $p$-subgroup of $L$, which is not the case.

Finally, assume that $L_{\Delta }\cong A_{4}$, $S_{4}$ or $A_{5}$. In the first two cases $\lambda$ must be divisible by $4$ by Lemma \ref{Jedan}(3), hence $\lambda /2 +1=\left\vert \Delta \right\vert =3$ since $\left\vert \Delta \right\vert$ is odd,
and so $\lambda =4$, whereas $\lambda >10$. Thus $L_{\Delta } \cong A_{5}$. The previous argument can be applied to exclude the case $\lambda \equiv 0 \pmod{4}$. Therefore $\lambda=18$, $\left\vert \Delta \right\vert =10$ and $145=\left\vert \Delta \right\vert =\frac{q(q^2-1)}{120}$ which has no integer solutions. 
\end{proof}

\bigskip

\begin{lemma}
\label{PSL2}$L$ is not isomorphic to $PSL_{n}(q)$, $n \geq 2$.
\end{lemma}

\begin{proof}
Assume that $L\cong PSL_{n}(q)$. Then $n\geq 3$ by Lemma \ref{m=2}, and hence 
$M\in \mathcal{C}_{1}(L)\cup \mathcal{C}_{2}(L)$ by Lemma \ref{Filt}.

Assume that $L_{\Delta }$ lies in a maximal parabolic subgroup $M$ of $L$. If $L_{\Delta }$ is not of type $%
P_{h,n-h}$ then $L_{\Delta }=M$ by \cite[Table 3.5.H--I]{KL} for $n \geq 13$ and by \cite[Section 8.2]{BHRD} for $3\leq n \leq 12$. Also, $L_{\Delta }$ is as in \cite[Proposition 4.1.17.(II)]{KL} and $%
\left\vert \Sigma \right\vert ={n\brack h}_{q}$, where $h\leq n/2$.

Assume that $L_{\Delta }^{\Delta }$ is solvable. Then $ \left \vert L_{\Delta }^{\Delta }\right\vert \mid q-1$ by \cite[Proposition 4.17.(II)]{KL}, and hence 
$$\frac{1}{2} q^{h(n-h)} \leq {n\brack h}_{q}=\left\vert \Sigma \right\vert < 2q^{2}.$$ 
Then either $n =3$ and $h=1$, or $(n,h,q)=(4,1,2),(4,2,2)$ since $\left\vert \Sigma \right\vert$. The numerical cases are immediately ruled by  Lemma \ref{Jedan}(1). Therefore $n=3$, $\left\vert \Sigma \right\vert=q^{2}+q+1$ and hence $X^{2}=q^{2}+(q+1)^{2}$, where $X^{2}=2\left\vert \Sigma \right\vert-1$ again by Lemma \ref{Jedan}(1). Thus $(q,q+1,X)$ is primitive solution of the Pythagorean equation and hence it is of the form as in \cite[P3.1]{Rib}. Easy computations show that $q=3$. Then $\left\vert \Sigma \right\vert=13$ and hence $\lambda=6$, whereas $\lambda>10$.
 
Assume that $L_{\Delta}^{\Delta}$ is non-solvable. Then $L_{\Delta }^{\Delta }$ acts $2$-transitively on $\Delta$ by Proposition \ref{unapred}. Then $\mathrm{Soc}(L_{\Delta }^{\Delta })$ is isomorphic to $PSL_{x}(q)$, where $x\in
\left\{ h,n-h\right\} $ and $x\geq 2$, by \cite[List (B)]{Ka} and by \cite[Proposition 4.17.(II)]{KL}. Note that $(x,q)\neq (2,5),(2,9)$ since $\lambda>10$. Moreover, $(x,q)\neq (2,7)$. It is not so, then $\lambda=12$ and hence $\left\vert \Sigma \right \vert =121$. However, this is impossible by \cite[Table B.4]{DM} since $L \cong PSL_{n}(7)$. Thus $\left\vert \Delta
\right\vert =\frac{q^{x}-1}{q-1}$, $\lambda =2q\frac{q^{x-1}-1}{q-1}$ and
hence%
\begin{equation}
8q^{2x-2}\geq 2q^{2}\left( \frac{q^{x-1}-1}{q-1}\right) ^{2}-2q\frac{%
q^{x-1}-1}{q-1}+1=\left\vert \Sigma \right\vert ={n\brack x}_{q}\geq \frac{%
\allowbreak q^{\frac{1}{2}x\left( 2n-x+1\right) }}{2q^{\frac{1}{2}x\left(
x+1\right) }}=\frac{1}{2}q^{x\left( n-x\right) }  \label{ULB}
\end{equation}%
and so $2^{x\left( n-x\right) -2x+2}\leq q^{h\left( n-h\right) -2x+2}\leq 16$%
. Then $x(n-2-x) \leq 2$, and hence $x=h=2$ and $n=5$ since $x \geq 2$ and $h \leq n/2$. Also, $q=2,3$, and hence $\left\vert \Sigma \right\vert=155$ or $1210$ respectively, but both values of $\left\vert\Sigma\right\vert$ contradict Lemma \ref{Jedan}(1).

Assume that $L_{\Delta }$ is of type $P_{h,n-h}$, where $h<n/2$. If $L_{\Delta }^{\Delta }$ is solvable, then $ \left \vert L_{\Delta }^{\Delta }\right\vert \mid (q-1)^{2}$ by \cite[Proposition 4.1.22.(II)]{KL}. Also, $\left\vert \Sigma \right\vert \geq {n\brack h}_{q}$ since $L_{\Delta }$ lies in a maximal parabolic subgroup of type $P_{h}$. Hence,   
$$\frac{1}{2} q^{h(n-h)} \leq {n\brack h}_{q}\leq\left\vert \Sigma \right\vert < 2(q-1)^{4}.$$
Then $n \leq 5$ and $h=1$ since $h<n/2$. We actually obtain $n=3$ since $\left\vert \Sigma
\right\vert =\frac{(q^{n}-1)(q^{n-1}-1)}{(q-1)^{2}}$. Hence, $%
2q^{3}+(2q+1)^{2}=2\left\vert \Sigma \right\vert -1=X^{2}$ for some positive odd integer $X$ by Lemma \ref{Jedan}(1). Then $%
2q^{3}=(X-2q-1)(X+2q+1)$ and hence $q=2^{t}$, $t\geq 1$. Then $2^{s}+2^{t+1}+1=X=2^{3t+1-s}-2^{t+1}-1$ for some integer $s$ such
that $0\leq s\leq 3t$. Thus $2^{3t+1-s}=2^{s}+2^{2(t+1)}+2$. If $s>1$ then $%
s=3t$, which is clearly impossible. Then $s=1$ and hence $2^{3t}=2^{2(t+1)}+4$%
, which has no integer solutions for $t\geq 1$. Thus, $L_{\Delta }^{\Delta }$ is a non-solvable group acting $2$-transitively on $\Delta$ by Proposition \ref{unapred}. Then $L_{\Delta }^{\Delta
} $ is isomorphic to $PSL_{x}(q)$, where $x\in \left\{ h,m-h\right\} $ and $%
x\geq 2$, by \cite[Proposition 4.1.22.II]{KL} and \cite[List (B)]{Ka}. Then the same conclusion of (\ref{ULB}) holds
since $\left\vert \Sigma \right\vert \geq {m\brack x}_{q}$. Thus $n=5$, $x=2$ and $q=2,3$. Then $\left\vert \Sigma \right\vert=1085$ or $7865$ respectively, but both values of $\left\vert\Sigma\right\vert$ contradict Lemma \ref{Jedan}(1).

Assume that $L_{\Delta}$ lies in a maximal member of $\mathcal{C}_{2}(L)$. Then $q$ is odd by \cite{LS}, and $L_{\Delta}$
is of type $GL_{n/t}(q)\wr S_{t}$, where either $t=2$, or $t=3$ and either $%
q\in \left\{ 5,9\right\} $ and $n$ odd, or $(n,q)=(3,11)$ by \cite[Proposition 4.7]{AB}. Moreover, 
\begin{equation}\label{hocudaumrem}
L_{\Delta}\cong \left[ \frac{(q-1)^{t-1}(q-1,n/t)}{(q-1,n)}\right] .PSL_{n/t}(q)^{t}.%
\left[ (q-1,n/t)^{t-1}\right] .S_{t}
\end{equation}%
by \cite[Proposition 4.2.9.(II)]{KL}. In addition, $L_{\Delta}^{\Delta}$ is non-solvable  and acts $2$-transitively on $\Delta$ by Lemma \ref{Sevdah}. Then $t<n$ by (\ref{hocudaumrem}) and $\left\vert L\right\vert
<\left\vert L_{\Delta }\right\vert ^{2}$ by Theorem \ref{Large}. Hence,
\begin{equation}\label{izbaciti}
q^{n^{2}-2}<\frac{(q-1)^{2t-2}(q-1,n/t)^{2t}}{(q-1,n)^{2}} \cdot (q^{2n^{2}/t-2t})\cdot (t!)^{2}
\end{equation}
by (\ref{hocudaumrem}) and \cite[Corollary 4.1.(i)]{AB}. If $t=3$ then $q^{n^{2}-2}<36q^{2n^{2}/3+2}$ and hence $q^{n^{2}/3-4}<36$, which has no solutions for $q \geq 3$ and $n \geq 6$. Thus $t=2$, and hence $\mathrm{Soc}(L_{\Delta }^{\Delta })\cong PSL_{n/2}(q)$ since $L_{\Delta}^{\Delta}$ is non-solvable and acts $2$-transitively on $\Delta$. Then either $%
\left\vert \Delta \right\vert =5$ and $(n,q)=(4,9)$, or $\left\vert \Delta
\right\vert =q$ and $(n,q)=(4,5),(4,7),(4,11)$, or $\left\vert \Delta
\right\vert =\frac{q^{n/2}-1}{q-1}$. In each case one has $q^{5n}\leq
q^{n(3n-2)/2}\leq \left\vert L:L_{\Delta }\right\vert =\left\vert \Sigma
\right\vert \leq 2\left\vert \Delta \right\vert ^{2}<2q^{n}$, which is a
contradiction. This completes the proof.
\end{proof}

\begin{lemma}
\label{PSU2}$L$ is not isomorphic to $PSU_{n}(q)$.
\end{lemma}

\begin{proof}
Assume that $L\cong PSU_{n}(q)$. Then $n\geq 3$ by Lemma \ref{m=2} since $%
PSU_{2}(q)\cong PSL_{2}(q)$. Moreover, one of the following holds by Lemma %
\ref{Filt} and by \cite[Proposition 4.17]{AB}:

\begin{enumerate}
\item[(i)] $q$ is even and $L_{\Delta }$ lies in a maximal parabolic
subgroup of $L$.

\item[(ii)] $q$ is odd and $L_{\Delta }$ is a maximal $\mathcal{C}_{1}$%
-subgroup of $L$ of type $GU_{t}(q)\perp GU_{n-t}(q)$;

\item[(iii)] $q$ is odd, $L_{\Delta }$ is a maximal $\mathcal{C}_{2}$%
-subgroup of $L$ of type $GU_{n/t}(q)\wr S_{t}$ and one of the following
holds:

\begin{enumerate}
\item $t=2$.

\item $t=3$ and $(q,d)=(5,3),(13,1)$, where $d=(n,q+1)$.

\item $t=n=4$ and $q=5$.
\end{enumerate}
\end{enumerate}

Suppose that (i) holds. Then $L_{\Delta }$ is a maximal parabolic subgroup
of $L$ by \cite[Table 3.5.H--I]{KL} for $n\geq 13$ and \cite[Section 8.2]%
{BHRD} for $3\leq n\leq 12$. If $L_{\Delta }^{\Delta }$ does not act $2$-transitively on $\Delta$, then  $L_{\Delta }^{\Delta }$ is solvable by Proposition \ref{unapred}. Thus $\left\vert L_{\Delta }^{\Delta }\right\vert \mid
q^{2}-1$ by \cite[Proposition 4.18.(II)]{KL}, and hence  
\[
(q-1)q^{n^{2}-3}<\left\vert L\right\vert <4\left\vert Out(L)\right\vert
^{2}\left\vert L_{\Delta }^{\Delta }\right\vert
^{2}=16f^{2}(n,q+1)^{2}(q^{2}-1)^{2}
\]%
by Theorem \ref{Large}(2b). Thus $q^{n^{2}-6}<16n^{2}(q+1)$, and hence $n=3$
and $q=4,8$ since $q$ is even. So, $\left\vert \Sigma \right\vert =65$ and $%
513$ respectively. However, both these contradict Lemma \ref{Jedan}(1). Thus $%
L_{\Delta }^{\Delta }$ acts $2$-transitively on $\Delta $ by Proposition \ref%
{unapred}. 

Assume that $L_{\Delta }^{\Delta }$ is non-solvable. Then either $t\geq 2$, $PSL_{t}(q^{2})%
\trianglelefteq L_{\Delta }^{\Delta }$ and $\left\vert \Delta \right\vert =%
\frac{q^{2t}-1}{q-1}$, or $n=4$, $t=1$, $PSL_{2}(q)\trianglelefteq L_{\Delta
}^{\Delta }$ and $\left\vert \Delta \right\vert =q+1$, or $n-2t=3$, $%
PSU_{3}(q)\trianglelefteq L_{\Delta }^{\Delta }$ and $\left\vert \Delta
\right\vert =q^{3}+1$ by \cite[Proposition 4.1.18(II)]{KL} and by \cite[List
(B)]{Ka} since $q$ is even. On the other hand, by \cite[Proposition 4.1.18.(II)%
]{KL} and \cite[Lemma 4.1]{AB}, one obtains 
\begin{equation}\label{smrt}
\left\vert \Sigma \right\vert =\frac{\prod_{i=4}^{2t+3}\left(
q^{i}-(-1)^{i}\right) }{\prod_{j=1}^{t}\left( q^{2j}-1\right) }>\frac{q^{(2t+3)(t+2)-3}}{q^{3}+1} \cdot \frac{1}{(q^{2}-1)(q^{4}-1)q^{t(t+1)-6}} >q^{t^{2}+6t}.
\end{equation}
If $\left\vert \Delta \right\vert =\frac{q^{2t}-1}{q-1}$ then 
\begin{equation}
2q^{4t-2}>2\left( q\frac{q^{2t-1}-1}{q-1}\right) ^{2}-2\left( q\frac{%
q^{2t-1}-1}{q-1}\right) +1=\left\vert \Sigma \right\vert >q^{t^{2}+6t}
\label{polud}
\end{equation}%
and we reach a contradiction.

In the remaining cases, we have $\lambda =2q^{i}$ and $\left\vert \Sigma
\right\vert =2q^{2i}-2q^{i}+1$ with $i=1,2$. Both these lead to $%
q^{t^{2}+6t}<\left\vert \Sigma \right\vert <8q^{4}$ and hence to a
contradiction. 

Assume that $L_{\Delta }^{\Delta }$ is solvable. As above $\left\vert L_{\Delta }^{\Delta }\right\vert \mid
q^{2}-1$, hence $\left\vert \Delta \right \vert \mid q^{2}-1$ by Proposition \ref{unapred}. Then $q^{t^{2}+6t} \leq \left\vert \Sigma \right \vert \leq 2 \left\vert \Delta \right \vert^{2} \leq 2(q^{2}-1)^{2}$ by (\ref{smrt}), which is clearly impossible for $t\leq 2$. 

Note that, $L_{\Delta }$ is clearly non-parabolic in the remaining cases.
Thus $p\mid \left\vert \Sigma \right\vert $, and hence $q\geq 5$ by Lemma %
\ref{Jedan}(3) since $q$ is odd.

Suppose that (ii) holds. Then 
$L_{\Delta }^{\Delta }$ is non-solvable and acts $2$-transitively on $\Delta $ by Lemma \ref{Sevdah}(2). Then either $t=3$ or $n-t=3$ and in both cases $L_{\Delta }\cong
PSU_{3}(q)$ by \cite[Proposition 4.1.4.(II)]{KL}. Then $\lambda /2+1=q^{3}+1$
and so $\lambda =2q^{3}$. Then $\left\vert \Sigma \right\vert
=2q^{6}-2q^{3}+1$, whereas $p\mid \left\vert \Sigma \right\vert $.

Suppose that (iii) holds. Then 
\begin{equation}
L_{\Delta }\cong \left[ \frac{(q+1)^{t-1}(q+1,n/t)}{(q+1,n)}\right]
.PSU_{n/t}(q)^{t}.\left[ (q+1,n/t)^{t-1}\right] .S_{t}  \label{ciufciuf}
\end{equation}
by \cite[Proposition 4.2.9.(ii)]{KL}. Case (iii.c) implies $\left\vert \Sigma
\right\vert =5687500$ but this contradicts Lemma \ref{Jedan}(1). So, it does
not occur. Thus $t=2$ or $3$. Also $L_{\Delta }^{\Delta }$ is non-solvable and acts $2$-transitively on $\Delta$ by Lemma \ref{Sevdah}(2). Then $t<n$ by (\ref%
{ciufciuf}) since $t=2,3$, and 
\begin{equation}
(q-1)q^{n^{2}-3}<\left\vert L\right\vert <\left\vert L_{\Delta }\right\vert
^{2}=\frac{(q+1)^{2t-2}(q+1,n/t)^{2t}\left( t!\right) ^{2}}{(q+1,n)^{2}}%
q^{n^{2}/t-t}  \label{omg}
\end{equation}%
by Theorem \ref{Large} and \cite[Corollary (ii)]{AB}. If $t=3$ then $n\geq 6$%
, as $t<n$, and hence (\ref{omg})  implies $%
(q-1)q^{2n^{2}/3}<4(q+1)^{4}n^{4}/9$ which has no admissible solutions for $%
q\geq 5$. Then $t=2$, $n\geq 4$ and $q\geq 5$. Also (\ref{omg}) implies $%
(q-1)q^{n^{2}/2-1}<4(q+1)^{2}n^{2}$ and again no admissible solutions. This completes the proof.
\end{proof}

\begin{lemma}
\label{PSp2}If $L$ is not isomorphic to $PSp_{n}(q)^{\prime}$.
\end{lemma}

\begin{proof}
Assume that $L\cong PSp_{n}(q)^{\prime}$. Then $n\geq 4$ by Lemma \ref{m=2} and $(n,q) \neq (4,2)$ by Lemma \ref{Alt2} since $PSp_{4}(2)^{\prime} \cong A_{6}$. Thus $L\cong PSp_{n}(q)$. By Lemma \ref{Filt} and by \cite{LS} and \cite[Proposition 4.22]{AB} one of the following holds:
\begin{enumerate}
\item[(i)] $L_{\Delta }$ lies in a maximal parabolic subgroup of $L$ and $q$ is
even.

\item[(ii)] $L_{\Delta }$ is a maximal $\mathcal{C}_{1}$-subgroup of $L$ of type $Sp_{i}(q)\perp Sp_{n-i}(q)$ with $q$
odd

\item[(iii)] $L_{\Delta }$ is a maximal $\mathcal{C}_{2}$-subgroup of $L$ of type $Sp_{n/t}(q)\wr S_{t}$, where $t=2,3$,
or $(n,t)=(8,4)$, or $(n,t)=(10,5)$ and $q=3$.
\end{enumerate}

Suppose that (i) holds. Then $L_{\Delta }$ is a maximal parabolic subgroup
of $L$ by \cite[Table 3.5.H--I]{KL} for $n \geq 13$ and by \cite[Section 8.2]{BHRD} for $4 \leq n \leq 12$. Thus
$$\left\vert \Sigma \right\vert =\prod_{i=0}^{t-1}\frac{q^{n-2i}-1}{q^{i+1}-1}>\frac{1}{2}q^{(n-1)t-3t(t-1)/2}$$

If $L_{\Delta }^{\Delta }$ is solvable then $\left\vert L_{\Delta }^{\Delta } \right \vert \mid (q-1,t)$
by \cite[Proposition 4.1.19(II)]{KL} and hence $\left\vert \Delta \right \vert \mid (q-1,t)$ by Proposition \ref{unapred}. Therefore $ q^{(n-1)t-3t(t-1)/2/2}<\left\vert \Sigma \right\vert <2(q-1,t)^{2}$, which is clearly impossible.

If $L_{\Delta }^{\Delta }$ is non-solvable then $L_{\Delta }^{\Delta }$ acts $2$-transitively on $\Delta$ by Proposition \ref{unapred}. Hence, by \cite[List (B)]{Ka}, one of the following holds:
\begin{enumerate}
\item[(I)]$\mathrm{Soc}(L_{\Delta }^{\Delta })\cong PSL_{t}(q)$, either $t\geq 2$, or $t=1$ and $n=4$, and $%
\left\vert \Delta \right\vert =\frac{q^{t}-1}{q-1}$;
\item[(II)]$\mathrm{Soc}(L_{\Delta }^{\Delta })\cong PSp_{n-t}(2)$, $n-t\geq 6$ and $\left\vert \Delta
\right\vert =2^{2(n-t)-1}\pm 2^{n-t-1}$.
\end{enumerate} 

Then (I) is ruled out since it implies $q^{(n-1)t-3t(t-1)/2/2}<\left\vert \Sigma \right\vert <2q^{2t}$, which is impossible; (II) is ruled out since $\left\vert \Delta \right\vert \not \equiv 0 \pmod{4}$.

Suppose that (ii) holds. Then $L_{\Delta }\cong Sp_{i}(q)\circ Sp_{n-i}(q)$
by \cite[Proposition 4.1.3.(II)]{KL}. Then $L_{\Delta }^{\Delta }\cong
PSp_{j}(q)$ with $j\in \left\{ i,n-i\right\} $ by Lemma \ref{Sevdah}. Hence, one of the following
holds (recall that $q$ is odd):

\begin{enumerate}
\item $j=2$ and $\left\vert \Delta \right\vert =6$ for $q=9$, $\left\vert
\Delta \right\vert =q$ and $q=5,7,11$, or $\left\vert \Delta \right\vert =q+1$;

\item $j\geq 6$ and $\left\vert \Delta \right\vert =2^{j-1}\pm 2^{j/2-1}$.
\end{enumerate}

Actually, in (1) $q=5$ and $\left\vert \Delta \right\vert =5$ cannot occur since $\lambda>10$, and $q\neq 7,9,11$ by Lemma \ref{Jedan}(2). Also (2) is ruled out since $\left\vert \Delta \right\vert \not \equiv 0 \pmod{4}$. Thus $\left\vert \Delta \right\vert =q+1$, $\lambda=2q$ and hence $\left\vert \Sigma \right\vert $ is coprime to $q$. So $L_{\Delta }$ must contain a Sylow $q$-subgroup of $L$, which is a contradiction.  

Suppose that (iii) holds. Then $q \mid \left \vert \Sigma \right \vert$ and hence $q \neq 3$ by Lemma \ref{Jedan}(2). Thus either $t=2,3$, or $(n,t)=(8,4)$. Also, $L_{\Delta }^{\Delta }$ is non-solvable and acts $2$-transitively on $\Delta$ by Lemma \ref{Sevdah}(2)

If $t=2$, then either $n=4$ and $L_{\Delta }^{\Delta }\cong PSL_{2}(q)$ and
either $\left\vert \Delta \right\vert =5$ for $q=5$, or $\left\vert \Delta
\right\vert =6$ for $q=9$, $\left\vert \Delta \right\vert =q$ for $q=7,11$,
or $\left\vert \Delta \right\vert =q+1$, or $n\geq 12$, $L_{\Delta }^{\Delta
}\cong PSL_{n/2}(2)$ and $\left\vert \Delta \right\vert =2^{n-1}\pm 2^{n/2-1}
$. However, all these cases are excluded by the same argument previously used. 

If $t=3,4$, by Theorem \ref{Large} and by \cite[Proposition 4.2.10.(II)]{KL} and \cite[Corollary 4.3(iii)]{AB}, we have 
\begin{equation}
q^{\frac{n(n+1)}{2}}/4<\left\vert L\right\vert <\left\vert
L_{\Delta }\right\vert ^{2}<2^{2(t-1)}q^{n(n/t+1)}(t!)^{2}  \label{ineq}
\end{equation}
which implies $q^{\frac{n(n+1)}{2}-\frac{n(n+t)}{t}}<2^{2t}(t!)^{2}$, and
hence $t=3$, $n=6$ and $q=5$ or $13$ since $q\equiv 1\pmod{4}$ by Lemma %
\ref{Jedan}(2). Therefore $\left\vert \Sigma \right\vert =44078125$ or $%
3929239732405$, but both contradict Lemma \ref{Jedan}(1). 
\end{proof}

\begin{lemma}
\label{Cla2}$L$ is not isomorphic to a simple classical group.
\end{lemma}

\begin{proof}
In order to complete the proof we need to tackle the case $L\cong P\Omega
_{n}^{\varepsilon }(q)$, where $\varepsilon \in \left\{ \pm ,\circ \right\} $
since the other simple groups are analyzed in Lemmas \ref{PSL2}, \ref{PSU2} and \ref{PSp2}. Since $L$ is non-abelian
simple, $n>2$ and $(n,\varepsilon )\neq (4,+)$. Also $(n,\epsilon ,)\neq
(3,\circ )$ for $q$ odd, $(4,-),(6,+)$ by Lemma \ref{PSL2}, since in these
cases $L$ is isomorphic to $PSL_{2}(q)$, $PSL_{2}(q^{2})$ or $PSL_{4}(q)$
respectively. Finally $(n,\epsilon )\neq (4,-)$, and $(n,\epsilon )\neq
(5,\circ )$ for $q$ odd, otherwise $L$ would be isomorphic to $PSU_{4}(q)$ or $PSp_{4}(q)$ respectively, which are excluded in Lemmas \ref{PSU2} and \ref{PSp2} respectively. Thus, $n\geq 7$. By \cite{LS} and by \cite[Proposition 4.23]{AB} one of the following holds:
\begin{enumerate}
\item[(1)] Either $q$ is even and $L_{\Delta}$ lies in a  maximal parabolic subgroup, or $q$ is odd and $L_{\Delta}$ is
the stabilizer in $L$ of a non-degenerate subspace of $PG_{n-1}(q)$.

\item[(2)] $L_{\Delta}$ is a $\mathcal{C}_{2}$-subgroup of $L$ of type $%
O_{n/t}^{\varepsilon ^{\prime }}(q)\wr S_{t}$, where $q$ is odd, and either $%
t=2$, or $n=t=7$ and $q=5$, or $7\leq n=t\leq 13$ and $q=3$.

\end{enumerate}

Assume that $q$ is even and that $L_{\Delta}$ lies in a  maximal parabolic subgroup $M$ of type $%
P_{m}$. Thus $\varepsilon =\pm $ and hence $n \geq 8$. If $(\varepsilon ,m)\neq (+,n/2-1)$, then 
$L_{\Delta }=M$ by \cite[Table 3.5.H--I]{KL} for $n\geq 13$ and by \cite[Section 8.2]{BHRD} for $7\leq n\leq
12$. Nevertheless, in each case we have that $L_{\Delta }\cong \lbrack q^{a}]:GL_{m}(q)\times
\Omega _{n-2m}^{\varepsilon }(q)$, where $a=nm-\frac{m}{2}(3m-1)$, by \cite[%
Proposition 4.1.20.II]{KL}. Therefore,
\begin{equation}
\left\vert \Sigma \right\vert ={ \frac{n-1+\varepsilon }{2} \brack m}%
_{q}\prod_{i=0}^{m-1}\left( q^{\frac{n-1-\varepsilon }{2}-i}+1\right)
>q^{m\left( \frac{n-1+\varepsilon -2m}{2}+\frac{n-\varepsilon -m}{2}\right)
}=q^{\left( n-\frac{3}{2}m-\frac{1}{2}\right) m}  \label{nosfe}
\end{equation}%
by (\ref{ULB}) (see also \cite[Exercise 11.3]{Tay}).

Assume that $L_{\Delta}^{\Delta}$ is solvable. Then $\left\vert L_{\Delta}^{\Delta} \right\vert \mid q-1$, and hence $\left\vert \Delta \right\vert \mid q-1$. Then $q^{\left( n-\frac{3}{2}m-\frac{1}{2}\right) m} <\left\vert \Sigma \right\vert \leq 2 \left\vert \Delta \right\vert ^{2}=2(q-1)^{2}$, which is impossible for $n \geq 8$.

Assume that $L_{\Delta}^{\Delta}$ is non-solvable. Then $L_{\Delta}^{\Delta}$ acts $2$-transitively on $\Delta$ by Proposition \ref{unapred}, hence one of the following holds by \cite{Ka}:

\begin{enumerate}
\item[(I)] $\mathrm{Soc}(L_{\Delta }^{\Delta })$ is isomorphic to $PSL_{m}(q)$, $m \geq 2$, and either $\left\vert \Delta \right\vert =\frac{q^{m}-1}{q-1}$, or $\left\vert \Delta \right\vert =8$ for $(m,q)=(4,2)$.

\item[(II)] $\mathrm{Soc}(L_{\Delta }^{\Delta })\cong \Omega _{4}^{-}(q)\cong
PSL_{2}(q^{2})$, $\varepsilon =-$, $n=2m+4$ and $\left\vert \Delta
\right\vert =q^{2}+1$.

\item[(III)] $\mathrm{Soc}(L_{\Delta }^{\Delta })\cong PSL_{2}(q)$, $\varepsilon =+$, $%
n=2m+4$ and $\left\vert \Delta \right\vert =q+1$.

\item[(IV)]$\mathrm{Soc}(L_{\Delta }^{\Delta })\cong PSL_{4}(q)$, $\varepsilon =+$, $%
n=2m+6$ and either $\left\vert \Delta \right\vert =\frac{q^{4}-1}{q-1}$, or $\left\vert \Delta \right\vert =8$ for $%
q=2$.
\end{enumerate}

Assume that (I) or (IV) holds. Then $\left\vert \Delta \right\vert \neq 8$ since $\left\vert \Delta \right\vert$ is not divisible by $4$. Then $\left\vert \Delta
\right\vert =\frac{q^{e}-1}{q-1}$, where either $e=m$, or $e=4$ and $n=2m+6$ and $\varepsilon=+$. Furthermore, $%
m\geq 2$ in both cases since $n\geq 8$. Now, $\left\vert \Delta \right\vert= \lambda/2+1$ and $\left\vert \Delta \right\vert= (\lambda^{2}-2\lambda+2)/2$ imply
\begin{equation*}
\left\vert \Sigma \right\vert =2\left( q\frac{q^{e-1}-1}{q-1}\right)
^{2}-2\left( q\frac{q^{e-1}-1}{q-1}\right) +1\text{.}
\end{equation*}%
and so $q^{\left( n-\frac{3}{2}m-\frac{1}{2}\right) m}<\left\vert \Sigma
\right\vert <2q^{2e-2}$.

If $e=m$ and $q^{\left( n-\frac{3}{2}m-\frac{1}{2}\right) m-2m+2}<2$ then $n \leq %
10-\frac{8}{m}$, and hence $(n,m,\varepsilon)=(8,4,+)$ since $n$ is even, $n \geq 8$ and $m \leq n/2$. Then \begin{equation*}
2q^{6}+4q^{5}+6q^{4}+2q^{3}-2q+1=\left\vert \Sigma \right\vert =\allowbreak
\left( q^{2}+1\right) \left( q^{3}+1\right) \left( q+1\right)
\end{equation*}%

If $e=4$, $n=2m+6$ and $\varepsilon=+$, then $m=1$ and so%
\begin{equation*}
2q^{6}+4q^{5}+6q^{4}+2q^{3}-2q+1=\left\vert \Sigma \right\vert
=(q^{4}-1 )(q^{3}+1 )\text{,}
\end{equation*}%
which has not integer solutions.

Assume that case (II) or (III) holds. Then $n=2m+4$ and $\left\vert \Delta
\right\vert =q^{j}+1$ with $j=2,1$ respectively. Then $\left\vert \Sigma
\right\vert <8q^{2j}$. On the other hand $\left\vert \Sigma \right\vert
>q^{\left( n-\frac{3}{2}m-\frac{1}{2}\right) m}$ by (\ref{nosfe}) since $%
n\geq m/2$. Therefore $q=2$ and $n=8$ and either $m=1$ or $4$, but each of
these contradicts $n=2m+4$. This excludes case (1) for $q$
even.

In the remaining cases, namely (1) and (2) for $q$ odd, it results that $p \mid \left \vert \Sigma \right \vert $ by \cite[Theorem 1.6]{Se}. Also $p$ is odd, and $p \neq 3$ by Lemma \ref{Jedan}(2). Therefore, in the sequel we may assume that $q \geq 5$. Then, by \cite[Table 3.5.H--I]{KL} for $n\geq 13$ and by \cite[Section 8.2]{BHRD} for $7\leq n\leq 12$, either $L_{\Delta }$ is maximal in $L$, or $L \cong P\Omega_{n}^{+}(5)$, $L_{\Delta }$ is a $\mathcal{C}_{2}$-subgroup of $L$ of type $O_{2}^{+}(5) \wr S_{n/2}$ and $G_{\Delta}$ is a novelty. In the latter case $n$ is forced to be $4$ by (2), whereas $n \geq 8$. Therefore, $L_{\Delta }$ is maximal in $L$. 
  
Assume that $L_{\Delta}$ is the stabilizer in $L$ of a non-degenerate subspace of $%
PG_{n-1}(q)$, $q$ odd. Since $L_{\Delta}^{\Delta}$ is non-solvable and acts $2$-transitively on $\Delta$ by Lemma \ref{Sevdah}, one of the following holds by \cite[%
Propositions 4.1.6.(II)]{KL} and \cite{Ka} and since $q$ is odd, $q\geq 5$
and $n\geq 7$:

\begin{enumerate}
\item[(i)] $L_{\Delta }$ preserves a non-degenerate $4$-subspace
of type $+$ and $\mathrm{Soc}(L_{\Delta }^{\Delta })\cong PSL_{2}(q)$;

\item[(ii)] $L_{\Delta }$ preserves a non-degenerate $4$-subspace
of type $-$ and $\mathrm{Soc}(L_{\Delta }^{\Delta })\cong PSL_{2}(q^{2})$;

\item[(iii)] $L_{\Delta }$ preserves a non-degenerate $6$%
-subspace of type $-$ and $\mathrm{Soc}(L_{\Delta }^{\Delta })\cong PSL_{4}(q)$.
\end{enumerate}

Assume that (i) or (ii) holds. Then either $\left\vert \Delta \right\vert =q^{j}+1$,
where $j=1$ or $2$ respectively, or $j=1$ and $\left\vert \Delta \right\vert
=q$ for $q=5,7,11$ or $\left\vert \Delta \right\vert =6$ for $q=9$, or $j=2$
and $\left\vert \Delta \right\vert =6$ for $q=3$. The first case implies $%
\lambda =2q^{j}$, therefore $\left\vert \Sigma \right\vert $ is coprime to $q$, whereas $L_{\Delta }$ must contain a Sylow $q$-subgroup of $L$, which is not the case. Also $q \neq 3,7,9,11$ by Lemma \ref{Jedan}(2). Finally, $\left\vert \Delta \right\vert
=q=5$ cannot occur since $\lambda>10$.

Assume that (iii) holds. Then $\left\vert \Delta \right\vert =\frac{%
q^{4}-1}{q-1}$ since $q$ is odd. Then $\lambda =2q\frac{q^{3}-1}{q-1}$, hence $\left\vert \Sigma \right\vert $ is coprime to $q$, whereas $L_{\Delta }$ must contain a Sylow $q$-subgroup of $L$, which is not the case. This excludes (1).

A similar argument to that used to rule out (iii) in the $PSp_{n}(q)$%
-case excludes $t=2$ in (2) as well (see Lemma \ref{PSp2}). Thus, $n=t=7$ and $q=5$ since we have seen that $q \geq 5$, and hence $L_{\Delta } \cong 2^{6}.A_{7}$ by \cite[Proposition 4.2.15(II)]{KL}. So $\left\vert \Sigma \right\vert=29752734375$, which contradicts Lemma \ref{Jedan}(1) and hence it is ruled out. This completes the proof.
\end{proof}

\bigskip

\begin{proof}[Proof of Theorem \protect\ref{T2}]
Since $L^{\Sigma }$ is almost simple by Theorem \ref{ASPQP}, the assertion
follows from Lemmas \ref{spor2}, \ref{Alt2}, \ref{Exc2} and \ref{Cla2}.
\end{proof}

\bigskip

\begin{proof}[Proof of Theorem \protect\ref{main}]
Since $\lambda \leq 10$ by Theorems \ref{T1} and \ref{T2}, the assertion
follows from Theorem \ref{MMS}.
\end{proof}


\bigskip

\begin{thebibliography}{99}
\bibitem{AZ} M. Aigner, G. M. Ziegler, Proofs from THE BOOK, (5th ed.).
Berlin, New York: Springer-Verlag. 2014

\bibitem{As} M. Aschbacher, Overgroups of Sylow subgroups in sporadic
groups, Mem. amer. Math. Soc. \textbf{343} (1986).

\bibitem{AB} S. H. Alavi, T. C. Burness, Large subgroups of simple groups,
J. Algebra \textbf{421} (2015) 187--233.

\bibitem{BP} J. Bamberg, T. Penttila, Overgroups of Cyclic Sylow Subgroups
of Linear Groups, Comm. Algebra, \textbf{36} (2008) 2503--2543.

\bibitem{BHRD} J. Bray, D. F. Holt, C. M. Roney-Dougal, The maximal
subgroups of the low-dimensional finite classical groups. With a foreword by
Martin Liebeck. London Mathematical Society Lecture Note Series, 407.
Cambridge University Press, Cambridge, 2013.

\bibitem{At} J. H. Conway, R. T. Curtis, R. A. Parker, R. A. Wilson, An
Atlas of Finite Groups, Clarendon Press, Oxford, 1985.

\bibitem{Da} H. Davies, Flag-transitivity and primitivity, Discrete Math. 
\textbf{63} (1987) 91--93.

\bibitem{Demb} P. Dembowski, Finite Geometries, Springer, Berlin,
Heidelberg, New York, 1968.

\bibitem{DGLPP} A. Devillers, M. Giudici, C. H. Li, G. Pearce, C. E. Praeger,
On imprimitive rank $3$ permutation groups, J. London Math. Soc. \textbf{84}
(2011) 649--669.

\bibitem{DM} J. D. Dixon, B. Mortimer, Permutation Groups. Springer Verlag,
New York (1966).

\bibitem{Gu} R. M Guralnick, Subgroups of prime power index in a simple group, J. Algebra \textbf{81} (1983) 304--311.

\bibitem{He} C. Hering, Transitive linear groups and linear groups which
contain irreducible subgroups of prime order, Geom. Dedicata \textbf{2}
(1974) 425--460.

\bibitem{HM} D. G. Higman and J. E. McLaughlin, Geometric ABA-groups,
Illinois J. Math. \textbf{5} (1961) 382--397.

\bibitem{Hu} Q. M. Hussain, On the totality of the solutions for the symmetrical incomplete block designs
$\lambda = 2$, $k= 5$ or $6$, Sankhya \textbf{7} (1945) 204–208.

\bibitem{LiSe} C. H. Li, \`{A}. Seress, The primitive permutation groups of squarefree degree, Bull. London Math. Soc. \textbf{35} (2003) 635--644. 

\bibitem{Ka} W. M. Kantor, Homogenous designs and geometric lattices\textit{%
, }J. Combin. Theory. Ser. A\textit{\ }\textbf{38} (1985) 66-74.

\bibitem{KL} P. Kleidman, M. W. Liebeck, The subgroup structure of the finite
classical groups. London Mathematical Society Lecture Note Series, 129.
Cambridge University Press, Cambridge, 1990.

\bibitem{La} E. S. Lander, Symmetric designs: An algebraic approach, London
Mathematical Society Lecture Note Series, vol. 74, Cambridge University
Press, Cambridge, 1983.

\bibitem{LPR} M. Law, C. E. Praeger, S. Reichard, Flag-transitive symmetric $%
2$-$(96,20,4)$-designs, J. Combin. Theory. Ser. A \textbf{116} (2019)
1009-1022.

\bibitem{LS} M. W. Liebeck, J. Saxl, The primitive permutation groups of odd
degree, J. London Math. Soc. \textbf{31} (1985) 250--264.

\bibitem{LSS} M. W. Liebeck, J. Saxl, G. M. Seitz, Subgroups of maximal rank
in finite exceptional groups of lie type, Proc. London Math. Soc. \textbf{65}
(1992) 297-325.

\bibitem{Lu} H. L\"{u}neburg, Translation Planes, Springer, Berlin, 1980.

\bibitem{MS} J. Mandi\'{c}, A. \v{S}ubasi\'{c}, Flag-transitive and
point-imprimitive symmetric designs with $\lambda \leq 10$, J. Combin. Theory
Ser. A, \textbf{189} (2022) 105620

\bibitem{MF} A. Montinaro, E. Francot, On Flag-transitive $2$-$k^{2},k,\lambda$ designs with $\lambda \mid k$, J. Combin. Des.
\textbf{30} (2022) 653–670.

\bibitem{Mo1} A. Montinaro, Classification of the non-trivial $2$-$(k^{2},k,\lambda )$ designs, with $\lambda \mid k$, admitting a flag-transitive almost simple automorphism group, J. Combin. Theory Ser. A \textbf{195} (2023) 105710

\bibitem{Mo2} A. Montinaro, Flag-transitive, point-imprimitive symmetric $2$-$%
(v,k,\lambda )$ designs with $k>\lambda (\lambda -3)/2$. arXiv: 2203.09261

\bibitem{Na} H. K. Nandi, Enumerations of nonisomorphic solutions of balanced incomplete block designs, Sankhya \textbf{7} (1946) 305--312.

\bibitem{ORR} E. O'Relly Reguerio, On primitivity and reduction for
flag-transitive symmetric designs, J. Combin. Theory. Ser. A \textbf{109}
(2005) 135-148.

\bibitem{P1} C .E. Praeger, An O'Nan-Scott theorem for finite quasiprimitive
permutation groups and an application to $2$-arc transitive graph , J.
London Math. Soc. \textbf{47} (1993) 227--239.

\bibitem{P2} C. E. Praeger, The flag-transitive symmetric designs with 45
points, blocks of size 12, and 3 blocks on every point pair, Des. Codes
Cryptogr. \textbf{44} (2007) 115--132.

\bibitem{PZ} C. E. Praeger, S. Zhou, Imprimitive flag-transitive symmetric
designs. J. Combin. Theory Ser. A \textbf{113} (2006) 1381--1395.

\bibitem{Rib} P. Ribenboim, Catalan's Conjecture: are $8$ and $9$ the only
consecutive powers? Academic Press, Inc. Boston, 1994.

\bibitem{Se} G. M. Seitz, Flag-transitive subgroups of Chevalley group, Ann.
of Math. \textbf{97} (1973) 27--56.

\bibitem{Tay} D. E. Taylor, The Geometry of the Classical Groups, Heldermann Verlag,
1992.

\bibitem{Va1} A. V. Vaslyev. Minimal permutation representations of finite
simple exceptional groups of types $G_{2}$ and $F_{4}$, Algebra i Logika 
\textbf{35} (1996) 663--684.

\bibitem{Va2} A. V. Vasilyev. Minimal permutation representations of finite
simple exceptional groups of types $E_{6}$, $E_{7}$ and $E_{8}$, Algebra i
Logika \textbf{36} (1997) 518--530.

\bibitem{Va3} A. V. Vasilyev. Minimal permutation representations of finite
simple exceptional groups of twisted type, \ Algebra i Logika \textbf{37}
(1998) 17--35.

\bibitem{Wi1} R. A. Wilson, Maximal subgroups of automorphism groups of
simple groups, J. London Math. Soc. \textbf{32} (1985) 460--466.

\bibitem{Wi2} R. A. Wilson, Maximal subgroups of sporadic groups,
arXiv:1701.02095v2

\bibitem{AtMod} R. Wilson, P. Walsh, J. Tripp, I. Suleiman, R. Parker, S.
Norton, S. Nickerson, S. Linton, J. Bray, R. Abbott, An Atlas of Finite
Groups Representation- Version 3\textit{,} available at\textit{\ }%
http://brauer.maths.qmul.ac.uk/Atlas/v3/

\bibitem{ZZ} Z. Zhang, J. Chen, S. Zhou, On Symmetric Designs with Flag-transitive and Point-quasiprimitive
Automorphism Groups, J. Combin Des. (in press).

\bibitem{GAP} The GAP Group, GAP Groups, Algorithms, and
Programming, Version 4.11.1, 2021, available at http://www.gap.system.org
\end{thebibliography}
\end{document}